\documentclass[11pt]{amsart}

\input epsf

\usepackage{amsmath}
\usepackage{amssymb}
\usepackage{amsfonts}
\usepackage[mathscr]{eucal}
\usepackage{graphicx}
\usepackage{epstopdf}
\usepackage{verbatim}
\usepackage{amscd}
\usepackage{upgreek}
\usepackage{dsfont}
\usepackage{caption}
\usepackage{amsthm}
\usepackage{color}
\usepackage{hyperref}
\usepackage[all]{xy}
\usepackage{pdfsync}
\usepackage{array}

\makeatletter
\def\@tocline#1#2#3#4#5#6#7{\relax
  \ifnum #1>\c@tocdepth % then omit
  \else
    \par \addpenalty\@secpenalty\addvspace{#2}%
    \begingroup \hyphenpenalty\@M
    \@ifempty{#4}{%
      \@tempdima\csname r@tocindent\number#1\endcsname\relax
    }{%
      \@tempdima#4\relax
    }%
    \parindent\z@ \leftskip#3\relax \advance\leftskip\@tempdima\relax
    \rightskip\@pnumwidth plus4em \parfillskip-\@pnumwidth
    #5\leavevmode\hskip-\@tempdima
      \ifcase #1
       \or\or \hskip 1em \or \hskip 2em \else \hskip 3em \fi%
      #6\nobreak\relax
    \dotfill\hbox to\@pnumwidth{\@tocpagenum{#7}}\par
    \nobreak
    \endgroup
  \fi}
\makeatother
\usepackage{hyperref}

\usepackage{cleveref}
\crefdefaultlabelformat{#2{\scshape #1}#3}
\crefname{subfigure}{{\scshape Fig.}}{{\scshape Figs.}}
\Crefname{subfigure}{{\scshape Figure}}{{\scshape Figures}}
\crefname{figure}{{\scshape fig.}}{{\scshape figs.}}
\Crefname{figure}{{\scshape Figure}}{{\scshape Figures}}

\usepackage{pgf,tikz}
\usepackage{mathrsfs}
\usetikzlibrary{arrows}

\usepackage{ccfonts,eulervm}
\usepackage[T1]{fontenc}

\usepackage{subcaption}

\usetikzlibrary{shadows}
\usetikzlibrary{shapes}
\usetikzlibrary{decorations}
\usetikzlibrary{arrows,decorations.markings}   
\usetikzlibrary{patterns}

\makeindex
\numberwithin{equation}{section}

\newtheorem{Theorem}{Theorem}[section]

\newtheorem{Proposition}[Theorem]{Proposition}
\newtheorem{Lemma}[Theorem]{Lemma}
\newtheorem*{KUC}{Koebe Uniformization Conjecture}
\newtheorem*{KCPT}{Koebe Circle Packing Theorem}
\newtheorem*{KAT}{Koebe-Andre'ev-Thurston Theorem I (for the sphere)}
\newtheorem*{KAT2}{Koebe-Andre'ev-Thurston Theorem II (for compact surfaces)}
\newtheorem*{BranchedKAT}{Polynomially Branched KAT Theorem}
\newtheorem*{DSPBVT}{Discrete Schwarz-Picard Boundary Value Theorem}
\newtheorem*{DUT}{Discrete Uniformization Theorem}
\newtheorem*{MDPT}{Maximal Disk Packing Theorem}
\newtheorem*{DBVT}{Discrete Boundary Value Theorem}
\newtheorem*{DSPL}{Discrete Schwarz-Pick Lemma}
\newtheorem*{DTTG}{Discrete Type Theorem for Graphs}
\newtheorem*{DTTPTG}{Discrete Type Theorem for Plane Triangulation Graphs}
\newtheorem*{HSKU}{He-Schramm Uniformization Theorem}
\newtheorem*{HSDU}{He-Schramm Discrete Uniformization Theorem}
\newtheorem*{DRMT}{Discrete Riemann Mapping Theorem}
\newtheorem*{HPL}{Hexagonal Packing Lemma}
\newtheorem*{DUTES}{Discrete Uniformization Theorem for Equilateral Surfaces}
\newtheorem{Conjecture}[Theorem]{Conjecture}
\newtheorem*{CTC}{Circumscribable Type Characterization}
\newtheorem*{ITC}{Incsribable Type Characterization}
\newtheorem*{MCP}{Midscribability of Convex Polyhedra}
\newtheorem*{CBM}{Convex Body Midscription}
\newtheorem*{CAL}{Cauchy Arm Lemma}
\newtheorem*{DFVL}{Discrete Four Vertex Lemma}
\newtheorem*{CCL}{Cauchy Combinatorial Lemma}
\newtheorem*{CRT}{Cauchy Rigidity Theorem}
\newtheorem*{CCHPC}{Compact Convex Hyperbolic Polyhedra Characterization}
\newtheorem*{FAR}{Face Angle Rigidity}
\newtheorem*{DTCHC}{Discrete Total Curvature for Polygonal Hyperbolic Loops}
\newtheorem*{CIHP}{Characterization of Convex Ideal Polyhedra}
\newtheorem*{CCHP}{Characterization of Convex Hyperideal Polyhedra}
\newtheorem*{WIC}{Weak Inscription Characterization}

\theoremstyle{definition}
\newtheorem*{Definition}{Definition}

\usepackage{geometry}  % See geometry.pdf to learn the many layout options.
\usepackage[parfill]{parskip}  % Activate to begin paragraphs with empty 
\usepackage[percent]{overpic}   % to put overlay latex on images -- 2 stage 
\title[Combinatorics Encoding Geometry]{Combinatorics Encoding Geometry:\\
The Legacy of Bill Thurston in the Story of One Theorem}

\author{Philip L. Bowers}

\date{\today} % Activate to display a given date or no date

\begin{document}

\maketitle

\tableofcontents

\begin{abstract}
This article presents a whirlwind tour of some results surrounding the \textit{Koebe-Andre'ev-Thurston Theorem}, Bill Thurston's seminal circle packing theorem that appears in Chapter 13 of \textit{The Geometry and Topology of Three-Manifolds}. 
\end{abstract}

%\tableofcontents

\section*{Introduction} 
Bill Thurston was the most original and influential topologist of the last half-century. His impact on the discipline of geometric topology during that time is unsurpassed, and his insights in the topology and geometry of three-manifolds led finally to the resolution of the most celebrated problem of topology over the last century---The Poincar\'e Conjecture. He made fundamental contributions to many sub-disciplines within geometric topology, from the theory of foliations on manifolds to the combinatorial structure of rational maps on the two-sphere, and from geometric and automatic group theory to classical polyhedral geometry. Of course his foundational work on three-manifolds, first laid out in his courses at Princeton in the late nineteen-seventies, compiled initially as a Princeton paper-back monograph inscribed by Bill Floyd and available upon request as \textit{The Geometry and Topology of Three-Manifolds} (GTTM)~\cite{Thurston:1980}, and maturing as the famous \textit{Thurston Geometrization Conjecture} of the early nineteen-eighties, was the driving force behind the development of geometric topology for the next thirty years. The final confirmation of the Geometrization Conjecture by Giorgi Perelman using the flow of Ricci curvature, following a program that had been introduced by Richard Hamilton, is one of the crown jewels of twentieth century mathematics.

Thurston marks a watershed in the short history of topology,\footnote{I will use the term \textit{topology} henceforth to mean \textit{geometric topology}. By dropping the adjective \textit{geometric} I certainly mean no slight of general, set-theoretic, or algebraic topology.} a signpost, demarcating topology before Thurston, and topology after Thurston. This is evidenced not only in the fabulous results he proved, explained, and inspired, but even more so in how he taught us to do mathematics. Topology before Thurston was dominated by the general and the abstract, entrapped in the rarified heights that captured the mathematical world in general, and topology in particular, in the period from the 1930's until the 1970's. Topology after Thurston is dominated by the particular and the geometric, a throwback to the nineteenth-century, having much in common with the highly geometric landscape that inspired Felix Klein and Max Dehn, who walked around and within Riemann surfaces, knew them intimately, and understood them in their particularity. Thurston's vision gave a generation of topologists permission to get their collective hands dirty by examining in great depth specific structures on specific examples.

One of the organizing principles that lies behind Thurston's vision is that geometry informs topology, and that the non-Euclidean geometry of Lobachevski, Bolyai, and Beltrami in particular is systemic to the study of topology. Hyperbolic geometry\index{hyperbolic geometry} permeates topology after Thurston, and it is hyperbolic geometry that becomes the common thread of the present article. This will be seen in the interrelated studies presented here. All to varying degrees are due to the direct influence of Bill Thurston and his generalization of the earlier results of Koebe and Andre'ev. All involve hyperbolic geometry in some form or influence, and even further all illustrate how combinatorics encodes geometry, another of the principles that underlies Thurston's vision. To my mind, the proposition that \textit{combinatorics encodes geometry, which in turn informs topology} has become a fundamental guiding motif for topology after Thurston. I offer this article as a celebration of Bill Thurston's vision and his immense influence over our discipline.

\subsection*{An introductory overview.} The Koebe-Andre'ev-Thurston Theorem represents a rediscovery and broad generalization of a curiosity of Paul Koebe's from 1936, and has an interpretation that recovers a characterization of certain three-dimensional hyperbolic polyhedra due to E.M.\,Andre'ev in two papers from 1970. This theorem is the foundation stone of the discipline that has been dubbed as \textit{discrete conformal geometry},\index{discrete conformal geometry} which itself has been developed extensively by many mathematicians in many different directions over the last thirty years. Discrete conformal geometry in its purest form is geometry born of combinatorics, but it has theoretical and practical applications. In the theoretical realm, it produces a discrete analytic function theory that is faithful to its continuous cousin, a quantum theory of complex analysis from which the classical theory emerges in the limit of large scales. In the realm of applications, it has been developed in a variety of directions, for practical applications in areas as diverse as biomedical imaging and 3D print head guidance. This rather large body of work flows from simple insights that Thurston presented in his lecture at Purdue University in 1985 on how to use the most elementary case of his circle packing theorem to provide a practical algorithm for approximating the Riemann mapping from a proper, simply-connected planar domain to the unit disk. A personal accounting of this development can be found in the author's own review~\cite{pB09} of the bible of circle packing theory, Ken Stephenson's \textit{Introduction to Circle Packing:\,The Theory of Discrete Analytic Functions}~\cite{kS05}. 

A perusal of the Table of Contents at the beginning of the article will give the reader a clue as to where I am going in this survey. I primarily stick with the theoretical results for which there are fairly direct lines from the Koebe-Andre'ev-Thurston Theorem to those results. This means in particular that I almost totally ignore the really vast array of practical applications that circle packing has found, especially in the last two decades as discrete differential geometry has become of primary importance in so many applications among computer scientists and computational geometers. A survey of applications will have to wait as space constraints preclude a discussion that does justice to the topic.

\subsection*{Dedication and Appreciation.} This article is dedicated to the memory of Bill Thurston and his student Oded Schramm, and to an appreciation of Jim Cannon and Ken Stephenson. I have spoken already of Bill Thurston's legacy. Oded Schramm was one of the first to press Thurston's ideas on circle packings to a high level of development and application, and his great originality in approaching these problems has bequeathed to us a treasure trove of beautiful gems of mathematics. Most of Oded's work on circle packing and discrete geometry was accomplished in the decade of the nineteen-nineties. As Bill is a demarcation point in the history of topology, Oded is one in the history of probability theory. In the late nineties, Oded became interested in some classical open problems in probability theory generated by physicists, in percolation theory and in random planar triangulations in particular. Physicists had much theoretical and computational evidence for the veracity of their conjectures, but little mathematical proof, or even mathematical tools to approach their verifications. In Oded's hands these venerable conjectures and problems began to yield to mathematical proof, using ingenious tools developed or refined by Oded and his collaborators, chief among which are $\mathrm{SLE}_{\kappa}$, originally Stochastic-Loewner Evolution, now renamed as Schramm-Loewner Evolution, and UIPT's, or Uniform Infinite Planar Triangulations. For a wonderful biographical commentary on Oded's contributions to mathematics, see Steffen Rohde's article \textit{Oded Schramm:\,From Circle Packing to SLE} in \cite{SR10}.

The two individuals who have had the greatest impact on my mathematical work are Jim Cannon and Ken Stephenson, the one a mathematical hero of mine, the other my stalwart collaborator for three decades. Jim's work has influenced mine significantly, and I greatly admire his mathematical tastes and contributions. Pre-Thurston, Jim had made a name for himself in geometric topology in the flavor of Bing and Milnor, having solved the famous double suspension problem and having made seminal contributions to cell-like decomposition theory and the characterization of manifolds. In the beginning of the Thurston era, his influential paper \textit{The combinatorial structure of cocompact discrete hyperbolic groups}~\cite{jC84} anticipated many of the later developments of geometric group theory, presaging Gromov's thin triangle condition and, \`a la Dehn, the importance of negative curvature in solving the classical word and conjugacy problems of combinatorial group theory. He with Thurston invented automatic group theory and then Jim settled upon the conjecture that bears his name as the work that for three decades has consumed his attention. Ken has been a joy with whom to collaborate over the past three decades. He was inspired upon attending Thurston's Purdue lecture in 1985 to change his mathematical attentions from a successful career as a complex function theorist, to a geometer exploring this new idea of circle packing using both traditional mathematical proof and the power of computations for mathematical experimentation. I began my foray into Thurston-style geometry and topology by answering in~\cite{BS92} a question of Ken and Alan Beardon from one of the first research papers~\cite{BSt90} to appear on circle packings after Rodin and Sullivan's 1987 paper~\cite{RS87} confirming the conjecture of Thurston from the Purdue lecture. Ken and I are co-authors on a number of research articles and his down-to-earth approach to the understanding of mathematics has been a constant check on my tendency toward flights of fancy. I have learned from him how to tell a good story of a mathematical topic. For Ken's warm friendship and collaboration I am grateful.

\subsection*{Acknowledgements.} I thank Prof.~Athanase Papadopoulos for inviting me to write on a favorite theme of mine to honor Bill Thurston and his legacy. It has been a pleasure for me to review the impact derived from this one beautiful little gem of Thurston. I thank Ken Stephenson for permission to use the graphics of Figures 1, 2, 5, and 7, and John Bowers for generating the graphics for Figures 3, 4, 6, 8, 9, and 10. 

\section{The Koebe-Andre'ev-Thurston Theorem, Part I}

\subsection{Koebe uniformization and circle packing.} In the early years of the twentieth century, rigorous proofs of the Riemann Mapping Theorem and the more general Uniformization Theorem were given by such eminent mathematicians as Osgood, Carath\'eodory, Poincar\'e, and Koebe, and refinements and re-workings would continue to be made by others, even up to the present.\footnote{The author recommends rather highly the article \textit{On the history of the Riemann mapping theorem} \cite{jG94} by Jeremy Gray and the monograph \textit{Uniformization of Riemann Surfaces:\,Revisiting a Hundred-Year-Old Theorem}~\cite{S-G16}. These two works give insightful historical accountings of the discovery, articulation, understanding, and finally rigorous proofs of the Riemann Mapping Theorem and the Uniformization Theorem. The narratives are at once engaging and perceptive, illustrating wonderfully the fact that mathematics is generally a common endeavor of a community of folks rather than the singular achievement of an enlightened few.} The generalization of the Riemann Mapping Theorem to multiply-connected domains fell to the hands of Paul Koebe, who in 1920 in~\cite{pK20} proved that every finitely-connected domain in the Riemann sphere is conformally equivalent to a \textit{circle domain}, a connected open set all of whose complementary components are points or closed round disks. Of course for a $1$-connected, or simply-connected, domain this is nothing more than the Riemann Mapping Theorem. He proved also a rigidity result, that any conformal homeomorphism between any two circle domains with finitely many complementary components is in fact the restriction of a M\"obius transformation.\footnote{Beware! This is not true in general. Two domains with uncountably many complementary components may be conformally equivalent yet fail to be M\"obius equivalent.} Koebe's real goal was what is known by its German name as \textit{Koebe's Kreisnormierungsproblem} and by its English equivalent as \textit{Koebe's Uniformization Conjecture},\index{Koebe uniformization conjecture} which he posed in 1908.

\begin{KUC}[\cite{pK08}]
    Every domain in the Riemann sphere is conformally homeomorphic to a circle domain.
\end{KUC}

\noindent This of course includes those domains with infinitely many, whether countably or uncountably many, complementary components. The general Koebe Uniformization Conjecture remains open to this day. More on this later.

In a paper of 1936, Koebe obtained the following circle packing theorem as a limiting case of his uniformization theorem of 1920. This went unnoticed by the circle packing community until sometime in the early nineteen-nineties. 

\begin{KCPT}[\cite{pK36}]
Every oriented simplicial triangulation $K$ of the 2-sphere $\mathbb{S}^{2}$ determines a univalent circle packing $K(\mathcal{C})$ for $K$, unique up to M\"obius transformations of the sphere.
\end{KCPT}
Here the \textit{circle packing}\index{circle packing} $K(\mathcal{C})$ is a collection $\mathcal{C} = \{C_{v}: v\in V(K)\}$ of circles $C_{v}$ in the sphere $\mathbb{S}^{2}$ indexed by the vertex set $V(K)$ of $K$ such that $C_{u}$ and $C_{v}$ are tangent whenever $uv$ is an edge of $K$, and for which circles $C_{u}$, $C_{v}$, and $C_{w}$ bound a positively oriented interstice whenever $uvw$ is a positively oriented face of $K$. The circle packing is \textit{univalent} if there is a collection $\mathcal{D} = \{D_{v} : v\in V(K)\}$ of disks with $C_{v} = \partial D_{v}$ whose interiors are pairwise disjoint.\footnote{Without univalence, packings with branching would be allowed, where the sequence of circles tangent to a single circle $C$ may wrap around $C$ multiple times before closing up. See Section~\ref{Section:branched}.} Connecting the centers of the adjacent circles by appropriate great circular arcs then produces a geodesic triangulation of $\mathbb{S}^{2}$ isomorphic to $K$. \Cref{fig:Koebe} shows a  circle packing of the sphere determined by an abstract triangulation $K$, and the realization of $K$ as a geodesic triangulation. Of course the circle packings for a fixed $K$ are M\"obius equivalent, while the corresponding geodesic triangulations are not, simply because neither circle centers nor great circles are M\"obius-invariant. I will look at a proof of the Koebe Circle Packing Theorem later, but first I'll present Thurston's generalization.
\begin{figure}
\begin{subfigure}[b]{0.4\textwidth}
	\includegraphics[width=\textwidth]{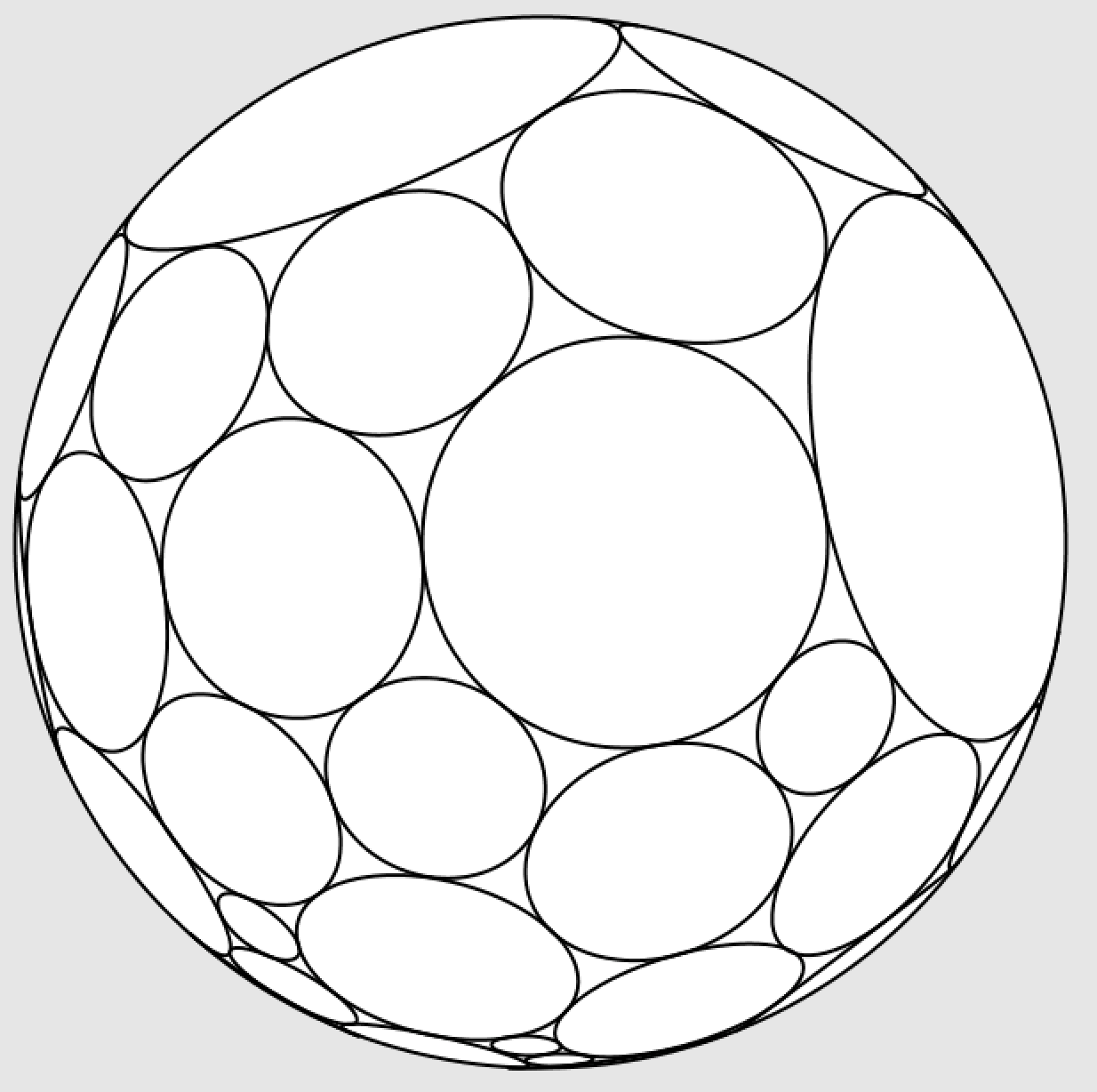}
\caption{The circle packing determined by a triangulation $K$ of $\mathbb{S}^{2}$.}\label{1.B}
\end{subfigure}
\quad
\begin{subfigure}[b]{0.4\textwidth}
	\includegraphics[width=\textwidth]{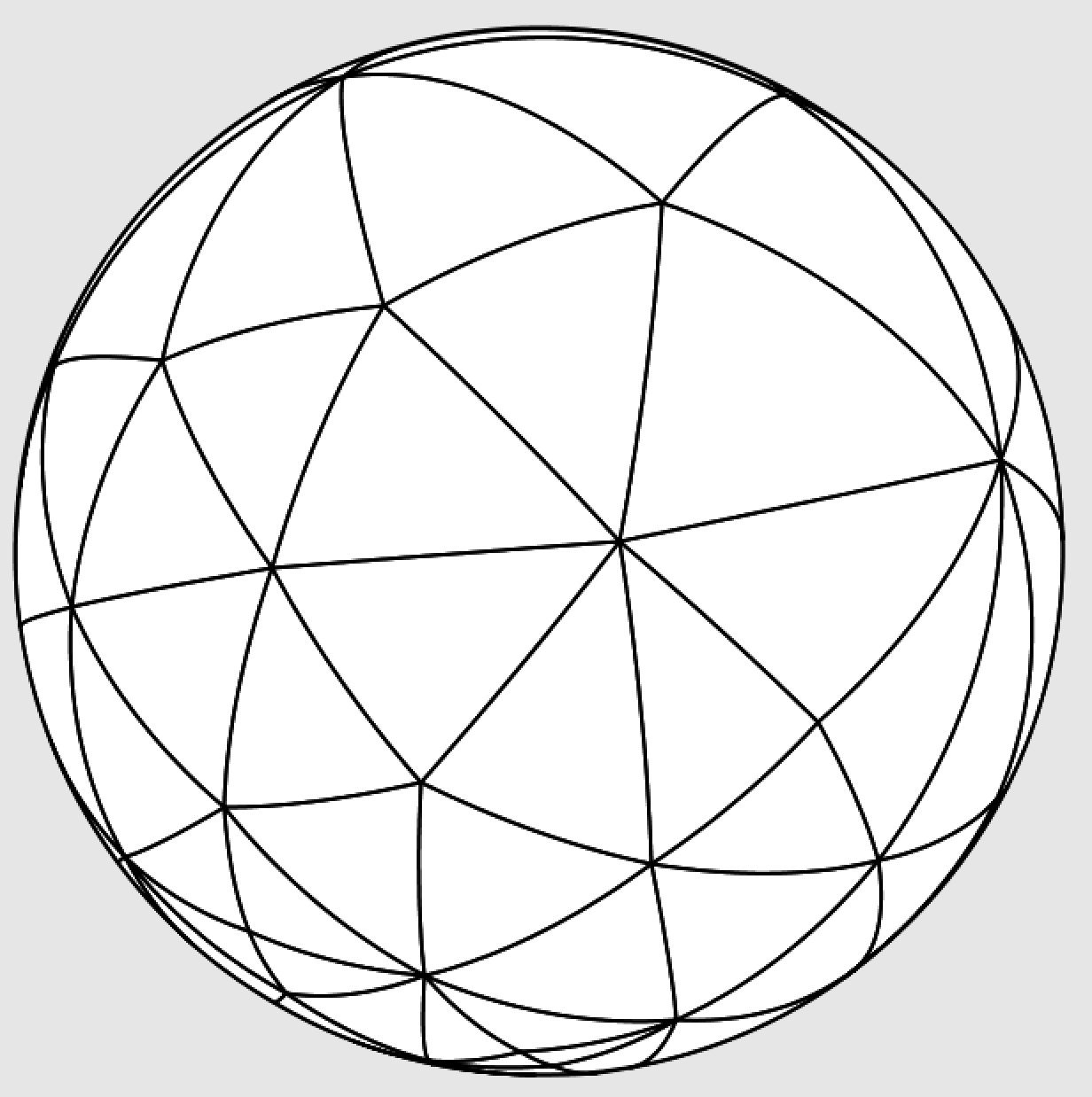}
\caption{The corresponding geodesic triangulation of $\mathbb{S}^{2}$.}\label{1.A}
\end{subfigure}
      \caption{An abstract triangulation $K$ of the $2$-sphere determines (A) a circle packing, which in turn determines a realization of $K$ as (B) a geodesic triangulation of the $2$-sphere.}\label{fig:Koebe}
\end{figure}

\subsection{Koebe-Andre'ev-Thurston, or KAT for short.}\index{Koebe-Andre'ev-Thurston Theorem} In his Princeton course of 1978-79, Thurston greatly generalized the Koebe Circle Packing Theorem, though at the time he was unaware of Koebe's result. He generalized in two ways, first by allowing adjacent circles to overlap and second by extending the theorem to arbitrary compact orientable surfaces. Thurston realized that his version of the theorem on the sphere $\mathbb{S}^{2}$ in fact encodes information about convex hyperbolic polyhedra, the connection of course through the fact that the sphere $\mathbb{S}^{2}$ serves as the space at infinity of the Beltrami-Klein and Poincar\'e ball versions of hyperbolic three-space $\mathbb{H}^{3}$ with circles on the sphere the ideal boundaries of hyperbolic planes in $\mathbb{H}^{3}$. These polyhedra had been characterized in two papers of Andre'ev from 1970, whose results can be interpreted in terms of the existence and uniqueness of the circle packings Thurston examined in his generalization of Koebe. Thurston's generalization to overlapping packings on the sphere is now known as the \textit{Koebe-Andre'ev-Thurston Theorem}, honoring its three principle protagonists.

\begin{KAT}
Let $K$ be an oriented simplicial triangulation of $\mathbb{S}^{2}$, different from the tetrahedral triangulation, and let $\Phi : E(K) \to [0, \pi/2]$ be a map assigning angle values to each edge of $K$. Assume that the following two conditions hold.
\begin{enumerate}
\item[(i)] If $e_{1}, e_{2}, e_{3}$ form a closed loop of edges from $K$ with $\sum_{i=1}^{3} \Phi(e_{i}) \geq \pi$, then $e_{1}$, $e_{2}$,  and $e_{3}$ form the boundary of a face of $K$.
\item[(ii)] If $e_{1}, e_{2}, e_{3}, e_{4}$ form a closed loop of edges from $K$ with $\sum_{i=1}^{4} \Phi(e_{i}) = 2\pi$, then $e_{1}$, $e_{2}$, $e_{3}$, and $e_{4}$ form the boundary of the union of two adjacent faces of $K$.
\end{enumerate}
Then there is a realization of $K$ as a geodesic triangulation of $\mathbb{S}^{2}$ and a family $\mathcal{C} = \{ C_{v} : v \in V(K)\}$ of circles centered at the vertices of the triangulation so that the two circles $C_{v}$ and $C_{w}$ meet at angle $\Phi(e)$ whenever $e=vw$ is an edge of $K$. The circle packing $\mathcal{C}$ is unique up to M\"obius transformations.
\end{KAT}

%{\color{red} Review Thurston's proof presented in GTTM }

Now I want to point out that exactly what is called the \textit{Koebe-Andre'ev-Thurston Theorem} is not at all settled. Some references use the term to mean the tangency case of the theorem ($\Phi \equiv 0$), which is nothing more than the Koebe Circle Packing Theorem, while others use the term to mean Thurston's full generalization of the theorem to arbitrary closed surfaces that is presented in Section~\ref{Section:FullThurston}. Exactly what Thurston proved in GTTM also often is misreported. In fact my introduction to this section is a bit of a misreporting, so let me take a little time to say exactly what Thurston does in Chapter 13 of GTTM. 

In terms of circle packings on the $2$-sphere, Thurston does not allow overlaps of adjacent circles, only tangencies. His version of the tangency case appears as Corollary 13.6.2 in Chapter 13 of GTTM, and appears as a corollary of Theorem 13.6.1, which he attributes to Andre'ev. This theorem concerns hyperbolic structures on orbifolds and, as it was Thurston who invented the notion of orbifold in his course at Princeton during 1976-77 as recorded in the footnote on page 13.5 of Chapter 13 of GTTM itself, this theorem is an interpretation of Andre'ev's in the context of orbifolds. Thurston does not give a proof of Theorem 13.6.1, but uses its result ensuring the existence of a hyperbolic structure on a suitable orbifold to prove Koebe's Theorem of 1936, Corollary 13.6.2. He does this by using the triangulation $K$ to define an associated polyhedron $P$ by cutting off vertices by planes that pass through midpoints of edges. He then uses the Andre'ev result to realize $P$ as a right-angled ideal polyhedron in $\mathbb{H}^{3}$. The faces of this polyhedron then lie in planes whose ideal circular boundaries are the circles of the desired tangency packing complemented by the orthogonal circles through three mutually adjacent points of tangency. He then invokes Mostow rigidity for uniqueness. 

It isn't until he presents Theorem 13.7.1 that Thurston allows for adjacent circles to overlap with angle between zero and $\pi/2$, and that only for surfaces other than the sphere, those surfaces with nonpositive Euler characteristic. Thurston proves this by assigning polyhedral metrics with curvature concentrated at the vertices $v_{1}, \dots ,v_{n}$ by assigning a radius $r_{i}$ at vertex $v_{i}$. Defining the mapping $c:\mathbb{R}^{n} \to \mathbb{R}^{n}$ that measures the curvature via $c(r)_{i} = 2\pi - (\text{the angle sum at vertex }v_{i})$, he then argues in nine pages that the origin $\mathbf{0}$ is in the image of $c$, which implies the desired result. It is the case that the version Thurston presents on the sphere, Corollary 13.6.2, is Koebe's result, and uses Andre'ev's ideas for the proof. It is only with this positive genus version, Theorem 13.7.1, that Thurston puts forth new geometric ideas, fertile enough to spawn an industry dedicated to understanding polyhedral metrics on surfaces and their induced circle packings.

Thurston's approach to circle packing is rather entwined with his overall concern, that of hyperbolic structures on three-dimensional manifolds and orbifolds. Since this work of the nineteen-seventies, Thurston's circle packing results have spawned a rather extensive theory that is more combinatorial and geometric, and related more to classical complex function theory and Riemann surfaces, and less to three-manifolds. It is related intimately to hyperbolic polyhedra and their generalizations, this the subject of Section~\ref{Section:Cage}, and has found several scientific applications. In the hands of folks like Ken Stephenson and his students and collaborators, it has spawned a discrete theory of complex analytic functions, laid out ever so elegantly in Stephenson's \textit{Introduction to Circle Packing}~\cite{kS05}. It has yielded beautiful results on, for example, discrete minimal surfaces in the hands of the Berlin school of Bobenko, Hoffman, Springborn, Pinkall, and Polthier; see for example~\cite{BHS06} and ~\cite{PP12}. Though the theory now is rather mature, it continues to interact in new and interesting ways with new areas, for instance lying in the background in conformal tilings~\cite{BS14a,BS14b}, or in the foreground with its interaction with the classical rigidity theory of Euclidean frameworks~\cite{BBP18}. There is an immense literature here, and so much of it owes a great debt of gratitude to the insights of Bill Thurston.

\subsection{A proof of the Koebe Circle Packing Theorem.}
Rather than proving the whole of KAT I, I will address the case where $\Phi$ is identically zero and prove Koebe's result. The proof presented here can be modified to give a complete proof of KAT I, which is done in~\cite{BS96} in proving a generalization.\footnote{See Section~\ref{Section:branched}.} Koebe's original proof of his namesake theorem uses a limiting process on circle domains and classical analytic arguments on convergence of analytic families of maps, very much in the flavor of what we now teach as classical techniques in our complex analysis courses. There are now many proofs of the Koebe Circle Packing Theorem. To name a few, besides Koebe's, there is Thurston's in GTTM already outlined above based on Andre'ev's results on hyperbolic polyhedra, Al Marden and Burt Rodin's using piecewise flat polyhedral metrics, Alan Beardon and Ken Stephenson's~\cite{BSt90} that adapts the classical Perron method for constructing harmonic maps as an upper envelope of subharmonic maps, Colin de Verdi\`ere's~\cite{CdV91a} based on a variational principle, Igor Rivins's hidden in his paper~\cite{Rivin:1994kg} on Euclidean structures on triangulated surfaces, the author's~\cite{pB93} that turns the Beardon-Stephenson proof upside down to address packings on punctured surfaces, and Alexandre Bobenko and Boris Springborn's~\cite{BoSp04} that uses a minimal principle on integrable systems. Here I present a geometric and combinatorial proof where hyperbolic geometry is the crucial ingredient. The proof is a twist on the Perron method used by Beardon and Stephenson in~\cite{BSt90} and is specialized from a more general version that applies to arbitrary surfaces of finite conformal type that appears in~\cite{pB93}. We will see that it has the advantage of generalizing in interesting ways.

\begin{proof}[Proof of the Koebe Circle Packing Theorem]
By removing one vertex $v_{0}$ from $K$ and its adjacent edges and faces, one obtains a triangulation $T$ of a closed disk. Place a piecewise hyperbolic metric on $T$ as follows. For any positive function $r:V(T) \to (0, \infty)$, let $|T(r)|$ be the metric space obtained by identifying the face $v_{1}v_{2}v_{3}$ of $T$ with the hyperbolic triangle of side lengths $r(v_{i}) + r(v_{j})$ for $i\neq j \in \{1,2,3\}$. This places a piecewise hyperbolic metric on $T$ with cone-like singularities at the interior vertices. This structure often now is called a piecewise hyperbolic \textit{polyhedral metric}, and the function $r$ is called variously a \textit{radius vector} or \textit{label}. For any vertex $v$, one can measure the angle sum $\theta_{r}(v)$ of the angles at $v$ in all the faces incident to $v$. I will say that $r$ is a \textit{superpacking label} for $T$ if the angle sums of all interior vertices are at most $2\pi$, and a \textit{packing label}\footnote{For emphasis one sometimes calls this a \textit{hyperbolic packing label} to distinguish it from \textit{flat} or \textit{Euclidean packing labels} that also find their use in this discipline.} if all are equal to $2\pi$.

Now modify this a little by allowing $r$ to take infinite values at the boundary vertices. This causes some ambiguity only if there is a separating edge in $T$ that disconnects $T$ when removed. This will be taken care of later, so for now assume no separating edge exists. The goal is to find a packing label $\mathfrak{r}$ with $\mathfrak{r}(w) = \infty$ whenever $w$ is a boundary vertex. Assuming that such an $\mathfrak{r}$ exists, we may glue on hyperbolic half planes along the faces with two boundary vertices to give a complete hyperbolic metric on a topological disk, which must be isometric to the hyperbolic plane. This implies that the metric space $|T(\mathfrak{r})|$ is isometric to an ideal polygon in the hyperbolic plane whose sides are hyperbolic lines connecting adjacent ideal vertices that correspond to the boundary vertices of $T$. Now placing hyperbolic circles of radii $\mathfrak{r}(v)$ centered at interior vertices $v$ and horocycles centered at ideal vertices determined by the boundary vertices gives a univalent circle packing of the hyperbolic plane realized as, say, the Poincar\'e disk, the unit disk $\mathbb{D}$ in the complex plane with Poincar\'e metric $ds = 2 |dz| / (1-|z|^{2})$. The boundary circles are horocycles in the hyperbolic metric on the disk and are therefore circles internally tangent to the unit circle. Stereographic projection to the sphere $\mathbb{S}^{2}$ and addition of the equator as the circle corresponding to the vertex $v_{0}$ removed initially produces a univalent circle packing of the sphere in the pattern of $K$ as desired. Uniqueness follows from uniqueness of the packing label $\mathfrak{r}$ with infinite boundary values, which follows from the construction of $\mathfrak{r}$ explained next.

Define the function $\mathfrak{r}$ as
\begin{equation}\label{EQ:rfrakdef}
\mathfrak{r}(v) = \inf \,\{r(v) : r\in \mathfrak{R}\}
\end{equation}
\noindent where
\begin{equation*}
\mathfrak{R} = \{r: V(T) \to (0, \infty]: r \text{ is a superpacking label for $T$ with infinite boundary values} \}.
\end{equation*}
\noindent The claim is that this is the desired packing label. The first observation is that $\mathfrak{R} \neq \emptyset$ so that we are not taking the infimum of the empty set. This is because one may choose label values so large on the interior vertices that all of the faces become hyperbolic triangles whose interior angles are no more than $2\pi / d$, where $d$ is the maximum degree of all the vertices of $T$. It follows that $\mathfrak{r}$ is a non-negative function with infinite boundary values. To verify that $\mathfrak{r}$ is a packing label, I show that
\begin{enumerate}
\item[(i)] $\mathfrak{r}$ cannot take a zero value on any interior vertex, which then implies that $\mathfrak{r} \in \mathfrak{R}$,
\item[] and,
\item[(ii)] the angle sum at any interior vertex is $2\pi$, meaning further that $\mathfrak{r}$ is a packing label.
\end{enumerate}
\noindent We need two preliminary observations.
\begin{enumerate}
\item[(iii)] \textit{Hyperbolic area is bounded away from zero.} The hyperbolic area of the singular hyperbolic surface $|T(r)|$ is $\geq \pi$ for all superpacking labels $r\in \mathfrak{R}$.
\item[(iv)] \textit{Monotonicity of angles.} For a face $f=v_{0}v_{1}v_{2}$ of $T$, let $\alpha_{r}(i)$, for $i=0,1,2$, be the angle that the label $r\in \mathfrak{R}$ gives to $f$ at vertex $v_{i}$. Then $\alpha_{r}(0) \uparrow \pi$, $\alpha_{r}(1)\downarrow 0$, and $\alpha_{r}(2) \downarrow 0$ monotonically as $r(v_{0}) \downarrow 0$ when $r(v_{1})$ and $r(v_{2})$ are held fixed.
\end{enumerate}

In calculating the hyperbolic area to confirm item (iii), let $V(T)$ and $F(T)$ be the respective vertex and face sets of $T$ of respective cardinalities $\mathtt{V}$ and $\mathtt{F}$. The sum of the angles of a face when given its metric by $r$ is denoted $\alpha_{r} (f)$ so that its hyperbolic area is $A_{r}(f) = \pi - \alpha_{r}(f)$. Finally, with $\mathtt{V}_{\mathrm{int}}$ and $\mathtt{V}_{\mathrm{bd}}$ denoting the numbers of interior and boundary vertices of $T$ so that $\mathtt{V} = \mathtt{V}_{\mathrm{int}} + \mathtt{V}_{\mathrm{bd}}$, one has
\begin{equation}\label{EQ:hyp-area}
\textrm{hyp-area}\left(|T(r)|\right) = \pi \,\mathtt{F} - \sum_{f\in F(T)} \alpha_{r}(f) 
    = \pi \,\mathtt{F} - \sum_{v\in V(T)} \theta_{r}(v) \geq \pi (\mathtt{F} -2\mathtt{V}_{\mathrm{int}}),
\end{equation}
\noindent since $\theta_{r}(v) \leq 2\pi$ at interior vertices and $\theta_{r}(v) = 0$ at boundary ones. An Euler characteristic exercise then shows that $\mathtt{F} -2\mathtt{V}_{\mathrm{int}} = \mathtt{V}_{\mathrm{bd}} - 2 \geq 1$, the inequality holding since $K$ is simplicial. It follows that every superpacking label with infinite boundary values produces a metric on $T$ with hyperbolic area at least $\pi$. Item (iv) is almost obvious from drawing examples, but can be given a rigorous proof using the hyperbolic law of cosines from hyperbolic trigonometry.

I now address item (i). First the claim is that the label $\mathfrak{r}$ cannot be identically zero on the set of interior vertices. Indeed, if $\mathfrak{r}$ is identically zero, one may choose a sequence of superpacking labels $r_{i}$ with infinite boundary values such that, for each interior vertex $v$, $r_{i}(v) \to 0$ as $i \to \infty$. This latter fact in turn follows from the observation that the minimum label $\min \{r_{1}, r_{2}\}$ is in $\mathfrak{R}$ whenever $r_{1}$ and $r_{2}$ are labels in $\mathfrak{R}$, which in turn is a consequence of the monotonicity of angles (iv). Recall that we are under the assumption that there are no separating edges so that at least one vertex of any face $f$ of $T$ is interior. Any such interior vertex has $r_{i}$-values converging to zero, and any boundary one is fixed at infinity, and with this it is easy to see that the hyperbolic area $A_{r_{i}}(f) \to 0$ as $i \to \infty$. But this implies that the hyperbolic area of $|T(r_{i})|$ converges to zero as $i\to \infty$, which contradicts item (iii). 

Now could it be that $\mathfrak{r}$ takes a zero value at some interior vertex, but not at all? The argument that this in fact does not happen is a generalization of what I have argued thus far. I will but give an indication of how it goes, referring the reader to \cite{pB93} for details. Let $T'$ be the subcomplex of $T$ determined by those faces of $T$ that have a vertex in $\mathfrak{r}^{-1}(0)$. An argument using Euler characteristic similar to that already given implies that the hyperbolic area of $|T'(r)|$ is positive and bounded away from zero for every superpacking label $r$ with fixed non-negative boundary values. But an argument as in the preceding paragraph shows that the hyperbolic areas of $|T'(r_{i})|$ converge to zero for a sequence of superpacking labels with fixed boundary values and interior vertex values converging to zero. This contradiction implies that $\mathfrak{r}$ is a positive function on the interior vertex set, and continuity of angles of a triangle with respect to edge lengths implies that $\theta_{\mathfrak{r}}(v) = \lim_{i\to \infty}\theta_{r_{i}}(v) \leq 2\pi$ at any interior vertex, since $\theta_{r_{i}}(v) \leq 2 \pi$ for all $i$. This shows that $
\mathfrak{r} \in \mathfrak{R}$ and completes the verification of item (i).

Item (ii) follows quickly from item (iv). Indeed, if (ii) fails, then there is an interior vertex $v$ of $T$ such that $\theta_{\mathfrak{r}}(v) < 2\pi$. By the monotonicity properties (iv), varying $\mathfrak{r}$ by slightly decreasing its value at $v$ without changing any other values increases $\theta_{\mathfrak{r}}(v)$ while decreasing $\theta_{\mathfrak{r}}(w)$ for any vertex $w$ incident to $v$. By making that decrease of $\mathfrak{r}(v)$ small enough to keep the angle sum at $v$ below $2\pi$, we obtain a superpacking label $r$ with infinite boundary values that satisfies $r(v) < \mathfrak{r}(v)$, which contradicts the definition of $\mathfrak{r}$ in Equation~\ref{EQ:rfrakdef}.

At this point I have shown that $\mathfrak{r}$ is a packing label with infinite boundary values, and I now claim that it is the only one. Suppose there is a packing label $r$ in $\mathfrak{R}$ that differs from the infimum label $\mathfrak{r}$ defined in Equation~\ref{EQ:rfrakdef}. Then $\mathfrak{r} (v) \leq r(v)$ for all vertices $v$, but there must be some interior vertex $w$ with $\mathfrak{r}(w) < r(w)$. This implies that the hyperbolic area of the surface $|T(\mathfrak{r})|$ is strictly less than that of $|T(r)|$. But this is impossible since $\mathfrak{r}$ and $r$ are packing labels with infinite boundary values, and as argued above, both $|T(\mathfrak{r})|$ and $|T(r)|$ are ideal hyperbolic polygons with $\mathtt{V}_{\mathrm{Bd}}$ sides. An easy exercise shows that the hyperbolic area of any such hyperbolic polygon is equal to $(\mathtt{V}_{\mathrm{Bd}} -2)\pi$. 

This completes the proof modulo the assumption that $T$ has no separating edge. This is handled by induction on the number of such edges. If there is one separating edge $uv$, cut $T$ into $T_{1}$ and $T_{2}$ along that edge, circle pack each in the unit disk with horocyclic boundary circles, and then using M\"obius transformations, place the $T_{1}$ packing in the upper half disk with the horocycles for $u$ and $v$ circles tangent at the origin and centered on the real axis, and place the $T_{2}$ packing in the lower half of the disk with those same horocylic circles for $u$ and $v$. This is possible since $T$ is oriented, and this gives an appropriate packing label of $T$ with infinite boundary values.\end{proof}

\subsection{Maximal packings and the boundary value problem.} This proof actually proves the following extremely useful fact, which Beardon and Stephenson~\cite{BSt90} exploited to give the first extension of the Koebe Circle Packing Theorem to infinite packings of the disk and the plane. The infinite theory is presented in Section~\ref{Section:InfinitePackings}.

\begin{MDPT}\label{Theorem:MDPT}
Every oriented simplicial triangulation $T$ of a closed disk determines a univalent circle packing $T(\mathcal{C})$ for $T$ in the unit disk $\mathbb{D}$ in the complex plane $\mathbb{C}$, unique up to M\"obius transformations of the disk, with the circles corresponding to boundary vertices of $T$ internally tangent to the unit circle boundary $\partial \mathbb{D} = \mathbb{S}^{1}$. Moreover, when given its canonical hyperbolic metric making $\mathbb{D}$ into the Poincar\'e disk model of the hyperbolic plane $\mathbb{H}^{2}$, the circle radii of the packing are uniquely determined by $T$.  
\end{MDPT}

The circle packing guaranteed by this theorem is called the \textit{maximal packing} for $T$. This theorem is in fact a special case of the more general result of Beardon and Stephenson~\cite{BSt91a} that solves the discrete version of the classical Dirichlet boundary value problem of harmonic analysis. In that paper, the authors also prove a discrete version of the classical Schwarz-Pick Lemma of complex analysis. These two theorems finish up the present section.

\begin{DBVT}[Beardon and Stephenson~\cite{BSt91a}]
Let $T$ be an oriented simplicial triangulation of a closed disk and $f:V_{\mathrm{Bd}}(T) \to (0,\infty]$ a function assigning positive or infinite values to the boundary vertices. Then there exists a unique hyperbolic packing label $\mathfrak{r}: V(T) \to (0,\infty]$ extending $f$. The resulting circle packing $T(\mathcal{C}_{\mathfrak{r}})$ of the unit disk $\mathbb{D}$ is unique up to M\"obius transformations of $\mathbb{D}$.
\end{DBVT}

\begin{proof}
The proof is a straightforward modification of that of the Koebe Circle Packing Theorem already presented. Again $\mathfrak{r} = \inf \mathfrak{R}$ is the desired packing label, provided that
\begin{equation*}
\mathfrak{R} = \{r: r \text{ is a superpacking label for $T$ with }r(w) \geq f(w) \text{ when } w\in V_{\mathrm{Bd}}(T) \}.
\end{equation*}
Of course, $f \equiv \infty$ gives the maximal packing of the preceding theorem.
\end{proof}
This proof is a modification of the Beardon-Stephenson proof, which uses \textit{subpacking} rather than superpacking labels. In a subpacking label, the interior angle sums are greater than or equal to $2\pi$ and one obtains the packing label as an upper envelope of subpacking labels, with the packing label given by $\mathfrak{r} = \sup \mathfrak{R}'$ where $\mathfrak{R}'$ is the set of subpacking labels with boundary values given by $f$. The advantage of approaching the desired packing label $\mathfrak{r}$ from above using superpackings ($\inf \mathfrak{R}$) rather than from below using subpackings ($\sup \mathfrak{R}'$) is that this \textit{upper Perron method} readily generalizes to include cusp type singularities and cone type singularities at interior vertices.\footnote{Another not insignificant advantage is that it is easy to show that $\mathfrak{R} \neq \emptyset$ while proving that $\mathfrak{R}' \neq \emptyset$ generally is difficult.} This is presented in Section~\ref{Section:Singularities}.

A word of warning here. When the boundary values are allowed to be finite, the resulting packing, though locally univalent, may not be globally univalent. This means that the disks bounded by the circles of the packing may overlap non-trivially, though ones neighboring the same interior vertex never do; this is the meaning of \textit{locally univalent}. \Cref{fig:DBVP} shows a locally univalent packing that is not globally univalent.
\begin{figure}
\includegraphics[width=3.5in]{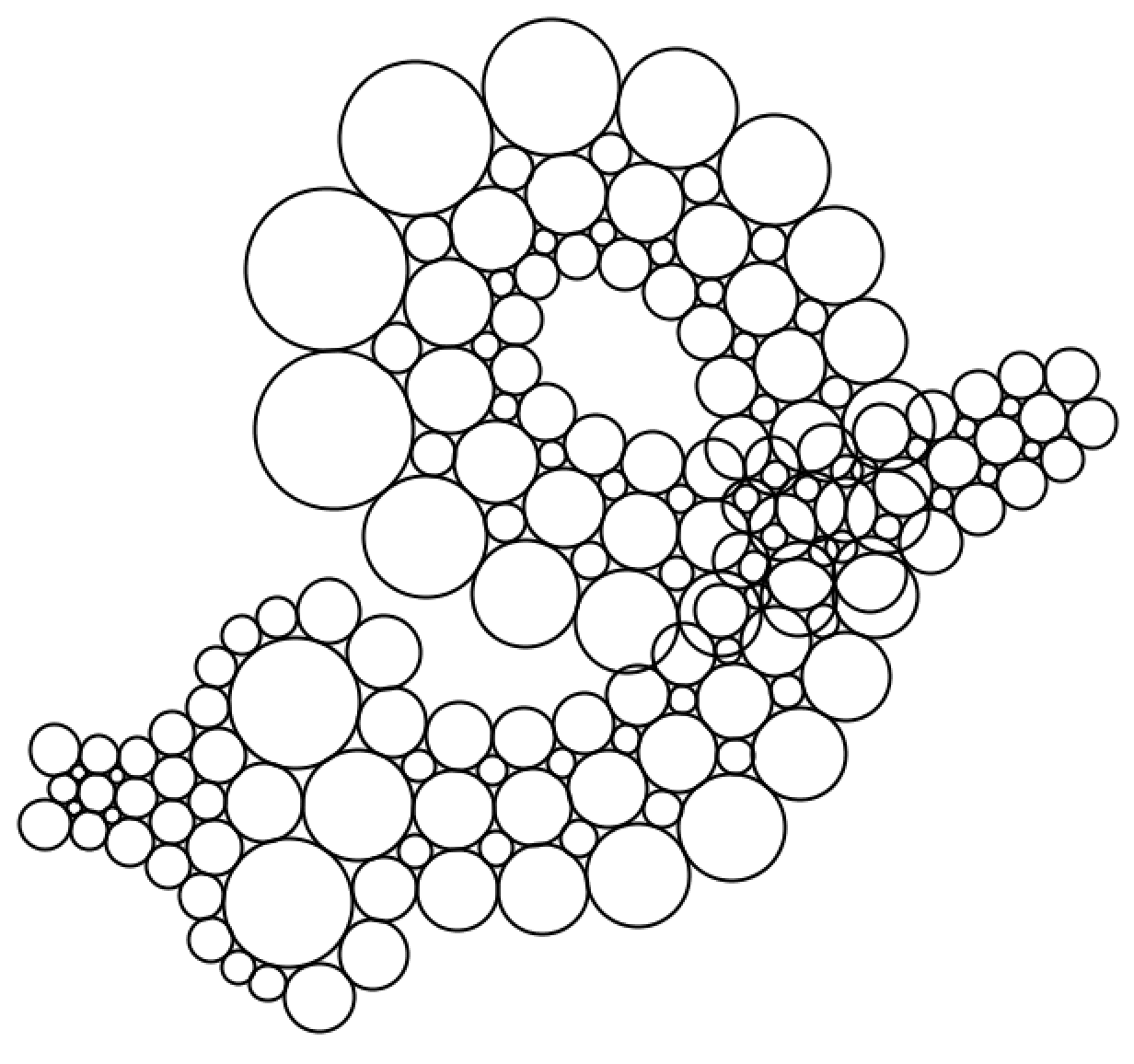}
\caption{A locally univalent circle packing that is not globally univalent.}
    \label{fig:DBVP}
\end{figure}

The second theorem of Beardon and Stephenson follows partly from the proof of the first (item (i)), and the rest of the theorem follows from a careful analysis of paths and angles in piecewise hyperbolic surfaces. The details of course appear in~\cite{BSt91a}. The theorem I state here is the generalization of what the reference~\cite{BSt91a} calls the Discrete Schwarz-Pick Lemma, which in turn is obtained by setting the boundary values of $r$ identically to infinity.
\begin{DSPL}[Beardon and Stephenson~\cite{BSt91a}]
Let $r, r' ; V(T) \to (0,\infty]$ be packing labels for the oriented simplicial triangulation $T$ of a closed disk. Suppose that $r(w) \geq r'(w)$ at every boundary vertex $w\in V_{\mathrm{Bd}}(T)$. Then
\begin{enumerate}
    \item[(i)] $r \geq r'$; i.e., $r(v) \geq r'(v)$ at every vertex $v$ of $T${\em ;}
    \item[(ii)] $\rho_{r}(u,v) \geq \rho_{r'} (u,v)$ for any two vertices $u$ and $v$, where $\rho_{r}$ is the distance function on the metric surface $|T(r)|$, and similarly for $\rho_{r'}${\em ;}
    \item[(iii)] $A_{r}(f) \geq A_{r'}(f)$ for any face $f$ of $T$. {\em (}Recall that $A_{r}(f)$ is the hyperbolic area of the face $f$.{\em )}
\end{enumerate}
Moreover, if a single instance of finite equality occurs at an interior vertex in {\em (i)}, or at vertices $u$ and $v$ at least one of which is interior in {\em (ii)}, or at any face in {\em (iii)}, then $r=r'$.
\end{DSPL}

\section{The Koebe-Andre'ev-Thurston Theorem, Part II}\index{Koebe-Andre'ev-Thurston Theorem}
\subsection{Circle packings of compact surfaces.} 
Thurston's other avenue of generalization of Koebe, indeed the more far-reaching one, is his extension of KAT to arbitrary orientable closed surfaces. Here there are striking illustrations of how purely combinatorial information encodes precise geometry. I will start with Thurston's tangency case of packings before presenting his version with overlaps.

\begin{Theorem}\label{Theorem:SurfacesPack}
Let $K$ be an oriented simplicial triangulation of a closed surface $S_{g}$ of positive genus. Then there is a metric of constant curvature, unique up to scaling, on $S_{g}$ that supports a univalent, tangency circle packing $\mathcal{C} = \{ C_{v}: v\in V(K)\}$ modeled on $K$. In particular, $C_{u}$ and $C_{v}$ are tangent whenever $uv$ is an edge of $K$. The packing $\mathcal{C}$ is unique up to isometries of $S_{g}$ in this metric when $g\geq 2$, and up to scaling when $g=1$. Connecting the centers of adjacent circles by geodesic shortest paths produces a geodesic triangulation of the surface in the pattern of $K$. The metric is locally Euclidean when $g=1$ and locally hyperbolic otherwise.
\end{Theorem}

Just in case the reader blinked and missed it, I aim to emphasize the extent to which combinatorics determines geometry in this theorem. The simplicial complex $K$ provides purely combinatorial data with topological overtones. Yet hidden inside of the combinatorics is precise geometry. For example in the hyperbolic case where $g > 1$, among the uncountably many possible pairwise distinct hyperbolic metrics of constant curvature $-1$ as tabulated in the $(6g-6$)-dimensional moduli space $\mathcal{M}(S_{g}) \cong \mathbb{R}^{6g-6}$, the complex $K$ chooses exactly one of these metrics, and in that metric, determines a univalent circle packing unique up to isometry! For none of the other metrics that $S_{g}$ supports is there a univalent tangency packing of circles in the pattern of $K$! Since there are only countably many pairwise distinct simplicial triangulations of the fixed surface $S_{g}$, only countably many of the metrics parameterized by $\mathcal{M}(S_{g})$ support any univalent tangency packing at all, though the set of metrics that do support such circle packings does form a dense subset of the moduli space. 

I present a proof of Theorem~\ref{Theorem:SurfacesPack} based on the upper Perron method used to prove the Koebe Circle Packing Theorem.
\begin{proof}
	Let $\mathfrak{R} = \{ r: V(K) \to (0, \infty): \theta_{r} (v) \leq 2\pi \text{ for all } v\in V(K)\}$, the set of superpacking labels for $K$. Here again, exactly as in the proof of the Koebe Circle Packing Theorem, the label $r$ determines a hyperbolic polyhedral metric surface $|K(r)|$. A unique packing label for which the angle sum at every vertex is equal to $2\pi$ would give all the claims of the theorem in the hyperbolic case. My claim is that when $g \geq 2$, the function $\mathfrak{r} = \inf \mathfrak{R}$ is the unique packing label for $K$, and when $g=1$, then $\mathfrak{r} = \inf \mathfrak{R}$ is identically zero, but provides a way to place a flat polyhedral metric on $K$ that meets the packing condition. 
	
Exactly the calculation of Inequality~\ref{EQ:hyp-area} gives $\text{hyp-area}(|K(r)|) \geq (\mathtt{F} - 2 \mathtt{V}) \pi$ for any superpacking label $r \in \mathfrak{R}$, and an Euler characteristic argument gives
	\begin{equation}
	\mathtt{F} - 2 \mathtt{V} = - 2 \chi(S_{g}) = 4g-4.
	\end{equation}
When $g\geq 2$ so that $\mathtt{F} - 2 \mathtt{V}$ is positive and hence $\text{hyp-area}(|K(r)|)$ is positive, the same argument used in the proof of the Koebe Circle Packing Theorem shows that items (i) and (ii) of that proof hold, so that $\mathfrak{r}$ is a packing label. Uniqueness follows exactly as in that proof.

The remaining case is when $g=1$ so that $S_{g}$ is a topological torus. Here are the steps in proving that $S_{1}$ supports a flat metric that supports a univalent circle packing in the pattern of $K$, both the packing and the metric unique up to scaling.
\begin{enumerate}
    \item[(i)] When $g=1$, $\mathtt{F} - 2\mathtt{V} = 0$ and this implies that $\mathfrak{r} = \inf \mathfrak{R}$ is identically zero on $V(K)$.
    \item[(ii)] Fix a vertex $v^{\dag}$ in $K$ and let $\mathfrak{R}^{\dag} = \{ r^{\dag}: r\in \mathfrak{R}\}$, where $r^{\dag}$ is the normalized label defined by $r^{\dag}(v) = r(v) / r(v^{\dag})$.
    \item[(iii)] Show that $\mathfrak{r}^{\dag} = \inf \mathfrak{R}^{\dag}$ takes only positive values.
    \item[(iv)] Let $|K(\mathfrak{r}^{\dag})|_{\text{flat}}$ be the flat polyhedral surface with cone type singularities obtained by identifying a face $v_{1}v_{2}v_{3}$ with the Euclidean triangle of side-lengths $\mathfrak{r}^{\dag}(v_{i}) + \mathfrak{r}^{\dag}(v_{j})$ for $i\neq j \in \{1,2,3\}$. 
    \item[(v)]Show that  $|K(\mathfrak{r}^{\dag})|_{\text{flat}}$ is non-singular; i.e.,  $\mathfrak{r}^{\dag}$ is a flat packing label with Euclidean angle sums $\theta_{\mathfrak{r}^{\dag}}^{\text{flat}}(v) = 2\pi$ at every vertex $v$.
    \item[(vi)] Show that $\mathfrak{r}^{\dag}$ is the unique flat packing label with value $1$ at $v^{\dag}$.
\end{enumerate}
The details of the argument appear in~\cite{BSt90}, but I will give an indication of why this outline works to prove the desired result. Let $A(r)$ be the hyperbolic area of the singular hyperbolic surface $|K(r)|$ when $r \in \mathfrak{R}$ and observe that 
\begin{equation}
A(r) - s(r) = (\mathtt{F} - 2 \mathtt{V}) \pi, \quad \text{where} \quad s(r) = \sum_{v\in V(K)} (2\pi - \theta_{r}(v)).
\end{equation}
Here $s(r)$ is the total \textit{angle shortage}.\footnote{Also called the \textit{discrete curvature}.} In the genus $1$ case, $\texttt{F} - 2 \mathtt{V} =0$ so $A(r) = s(r)$ for all superpacking labels $r\in \mathfrak{R}$. Now assuming that item (i) has been verified, any superpacking label $r$ that is close to the infimum $\inf \mathfrak{R} = 0$ has area $A(r)$ close to zero and hence so too is the shortage $s(r)$ close to zero. In the limit as $r\to \inf \mathfrak{R} =0$, the shortages $s(r) \to 0$ and this implies that the singular hyperbolic surfaces $|K(r)|$ have angle sums $\theta_{r} (v) \to 2\pi$ for every vertex $v$. Since Euclidean geometry is the small scale limit of hyperbolic geometry, this implies that the Euclidean angle sums $\theta_{r}^{\text{flat}}(v) \to 2\pi$ as $r \to 0$. Thus the collection $\{|K(r)|_{\text{flat}}\}_{r \in \mathfrak{R}}$ is a collection of singular flat surfaces whose singularities are removed in the limit as $r\to 0$. Of course there is no limiting surface since $r \to 0$. Whereas this cannot be remedied in hyperbolic geometry, it can be remedied in Euclidean geometry by rescaling the labels $r$ as described in item (ii). With item (iii) confirmed so that the flat polyhedral surface $|K(\mathfrak{r}^{\dag})|_{\text{flat}}$ of item (iv) exists, since similarity transformations exist in Euclidean geometry, these rescalings preserve the Euclidean angles and imply that the limit surface $|K(\mathfrak{r}^{\dag})|_{\text{flat}}$ is non-singular.  Items (v) and (vi) just state formally the result of making this imprecise but rather accurate discussion rigorous.
\end{proof}

\subsection{KAT for compact surfaces}\label{Section:FullThurston} Thurston's Theorem 13.7.1 of GTTM combines the introduction of surfaces of genus greater than zero in Theorem~\ref{Theorem:SurfacesPack} with the overlap conditions of the KAT Circle Packing Theorem.

\begin{KAT2}[Theorem 13.7.1, GTTM]
	Let $K$ be an oriented simplicial triangulation of a surface $S_{g}$ of genus $g\geq 1$, and let $\Phi : E(K) \to [0, \pi/2]$ be a map assigning angle values to each edge of $K$. Assume that the following two conditions hold.
\begin{enumerate}
\item[(i)] If $e_{1}, e_{2}, e_{3}$ form a closed loop of edges from $K$ with $\sum_{i=1}^{3} \Phi(e_{i}) \geq \pi$, then $e_{1}$, $e_{2}$,  and $e_{3}$ form the boundary of a face of $K$.
\item[(ii)] If $e_{1}, e_{2}, e_{3}, e_{4}$ form a closed loop of edges from $K$ with $\sum_{i=1}^{4} \Phi(e_{i}) = 2\pi$, then $e_{1}$, $e_{2}$, $e_{3}$, and $e_{4}$ form the boundary of the union of two adjacent faces of $K$.
\end{enumerate}
Then there is a metric of constant curvature on $S_{g}$, unique up to scaling, and a realization of $K$ as a geodesic triangulation in that metric, as well as a family $\mathcal{C} = \{ C_{v} : v \in V(K)\}$ of circles centered at the vertices of the triangulation so that the two circles $C_{v}$ and $C_{w}$ meet at angle $\Phi(e)$ whenever $e=vw$ is an edge of $K$. The circle packing $\mathcal{C}$ is unique up to isometry.
\end{KAT2} 

I already have discussed the proof in GTTM. Let me say further that it was in this proof that Thurston introduced the idea of using labels, or radii assignments to vertices, to build a polyhedral surface with cone type singularities, and then to vary the labels until the packing condition is met. This is still the basic idea for proving many packing results, though the way in which one varies the labels and the choice of initial labels changes from researcher to researcher and from application to application. The Perron method used in this article is a modification of the method of Beardon and Stephenson~\cite{BSt90}. This idea also led to a practical algorithm for producing the packing labels that was the starting point for Ken Stephenson's \texttt{CirclePack}. This sophisticated software package for computing circle packings has enjoyed extensive development over the past thirty years and is freely available at Ken's webpage.  

Before I introduce infinite circle packings and their really interesting and novel features in Section~\ref{Section:InfinitePackings}, I'll discuss two generalizations of the KAT Theorems. The first is presented in Section~\ref{Section:branched} and generalizes KAT I to certain branched packings of the $2$-sphere where circles tangent to a given one wrap around that one more than once. These packings of course fail to be univalent, but provide a rich family of packings that model the behavior of polynomial mappings of the Riemann sphere. The ultimate goal is to model arbitrary rational mappings of the sphere, which would require the theory to extend to more general branch structures, this a topic of current research; see for example~\cite{ACS16}. The second is presented in Section~\ref{Section:Singularities} and examines how to include both cusps with ideal vertices as well as prescribed discrete curvature at pre-chosen vertices.

\subsection{A branched KAT theorem and polynomial branching}\label{Section:branched}\index{circle packing!branched} Ken Stephenson and I generalized KAT I by allowing for \textit{polynomial branching} to occur in the circle packing. \textit{Branching} means that we allow for the angle sums at predetermined vertices to be a positive integer multiple of $2\pi$ rather than just $2\pi$, or stated differently, we allow the circles tangent to a given one to wrap around that given circle multiple times before closing up; see \Cref{Figure:Branching}.
\begin{figure}
\includegraphics[width=3.5in]{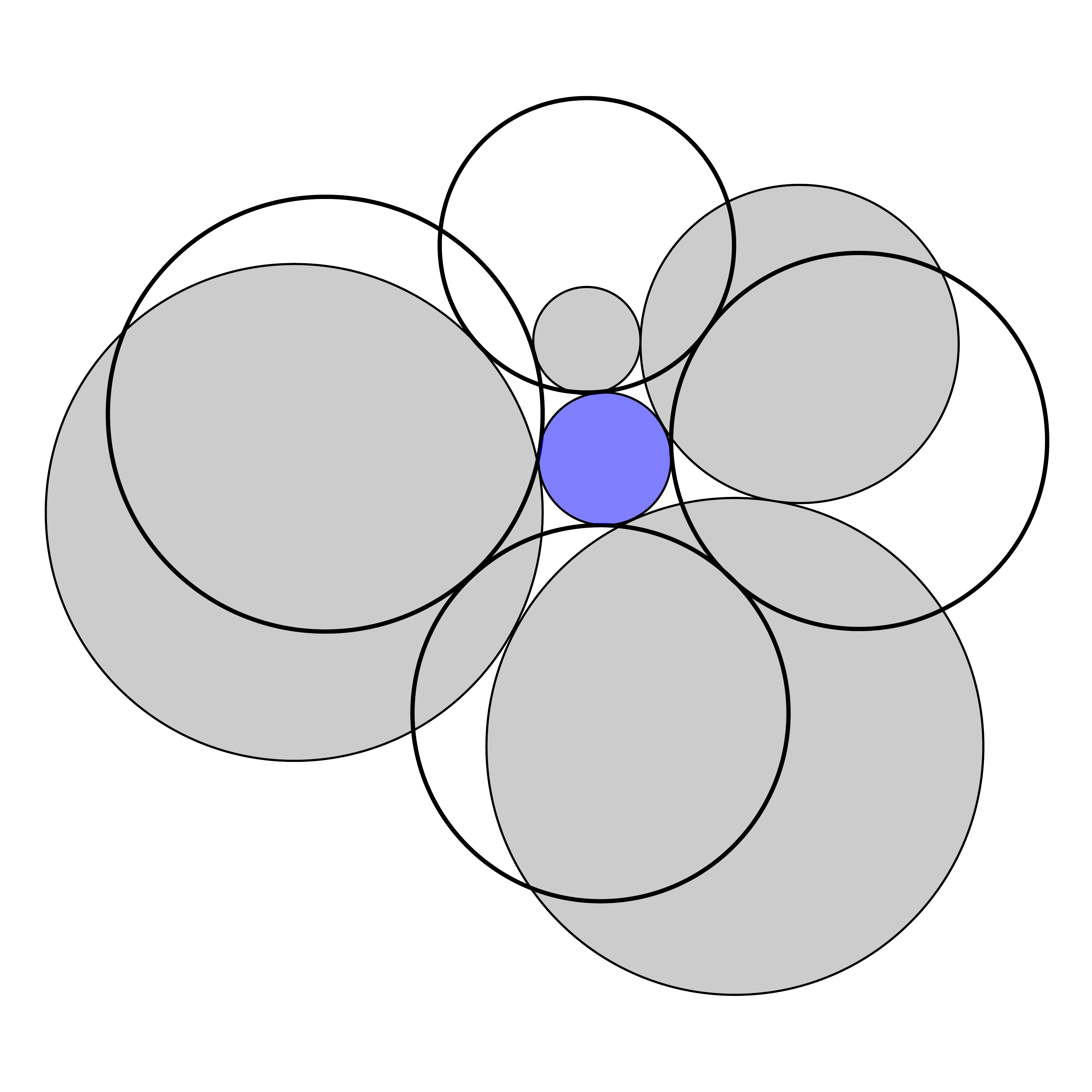}
    \caption{Branching of multiplicity $m=2$ or order $\mathfrak{o} = 1$. Starting with the grey disk on the left and moving counterclockwise, four sequentially tangent grey disks wrap around the blue central disk nearly one full turn, at which point the sequentially tangent transparent (or white) disks take over to wrap around slightly more than one full turn to close up the flower of circles with angle sum $\theta = 4\pi$.}
    \label{Figure:Branching}
\end{figure}
\textit{Polynomial} means that half the branching is concentrated at one vertex. The terminology comes from the classical theory of rational maps. Indeed, rational mappings may be thought of as branched self-mappings of the $2$-sphere, and the polynomial mappings are precisely those in which there is an even amount of branching with half the branching occurring at a single point. Taken together, a circle packing promised by the next theorem mimics the behavior of a polynomial mapping of the Riemann sphere. 

Our proof of the theorem as presented in~\cite{BS96} offers an independent proof of KAT I, which the branched version reduces to when the \textit{branch structure} $\beta$ is empty. In fact as far as I know, it was the first full direct proof of KAT I given that Thurston proves only the tangency case (the Koebe Circle Packing Theorem) and Marden-Rodin~\cite{MR90}, though allowing overlapping circles, has more restrictive hypotheses. KAT I is implied by Igor Rivin's earlier work, which bears the same resemblance to KAT I as does Andre'ev's in that it is a result on the existence of hyperbolic polyhedra. 

I state the result and then backtrack to fill in definitions and discuss the proof.

\begin{BranchedKAT}[Bowers and Stephenson~\cite{BS96}]
Let $K$ be an oriented simplicial triangulation of $\mathbb{S}^{2}$, different from the tetrahedral triangulation, and let $\Phi : E(K) \to [0, \pi/2]$ be a map assigning angle values to each edge of $K$. Assume that the following two conditions hold.
\begin{enumerate}
\item[(i)] If $e_{1}, e_{2}, e_{3}$ form a closed loop of edges from $K$ with $\sum_{i=1}^{3} \Phi(e_{i}) \geq \pi$, then $e_{1}$, $e_{2}$,  and $e_{3}$ form the boundary of a face of $K$.
\item[(ii)] If $e_{1}, e_{2}, e_{3}, e_{4}$ form a closed loop of edges from $K$ with $\sum_{i=1}^{4} \Phi(e_{i}) = 2\pi$, then $e_{1}$, $e_{2}$, $e_{3}$, and $e_{4}$ form the boundary of the union of two adjacent faces of $K$.
\end{enumerate}
If $\beta$ is a polynomial branch structure for the edge-labeled triangulation $(K, \Phi)$, then there exists a circle packing $\mathcal{C} = \{ C_{v} : v \in V(K)\}$ for $(K, \Phi)$, a family of circles in $\mathbb{S}^{2}$ so that the two circles $C_{v}$ and $C_{w}$ meet at angle $\Phi(e)$ whenever $e=vw$ is an edge of $K$, with $\mathrm{br}(\mathcal{C}) = \beta$. The circle packing $\mathcal{C}$ is unique up to M\"obius transformations.
\end{BranchedKAT}

A branch structure essentially is a listing of some of the vertices of $K$, each paired with an integer $\geq 2$ that indicates how many times the circles adjacent to the ones corresponding to the selected vertices wrap around before closing up. Before making this precise, let's observe that there must be further combinatorial conditions to ensure that a branched circle packing exists for the branch structure. Indeed, note that when there is no branching, the fact that $K$ is a simplicial triangulation implies that the degree of each vertex is at least three, and this local condition guarantees that there are enough circles adjacent to a given circle to wrap around once, with angle sum $2\pi$, at least in the tangency case. A moment's thought will show that if the desire is that there be branching of \textit{multiplicity} $m\geq 2$ at a circle $C_{v}$, meaning that the circles adjacent to $C_{v}$ wrap around $m$ times before closing up, there had better be at least $1+ 2m$ adjacent ones to achieve the angle sum of $2\pi m$. This may not be sufficient but certainly is necessary, and the definition of a polynomial branch structure includes enough combinatorial conditions to ensure sufficiency.

To clothe this discussion in a bit of flesh, suppose that $\mathcal{C} = \{C_{v}: v \in V(K)\}$ is a circle packing for the pair $(K, \Phi)$. For each vertex $v\in V(K)$, identify $v$ with the center of its corresponding circle $C_{v}$. Fixing a vertex $v$, let $v_{1}, \dots, v_{n}$ be the list of neighbors of $v$ forming the consecutive vertices in a walk around the boundary of the star $\text{st}(v)$ of $v$, and let $\alpha_{i}$ be the measure of the spherical angle $\angle v_{i}vv_{i+1}$. Then $v$ is said to be a \textit{branch point} of \textit{order $\mathfrak{o} = m-1$}, or of \textit{multiplicity $m$}, if $\theta(v) = 2\pi m$ for some integer $m\geq 2$, where $\theta(v) = \alpha_{1} + \cdots + \alpha_{n}$ is the \textit{angle sum} at $v$; again, see~\Cref{Figure:Branching}. The \textit{branch set} $\mathrm{br}(\mathcal{C})$ of the circle packing is the set of ordered pairs $(v,\mathfrak{o}(v))$ as $v$ ranges over the branch points and $\mathfrak{o}(v)$ is the order of $v$. It is clear that the combinatorics of $K$ as well as the values of $\Phi$ restrict the branch orders. 

My aim is to construct circle packings of $\mathbb{S}^{2}$ in the pattern of $K$ with overlaps given by $\Phi$ with a given, predetermined branch set. Toward this end, I will define a branch structure on the complex $T = K \setminus\text{Int}[\text{st}(v_{\infty})]$ that triangulates the closed disk one obtains by deleting one vertex, $v_{\infty}$, and its incident open cells from $K$. I will use $\Phi_{T}$ to mean the restriction of $\Phi$ to the vertices of $T$.
\begin{Definition}[\textsc{branch structure}]
A set $\beta = \{(v_{1}, \mathfrak{o}_{1},), \dots ,(v_{\ell}, \mathfrak{o}_{\ell})\}$, where $v_{i}, \dots, v_{\ell}$ is a pairwise distinct list of interior vertices of $T$ and each $\mathfrak{o}_{i}$ is a positive integer, is a \textit{branch structure} for the pair $(T, \Phi_{T})$ if the following condition holds: for each simple closed edge path $\gamma = e_{1} \cdots e_{n}$ in $T$ that bounds a combinatorial disk $D$ that contains at least one of the vertices $v_{i}$, the inequality
\begin{equation}\label{EQ:BranchCondition}
\sum_{i=1}^{n} \left[ \pi - \Phi_{T}(e_{i})\right] > 2\pi  (\mathfrak{o}(D) +1)
\end{equation}
holds, where $\mathfrak{o}(D) = \sum \mathfrak{o}_{i}$, the sum taken over all indices $i$ for which $v_{i} \in \text{Int} (D)$.
\end{Definition}
We will see that this condition on the combinatorics of $T$ and the values of $\Phi_{T}$ ensures that there are no local obstructions to the existence of a circle packing for $(T, \Phi_{T})$ whose branch set is $\beta$, and in fact is enough to ensure that there are no global ones. 
\begin{Definition}[\textsc{polynomial branch structure}]
Let $K$ be a simplicial triangulation of $\mathbb{S}^{2}$ with edge function $\Phi : V(K) \to [0,\pi/2]$. A collection
\begin{equation*}
\beta = \{(v_{\infty},\mathfrak{o}_{\infty}), (v_{1}, \mathfrak{o}_{1}), \dots (v_{\ell}, \mathfrak{o}_{\ell})\}  
\end{equation*}
is a \textit{polynomial branch structure} for $(K, \Phi)$ if the following conditions prevail.
\begin{enumerate}
    \item $\mathfrak{o}_{\infty} = \mathfrak{o}_{1} + \cdots + \mathfrak{o}_{\ell}$.
    \item The vertices $v_{1}, \dots, v_{\ell}$ are all interior vertices of the complex $T = K \setminus\text{Int}[\text{st}(v_{\infty})]$.
    \item $\beta_{T} = \{ (v_{1}, \mathfrak{o}_{1}), \dots (v_{\ell}, \mathfrak{o}_{\ell})\}$ is a branch structure for $(T,\Phi_{T})$.
    \item No $\Phi$-edge labeled subgraph of the type given in~\Cref{Figure:Forbidden} occurs in $K$ where $v$ is one of the branch vertices $v_{1}, \dots ,v_{\ell}$.
\end{enumerate}
\end{Definition}
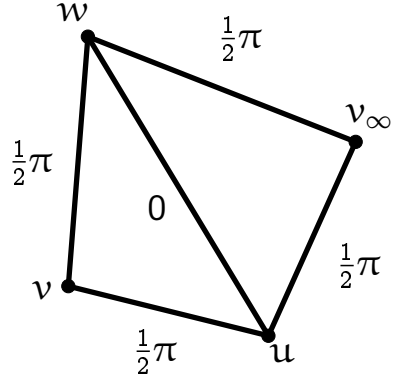
\begin{figure}

\begin{tikzpicture}[line cap=round,line join=round,>=triangle 45,x=1.0cm,y=1.0cm]
\clip(-0.5,-0.5) rectangle (5.3,5.5);
\draw [line width=2.pt] (0.6,1.28)-- (3.26,0.62);
\draw [line width=2.pt] (3.26,0.62)-- (4.42,3.2);
\draw [line width=2.pt] (4.42,3.2)-- (0.86,4.6);
\draw [line width=2.pt] (0.86,4.6)-- (0.6,1.28);
\draw [line width=2.pt] (0.86,4.6)-- (3.26,0.62);
\begin{Large}
\draw [fill=black] (0.6,1.28) circle (2.5pt);
\draw[color=black] (0.26,1.23) node {$v$};
\draw [fill=black] (3.26,0.62) circle (2.5pt);
\draw[color=black] (3.44,0.39) node {$u$};
\draw [fill=black] (4.42,3.2) circle (2.5pt);
\draw[color=black] (4.61,3.54) node {$v_\infty$};
\draw [fill=black] (0.86,4.6) circle (2.5pt);
\draw[color=black] (0.66,4.93) node {$w$};
\draw[color=black] (1.75,0.41) node {$\frac{1}{2}\pi$}; % bottom edge
\draw[color=black] (4.46,1.55) node {$\frac{1}{2}\pi$}; % right edge
\draw[color=black] (2.88,4.57) node {$\frac{1}{2}\pi$}; % top edge
\draw[color=black] (0.1,2.93) node {$\frac{1}{2}\pi$}; % left edge
\draw[color=black] (1.78,2.33) node {0};
\end{Large}
\end{tikzpicture}

\caption{A forbidden edge-labeled subgraph for a polynomial branch structure.}
\label{Figure:Forbidden}
\end{figure}
A few comments concerning this definition are in order. Item (1) says that there is an even amount of branching and half of it occurs at vertex $v_{\infty}$; item (2) says that no branch vertex from the list $v_{1}, \dots, v_{\ell}$ is adjacent to the vertex $v_{\infty}$; item (3) in particular says that Inequality~\ref{EQ:BranchCondition} holds for $(T, \Phi_{T})$; item (4) is a technical condition that avoids impossible configurations.

\begin{proof}[Discussion of proof.]
How do we put all of this together to prove the Polynomially Branched KAT Theorem? Letting $\beta = \{(v_{\infty},\mathfrak{o}_{\infty}), (v_{1}, \mathfrak{o}_{1}), \dots (v_{\ell}, \mathfrak{o}_{\ell})\}$ be a polynomial branch structure for $(K, \Phi)$, we remove the vertex $v_{\infty}$ and work with hyperbolic polyhedral metrics on the disk triangulation $T$ as in the proof of the Koebe Circle Packing Theorem. The idea is the same as there in that we want to use vertex labels on $T$ to describe hyperbolic triangles that then are identified with faces to form a singular hyperbolic surface, and then vary the labels to meet angle targets at the vertices. There are three new difficulties that appear. 
\begin{enumerate}
    \item[(i)] Target overlap angles are given by $\Phi_{T}$ for adjacent circles rather than tangencies. 
    \item[(ii)] Rather that $2\pi$, the target angle sums at branch vertices are $2\pi m_{i}$ for integers $m_{i} = \mathfrak{o}_{i} + 1 \geq 2$. 
    \item[(iii)] As the boundary $\partial \mathbb{D}$ ultimately will serve as the circle corresponding to $v_{\infty}$ in the desired circle packing, the overlaps of the boundary circles of the packing for $T$ must intersect the unit circle at the angles demanded by $\Phi$.
\end{enumerate}
Now items (i) and (ii) are really no problem as superpacking labels can be described that allow for prescribed overlap angles for adjacent circles and target angles prescribed by the branch structure. The real difficulty is item (iii). If we use radius labels, the best we can do is, as in the proof of the Koebe Theorem, get boundary circles that meet the unit circle at single points with intersection angle zero. The hint for resolving this difficulty is found in thinking a bit more about the role of circles in hyperbolic geometry, and in particular in the Poincar\'e disk model where $\mathbb{H}^{2}$ is identified with the unit disk $\mathbb{D}$, and the ideal boundary of $\mathbb{H}^{2}$ is identified with $\partial \mathbb{D} = \mathbb{S}^{1}$. Euclidean circles that meet the Poincar\'e disk $\mathbb{D}$ not only serve as hyperbolic circles, but also as horocycles and hypercycles. Those that lie entirely within $\mathbb{D}$ are hyperbolic circles, those internally tangent to the ideal boundary $\mathbb{S}^{1}$ are horocycles, and those that meet the boundary in two points $a$ and $b$ are hypercycles whose points in $\mathbb{D}$ lie equidistant to the hyperbolic line with ideal endpoints $a$ and $b$. This latter case includes the hyperbolic geodesic lines. What proves fruitful here is the fact that, when oriented, these Euclidean circles and circular arcs are precisely the curves of constant geodetic curvature in the hyperbolic plane. This is implied immediately by the fact that these are the flow lines of $1$-parameter groups of hyperbolic isometries, the hyperbolic circles the flow lines of elliptic flows, horocycles of parabolic flows, and hypercycles of hyperbolic flows. 

\label{Page:ConstantCurve}
Here are the salient facts about the geodetic curvature $\kappa$ of an arc of an oriented Euclidean circle that lies in the Poincar\'e disk $\mathbb{D}$. Call an arc $c= C\cap \mathbb{D}$, where $C$ is a Euclidean circle that meets $\mathbb{D}$, a \textit{cycle} with \textit{parent circle} $C$. There is a normalized setting in which the curvature can be read off easily. Apply a conformal automorphism of the disk so that $c$ passes through the origin and its parent circle $C$ is centered on the positive real axis. Orient $c$ counterclockwise and let $t$, $0 < t \leq \infty$ denote the point of intersection of $C$ with the interval $(0, \infty]$. Then the curvature satisfies $\kappa = \kappa (c) = 1/t$. In terms of intrinsic parameters, for counterclockwise-oriented hypercycles when $t > 1$, $\kappa(c) = \cos \alpha$ where $\alpha$ is the acute angle of intersection of $C$ with the unit circle. This includes the case of a hyperbolic geodesic where $\alpha = \pi/2$ and $\kappa = 0$. Assuming still the counterclockwise orientation, when $t=1$, $c$ is a horocycle with $\kappa(c) = 1$, and when $t<1$, $c = C$ is a hyperbolic circle of some hyperbolic radius $r$ with $\kappa(c) = \coth r$.

For our purposes it is quite fortuitous that monotone curvature parameters for cycles can be used as vertex labels on $T$ in place of radii labels to encode a singular hyperbolic metric on a disk that $T$ triangulates. The curvature is inversely related to the radii, but the really important feature is that, unlike radii labels, the curvature label can be used to identify faces of $T$, not only with hyperbolic triangles with both finite and ideal vertices, but also triangles with ``hyperideal vertices.''\footnote{When the Klein disk is used as the model for the hyperbolic plane  these are in fact Euclidean triangles that meet the disk, but whose vertices may lie within the disk, on the ideal boundary, or outside the closed disk. The hyperideal vertices are the latter ones.} This means that when curvatures $\kappa_{1}$, $\kappa_{2}$, and $\kappa_{3}$ label the vertices of the face $f$ and values $\Phi_{T}(e_{i})$ for $i=1,2,3$ label the opposite edges, the face $f$ may be identified with the region of the hyperbolic plane determined by cycles of curvatures $\kappa_{1}$, $\kappa_{2}$, and $\kappa_{3}$ overlapping with angles $\Phi_{T}(e_{i})$ for $i=1,2,3$. This accomplishes two things. First, the overlaps of cycles are given by the edge function $\Phi_{T}$. Second, and very importantly, if the vertex $w$ of $f$ is a boundary vertex and the value $\kappa = \cos \Phi_{T} (w v_{\infty})$ is used for the curvature, then the boundary cycle corresponding to $w$ overlaps with the unit circle by an angle of $\Phi_{T} (w v_{\infty})$. 

The important point is that the set $\mathfrak{K}$ of curvature labels, ones whose boundary values are given by $g(w) = \cos \Phi(wv_{\infty})$ for the boundary vertex $w$, and that produce superpackings where the angle sums at interior vertices are no more than $2\pi$ at non-branch points and no more than $2\pi m_{i}$ at branch point $v_{i}$, may be varied to obtain a $\beta$-packing label, this time as $\sup \mathfrak{K}$, the supremum instead of the infimum since curvatures are inversely related to radii. Of course by \textit{$\beta$-packing label} I mean that the angle sum at any interior vertex that is not a branch vertex is $2\pi$, and at $v_{i}$ is $2\pi m_{i}$. The argument is akin to that of the proof of the Koebe Circle Packing Theorem, but, though still elementary, is much more intricate and involved. The full detailed proof appears in ~\cite{BS96} where the key proposition, stated below, generalizes the Discrete Boundary Value Theorem of Beardon and Stephenson. Setting up this result with appropriate definitions and analysis of hyperideal hyperbolic triangles, as well as the proof itself, takes up most of the content of the paper.

\begin{Proposition}[Bowers and Stephenson~\cite{BS96}]\label{Theorem:GenBoundary}
	Let $g$ be a proper boundary label for $T$ and $\beta$ a branch structure for $(T, \Phi_{T})$. Then there exists a unique $\beta$-packing label $\mathfrak{k}$ for $(T, \Phi_{T})$ such that $\mathfrak{k} (w) = g(w)$ for every boundary vertex of $T$.
\end{Proposition}

This then is used to complete the proof of the Polynomially Branched KAT Theorem by using the circle packing produced by Proposition~\ref{Theorem:GenBoundary}, augmented by the unit circle corresponding to the removed vertex $v_{\infty}$ to define $\mathcal{C}$. Much of this becomes routine at this point, except that one still must confirm that half the branching occurs at $v_{\infty}$. This turns out to be nontrivial. Again the details are rather involved and can be found in~\cite{BS96}.
\end{proof}

\subsection{Cusps and cone type singularities.}\label{Section:Singularities} In this section I offer a generalization of KAT II where prescribed target angle sums at vertices are assigned, and  necessary and sufficient conditions are sought to guarantee existence of such packings. This is the discrete version of the classical Schwarz-Picard problem of the existence of hyperbolic metrics on Riemann surfaces with prescribed cone type singularities. For simplicity I am going to restrict to the tangency case where $\Phi$ is identically zero. 

To set up the problem, let $K$ be a simplicial triangulation of a compact surface, possibly with boundary, with $\mathtt{F}$ faces, $\mathtt{E}$ edges, and $\mathtt{V}$ vertices. The vertex set $V(K)$ is partitioned into three sets: two disjoint subsets of interior vertices denoted as $V_\text{{Int}}$ and $V_{\text{cusps}}$, and the set $V_{\text{Bd}}$ of boundary vertices, with respective cardinalities $\mathtt{V}_{\text{Int}}$, $\mathtt{V}_{\text{cusps}}$, and $\mathtt{V}_{\text{Bd}}$. Elements of $V_{\text{Int}}$ are called \textit{interior vertices} and of $V_{\text{cusps}}$ are called \textit{cusp vertices}. Two functions are given, the first $f: V_{\text{Bd}} \to (0, \infty]$ giving target radii for the boundary vertices and the second $\theta: V_{\text{Int}} \to (0, \infty)$ giving target angle sums at interior vertices. The target angle sums at the \textit{cusp vertices} in $V_{\text{cusps}}$ are zero. The task is to give necessary and sufficient conditions on $K$ to guarantee the existence of a packing label $r: V(K) \to (0, \infty]$ for this data such that $r = f$ on $V_{\text{Bd}}$, $r = \infty$ on $V_{\text{cusps}}$, and $\theta_{r} (v) = \theta (v)$ for every interior vertex $v \in V_{\text{Int}}$.

To describe a solution to this problem, for any set $V$ of vertices, let $\mathtt{F}_{V}$ denote the number of faces of $K$ that meet $V$, and let $\theta(V) = \sum_{v\in V} \theta (v)$ denote the total angle sum of the vertices of $V$. Let
\begin{equation*}
\mathfrak{R} = \{ r: V(K) \to (0, \infty] : r=f \text{ on }  V_{\mathrm{Bd}}, \,  r= \infty \text{ on } V_{\text{cusps}}, \, \theta_{r}(v) \leq \theta(v) \text{ for all } v\in V_{\text{int}}\}.
\end{equation*}
This describes the set of \textit{superpacking labels} for the data $\theta$ with boundary values given by $f$ and cusp set $V_{\text{cusps}}$. A \textit{packing label} for this data is a superpacking label where, in addition, the target angle sums given by $\theta$ are met, so that $\theta_{r}(v) = \theta(v)$ for all $v\in V_{\mathrm{Int}}$. For any superpacking label $r$ and vertex set $V$, let $\theta_{r} (V) = \sum_{v\in V} \theta_{r}(v)$. The next theorem gives necessary and sufficient conditions for a solution to the discrete Schwarz-Picard boundary value problem. The proof is a generalization of the proof presented herein for the Koebe Circle Packing Theorem. There the important invariant is $\mathtt{F} - 2 \mathtt{V}_{\text{Int}}$. In the borderless case of Theorem~\ref{Theorem:SurfacesPack}, the important invariant is $\mathtt{F} - 2 \mathtt{V}$. These arise from writing the hyperbolic area of the surface determined by a packing label, provided one exits, in terms of combinatorial invariants. The corresponding fact in this setting is that, for any packing label $\mathfrak{r}$ for the data $f$, $\theta$, and $V_{\text{cusps}}$, 
\begin{equation*}
	\text{hyp-area}(K(\mathfrak{r})) + \theta_{\mathfrak{r}}(V_{\text{Bd}}) = \pi \, \mathtt{F} - \theta_{\mathfrak{r}} ( V_{\text{Int}}) = \pi\, \mathtt{F} - \theta (V_{\text{Int}}).
\end{equation*}  
The right hand side of this equation is an invariant of $K$ and $\theta$ and must be positive since the left hand side is positive. Also, for every interior vertex $v$, 
\begin{equation*}
	\theta(v) = \theta_{\mathfrak{r}}(v) < \pi \,\text{deg}\, v
\end{equation*}
These give two necessary conditions for a desired label to exist, but these are not sufficient. Nonetheless, these two conditions are the extreme cases of the sufficient condition that appears as item (i) of the theorem. 

\begin{DSPBVT}[Bowers~\cite{pB93}]
 The following are equivalent.
 \begin{enumerate}
 \item[(i)] For every edge-path connected set $V \subset V_{\mathrm{Int}}$ of interior vertices, the invariant $\pi \,\mathtt{F}_{V} - \theta(V)$ is positive.
 \item[(ii)] The function $\mathfrak{r} = \inf \mathfrak{R}$ does not take a zero value at any vertex.
 \item[(iii)] The function $\mathfrak{r} = \inf \mathfrak{R}$ is the unique packing label for $K$ with data $f$, $\theta$, and $V_{\mathrm{cusps}}$.
 \item[(iv)] There exist a packing label for $K$ for the data $f$, $\theta$, and $V_{\mathrm{cusps}}$.
 \end{enumerate}
\end{DSPBVT}

A word of caution is in order. Though this does solve the discrete Schwarz-Picard problem, the combinatorial condition of item (i), that $\pi \,\mathtt{F}_{V} - \theta(V) > 0$ \textit{for every} path connected subset $V$ of interior vertices, is a very difficult condition to check once the size of $K$ becomes in any way substantial. This pure mathematician has learnt to appreciate the difficulties our computational geometer cousins face when trying to make the elegant output of our theorems practical tools for performing geometric computations. This difficulty often is unrecognized or left unacknowledged by my pure mathematician siblings.

\section{Infinite Packings of Non-Compact Surfaces.}\label{Section:InfinitePackings}\index{circle packing!infinite} I now turn our attention to infinite packings of non-compact surfaces. Here new and interesting phenomena arise, fraught with their own peculiar difficulties. To keep the conversation manageable, I am restricting attention to tangency circle packings of simply connected domains and will concentrate on one very interesting problem that arises in this setting---the \textit{type problem}---and one great success in attacking the Koebe Uniformization Conjecture.

\subsection{The Discrete Uniformization Theorem.} Does every simplicial triangulation $K$ of every topological surface $S$, compact or not, admit a circle packing in some geometric structure on $S$? By passing to the universal covering surface $\widetilde{S}$ and lifting the triangulation to a triangulation $\widetilde{K}$ of $\widetilde{S}$, the question may be approached by asking whether any $G$-invariant simplicial triangulation of a simply connected surface admits a $G$-invariant circle packing in some geometric structure, where $G$ is a group of symmetries of the complex. There are only two simply connected surfaces up to homeomorphism, the sphere and the plane. The former case is addressed by the Koebe Circle Packing Theorem. In this section I will address the latter case.

Let $\mathcal{T}$ be a \textit{plane triangulation graph}, by which I mean that $\mathcal{T}$ is the $1$-skeleton of a simplicial triangulation $K$ of the topological plane. There are precisely two inequivalent conformal structures on the plane, the one conformally equivalent to the complex plane $\mathbb{C}$ and the other to the open unit disk $\mathbb{D}$. There are precisely two complete metrics of constant curvature up to scaling on the plane, the one isometric to Euclidean $2$-space $\mathbb{E}^{2}$ and of constant zero curvature, the other isometric to the hyperbolic plane $\mathbb{H}^{2}$ and of constant negative curvature. Fortunately, the conformal and the geometric structures mesh nicely in that the complex plane $\mathbb{C}$ is a conformal model of plane Euclidean geometry via its standard Euclidean metric $ds_{\mathbb{C}}= |dz|$, and the disk $\mathbb{D}$ is a conformal model of plane hyperbolic geometry via the Poincar\'e metric $ds_{\mathbb{D}} = 2|dz|/(1-|z|^{2})$. Metric circles in these two geometries are precisely the Euclidean circles contained in their point sets, so circle packings in these geometric surfaces can be identified with Euclidean circle packings of $\mathbb{C}$ and $\mathbb{D}$.  I will use $\mathbb{G}$\footnote{$\mathbb{G}$ means $\mathbb{G}$eometry.} to mean either $\mathbb{C}$ or $\mathbb{D}$ with the intrinsic Euclidean or hyperbolic geometry determined by either $ds_{\mathbb{C}}$ or $ds_{\mathbb{D}}$ when referring to geometric quantities like geodesics and angles, etc. Here is the foundational result in this setting. 

\begin{DUT}[Beardon and Stephenson~\cite{BSt90}, He and Schramm~\cite{HS93}]
	Every plane triangulation graph $\mathcal{T}$ can be realized as the contacts graph of a univalent circle packing $\mathcal{T}(\mathcal{C})$ that fills exactly one of the complex plane $\mathbb{C}$ or the disk $\mathbb{D}$. The packing is unique up to conformal automorphisms of either $\mathbb{C}$ or $\mathbb{D}$. 
\end{DUT}

The \textit{contacts} graph of a collection is a graph with a vertex for each element of the collection and an edge between two vertices if an only if the corresponding elements meet. The \textit{carrier} of the circle packing $\mathcal{C}$ in the geometry $\mathbb{G}$ is the union of the geodesic triangles formed by connecting centers of triples of mutually adjacent circles with geodesic segments, and $\mathcal{C}$ \textit{fills} $\mathbb{G}$ whenever its carrier is all of $\mathbb{G}$. When $\mathcal{C}$ is univalent and fills $\mathbb{G}$, $\mathcal{C}$ is said to be a \textit{maximal packing} for $\mathcal{T}$ or $K$, and $K$ may be realized as a geodesic triangulation of $\mathbb{G}$ whose vertices are the centers of the circles of $\mathcal{C}$ with geodesic edges connecting adjacent centers.

Once this theorem is in place, the whole of the theory of tangency circle packings on non-compact surfaces comes into play. As already indicated, in a thoroughly classical way packing questions on surfaces can be transferred to questions of packings on simply connected surfaces, this by passing to covering spaces acted upon by groups of deck transformations. Any combinatorial symmetries of the complex $K$ are realized as automorphic symmetries of $\mathbb{G}$, this from the uniqueness of the Discrete Uniformization Theorem, and this offers an alternate proof of Theorem~\ref{Theorem:SurfacesPack}, and an extension of that theorem to triangulations of arbitrary, non-compact surfaces.

Beardon and Stephenson~\cite{BSt90} proved the Discrete Uniformization Theorem when $\mathcal{T}$ has \textit{bounded degree}, a global bound on the degrees of all the vertices of $\mathcal{T}$. In this foundational paper as well as in their subsequent one~\cite{BSt91a}, Beardon and Stephenson laid out a beautiful theory of circle packings on arbitrary surfaces, gave a blueprint for developing a theory of discrete analytic functions, and articulated one of the most interesting problems in the discipline, that of the \textit{circle packing type problem} for non-compact surfaces, this latter the subject of the section following. The bounded degree assumption was needed both to verify that the packing fills $\mathbb{G}$ and for the uniqueness, and He and Schramm~\cite{HS93} removed the bounded degree hypothesis and proved the general case where there is no global bound on the degrees of vertices. Earlier, Schramm~\cite{oS91} had proved a very general rigidity theorem for infinite packings of planar domains whose complementary domains are a countable collection of points, and He and Schramm~\cite{HS93} extended this to general countably connected domains. 

\begin{proof}[Discussion of Proof.]The full proof is scattered throughout several articles published in the nineteen-nineties. In what constitutes a significant service to the discipline, Ken Stephenson has laid out a complete proof in roughly fifty pages of his wonderful text \textit{Introduction to Circle Packing}~\cite{kS05}. I have not the space here to do justice to the argument, but I will make some comments.

Beardon and Stephenson's proof of existence relies on the Maximal Disk Packing Theorem and uses a diagonal argument on a sequence of finite subcomplexes of $K$ that exhausts $K$. It does not depend on any bounded degree assumption and is quite straightforward. The proof of existence goes like this. Write $K = \cup_{i=1}^{\infty} K_{i}$ as a nested, increasing union of finite subcomplexes $K_{i}$, each a simplicial triangulation of a closed disk. Apply the Maximal Disk Packing Theorem to obtain a sequence $\mathcal{C}_{i}$ of univalent, maximal circle packings for the complexes $K_{i}$ in the unit disk $\mathbb{D}$ realized as the Poincar\'e disk model of hyperbolic geometry. Fix a base vertex $v_{0}$ of $K_{1}$ and let $C_{i}$ be the circle of $\mathcal{C}_{i}$ that corresponds to $v_{0}$. By applying an automorphism of the disk if needed, assume that $C_{i}$ is centered at the origin and of hyperbolic radius $\mathfrak{r}_{i}(v_{0})$. Now the Discrete Schwarz-Pick Lemma implies that the sequence $\mathfrak{r}_{i}(v_{0})$ of hyperbolic radii is non-increasing, hence has a limit, say $\mathfrak{r}(v_{0}) \geq 0$, as $i\to \infty$. There are two cases.
	\begin{enumerate}
	\item[(I)] The limit radius $\mathfrak{r}(v_{0}) \neq 0$;
	\item[(II)] The limit radius $\mathfrak{r}(v_{0}) = 0$.
	\end{enumerate}
The first claim is that if $v$ is any other vertex of $K$ whose corresponding circle of $\mathcal{C}_{i}$, for large enough $i$, has hyperbolic radius $\mathfrak{r}_{i}(v)$, then $\lim_{i\to \infty}\mathfrak{r}_{i}(v)$ is not zero when case (I) occurs and is equal to zero when case (II) occurs. This means that the limit radius function $\mathfrak{r}: V(K) \to [0, \infty)$ never takes a zero value in case (I) and is identically zero in case (II). The proof of this claim uses the \textit{Ring Lemma}\label{RingLemma} of Burt Rodin and Dennis Sullivan that was crucial in~\cite{RS87} in their confirmation of Thurston's outlined proof of the \textit{Discrete Riemann Mapping Theorem} presented in his 1985 Purdue lecture; see Section~\ref{Section:DRMT}. The Ring Lemma guarantees the existence of a sequence of positive constants $\mathfrak{c}_{d}$ such that, when $d\geq 3$ disks form a cycle of sequentially tangent disks all tangent to a central disk of Euclidean radius $R$, and the disks have pairwise disjoint interiors, then the smallest disk has Euclidean radius $\geq \mathfrak{c}_{d}R$. The Ring Lemma is applied as follows. Let $v_{0} \, v_{1} \, \cdots \, v_{n} = v$ be a path of vertices in $K$ from $v_{0}$ to $v$ and choose $N$ so large that this path of vertices is contained in the interior of $K_{i}$, for all $i \geq N$. The Ring Lemma applied sequentially to the chain of pairwise tangent circles in $\mathcal{C}_{i}$ corresponding to the path $v_{0} \, v_{1} \, \cdots \, v_{n} = v$ implies that there is a positive constant $\mathfrak{c}$ such that $R_{i}(v) \geq \mathfrak{c}\, R_{i}(v_{0})$, where $R_{i}$ is the Euclidean radius function on $\mathcal{C}_{i}$.  This holds for all $i \geq N$ and the constant $\mathfrak{c}$ is independent of $i$. As hyperbolic and Euclidean radii of circles in the disk are comparable in the small, this implies the claim.

Now order the vertex set $V(K)$ as $v_{0}, v_{1}, \dots$. In case (I), choose a subsequence $i_{j}$ so that the hyperbolic centers of the circles of the sequence $\mathcal{C}_{i_{j}}$ all corresponding to the vertex $v_{1}$ converge in the closed disk $\overline{\mathbb{D}}$ to a point $c_{1}$. An application of item (ii) of the Discrete Schwarz Pick Lemma implies that $c_{1}$ is contained in the open disk $\mathbb{D}$. Repeat to find a subsequence of $i_{j}$ for which the hyperbolic centers of the circles corresponding to $v_{2}$ converge to a point $c_{2}$ in $\mathbb{D}$. Iterating and applying a diagonal argument gives a subsequence of the sequence of circle packings $\mathcal{C}_{i}$ for which the hyperbolic centers of the circles corresponding to the vertex $v_{n}$ of $K$ converges to a point $c_{n}$ in $\mathbb{D}$ for all positive integers $n$. Centering a circle of hyperbolic radius $\mathfrak{r}(v_{n})$ at the point $c_{n}$ produces a circle packing in the Poincar\'e disk in the pattern of $K$. In case (II) when $\mathfrak{r}$ is identically zero, a diagonal argument applied to the scaled packing $\frac{1}{R_{i}}\mathcal{C}_{i}$, where $R_{i}$ is the Euclidean radius of $C_{i}$, produces a circle packing in the plane $\mathbb{C}$ in the pattern of $K$. Call the limit circle packing in either case $\mathcal{C}$.

There are three facts left to prove: first, that $\mathcal {C}$ is univalent; second, that $\mathcal{C}$ fills the disk in case (I) and the plane in case (II); third, that $\mathcal {C}$ is unique up to automorphisms. The first claim of univalence follows from the fact that each circle packing $\mathcal{C}_{i}$ is univalent and the convergent subsequence of radii and centers described above essentially describes geometric convergence of circle packings. Beardon and Stephenson's original proof of the second claim that the packing fills $\mathbb{G}$ relied critically on the bounded degree assumption. It was used to ensure that piecewise linear maps from the complexes $K_{i}$ into the geometry $\mathbb{G}$ defined using the convergent sequence of circle packings are uniformly quasiconformal so that the Carath\'eodory Kernel Theorem~\cite{cC1912} applies to ensure that the image of the limit function is the kernel of the image sets, which is the whole of $\mathbb{G}$. The third claim of uniqueness in the hyperbolic case (I) follows from the uniqueness of the limiting radius function, but in the Euclidean case (II), uniqueness uses the bounded degree assumption. Later He and Schramm removed the bounded degree assumption. Their proof of uniqueness in case (II) is particularly elegant. It is a topological proof based on the winding numbers of  mappings defined on the boundaries of corresponding intersticial regions in two circle packings for the same complex $K$, both of which fill $\mathbb{C}$. All of this is rather nicely laid out in Stephenson's \textit{Introduction to Circle Packing}~\cite{kS05}.
\end{proof}

\subsection{Types of type.}\label{Section:type} The dichotomy between hyperbolic and Euclidean behavior is evident in the Discrete Uniformization Theorem. Indeed, the combinatorial complex $K$, or its $1$-skeleton $\mathcal{T}$, determines uniquely its geometry in that the maximal circle packing $\mathcal{T}(\mathcal{C})$ fills either the disk $\mathbb{D}$ or the complex plane $\mathbb{C}$, but forbids two packings where one fills the disk and the other the plane. This leads to the next definition.
\begin{Definition}[\textsc{cp-type}]
	A simplicial triagulation $K$ of the plane, and its $1$-skeleton plane triangulation graph $\mathcal{T} = K^{(1)}$, are said to \textit{CP-parabolic} or \textit{CP-hyperbolic} when the maximal circle packing $\mathcal{T}(\mathcal{C})$ fills respectively the complex plane $\mathbb{C}$ or the disk $\mathbb{D}$. The \textit{CP-type problem} is the problem of determining whether a given complex $K$ or plane triangulation graph $\mathcal{T}$ is CP-parabolic or CP-hyperbolic. One seeks conditions or invariants on the complex $K$ or the graph $\mathcal{T}$, reasonably checked or computed, that can determine which of the two CP-types adheres. See~\Cref{Figure:CPtype}.
\end{Definition}
This is a discrete version of the classical \textit{conformal type problem}, or just \textit{type problem}\index{circle packing!type problem} for short, that of determining whether, \`a la classical Uniformization Theorem, a given non-compact simply connected Riemann surface is \textit{parabolic} and conformally equivalent to the complex plane $\mathbb{C}$, or \textit{hyperbolic} and conformally equivalent to the disk $\mathbb{D}$. 
\begin{figure}
\begin{subfigure}[b]{0.45\textwidth}
	\includegraphics[width=\textwidth]{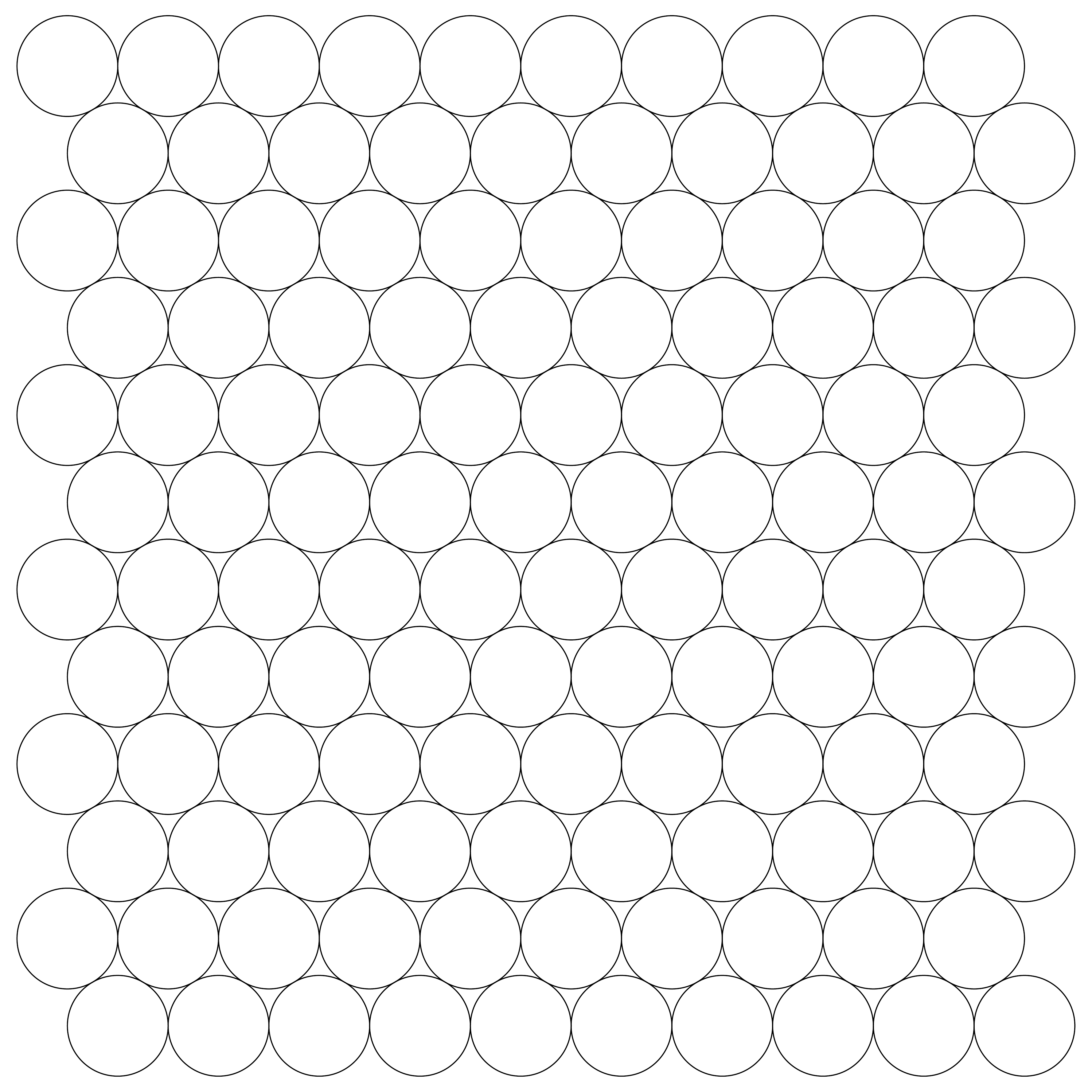}
\caption{The penny packing, the maximal circle packing for the constant $6$-degree plane triangulation graph $\mathcal{G}_{6}$, fills the plane $\mathbb{C}$. The graph $\mathcal{G}_{6}$ is parabolic.}\label{Figure:PP}
\end{subfigure}
\quad
\begin{subfigure}[b]{0.45\textwidth}
	\includegraphics[width=\textwidth]{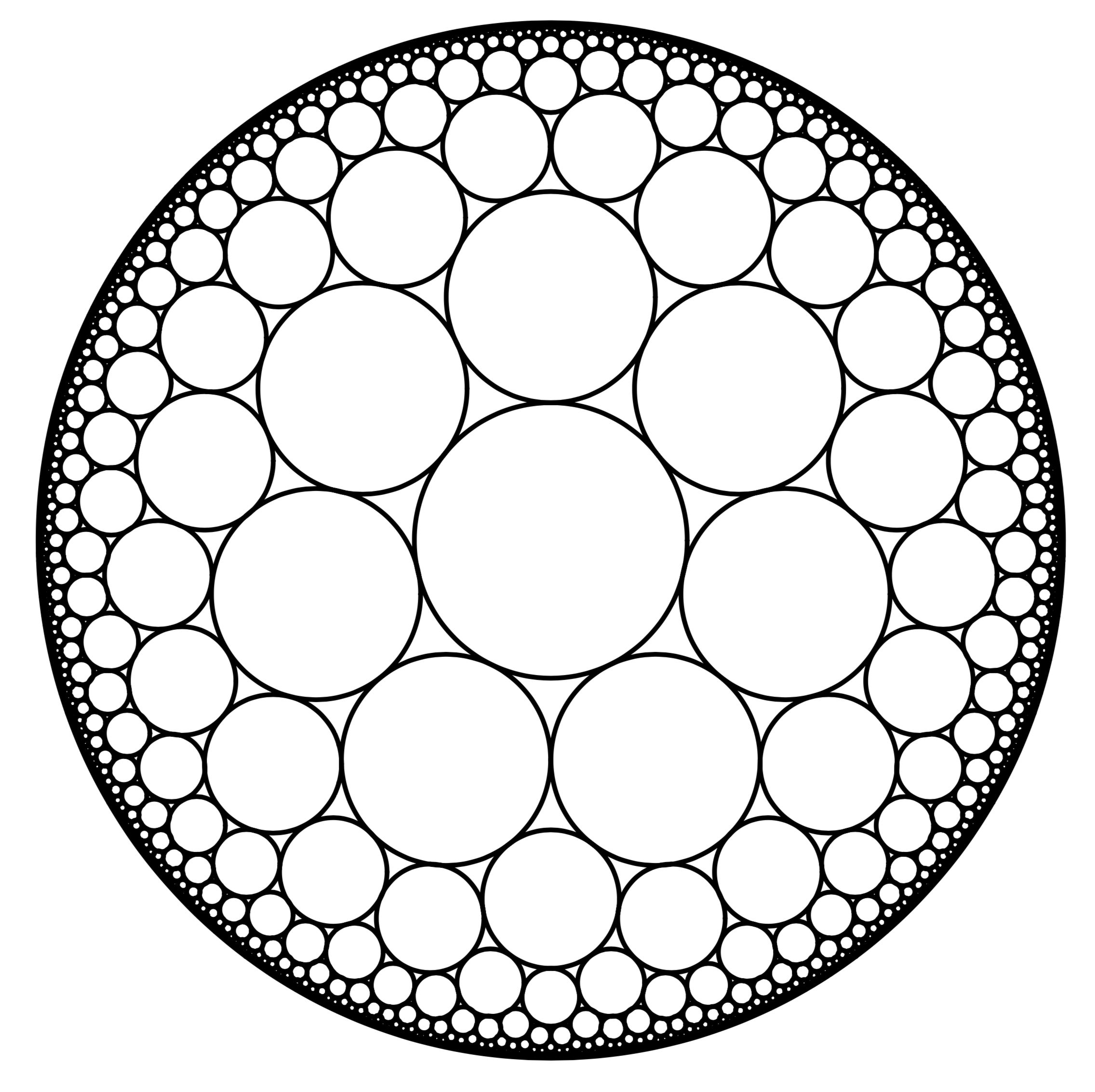}
\caption{The maximal circle packing for the constant $7$-degree plane triangulation graph $\mathcal{G}_{7}$ fills the disk $\mathbb{D}$. The graph $\mathcal{G}_{7}$ is hyperbolic, as are the graphs $\mathcal{G}_{d}$ for all $d\geq 7$.}\label{Figure:deg7}
\end{subfigure}
      \caption{The CP-type of a plane triangulation graph is determined by the corresponding maximal circle packing and whether it fills the plane or the disk.}\label{Figure:CPtype}
\end{figure}

Historically this is not the first discrete type problem. That honor probably goes to the  problem of determining the \textit{random walk type}, or \textit{RW-type} for short, of an infinite graph. My aim in this section is to review this and several other species of discrete type problems and explore their interactions in the context of plane triangulation graphs. In all I will examine six different species of discrete type that go under the abbreviations CP, RW, EL, EEL, VEL, and EQ-type.

Consider the standard simple random walk on a simple, connected, locally finite  graph $\mathcal{G}$ where the probability of walking across a particular edge $uv$ from vertex $u$ to vertex $v$ is $1/ \deg u$. The graph $\mathcal{G}$ is said to be \textit{random walk parabolic}, or \textit{RW-parabolic}, if a walker almost surely returns to a fixed base vertex, and \textit{random walk hyperbolic} or \textit{RW-hyperbolic} otherwise. More common terminology is that the graph is \textit{recurrent} when RW-parabolic and \textit{transient} when RW-hyperbolic. In a transient graph, a random walker has a positive probability for escaping to infinity whereas in a recurrent one, the escape probability vanishes and, in fact, the walker almost surely returns to every vertex infinitely often. Woess~\cite{woW00} is a fantastic reference for the classical theory of random walks on graphs and Lawler and Limic~\cite{LL10} is an up-to-date reference with many recent results.

Early on in the development of circle packing theory, Ken Stephenson made a connection between the CP- and RW-type problems. The intuition for the connection arises from the close connection in classical complex function theory between the conformal type problem and brownian motion on a Riemann surface. Stephenson~\cite{kS96} proved that the CP- and RW-types of bounded degree plane triangulation graphs always coincide. Later in~\cite{HS95a}, He and Schramm gave an example of a plane triangulation graph, necessarily of unbounded degree, that is CP-parabolic but RW-hyperbolic. There the authors focused more sharply on the distinction between these two species of type and recalled Duffin's EEL-type from~\cite{rD62} and developed Cannon's VEL-type inspired by~\cite{jC94} in articulating the distinction. 

Before continuing with the discussion of CP-type, let's review a bit of history. The story of discrete type really begins in the nineteen-twenties with P\'oyla's study~\cite{gP21} of the RW-type of the integer grid in $\mathbb{R}^{d}$ where he proved that the integer grid in $\mathbb{R}^{2}$ is RW-parabolic while the grid in higher-dimensional Euclidean spaces is RW-hyperbolic. In 1959, Nash-Williams in~\cite{N-W59} used a method of Lord Rayleigh to link the RW-type of a locally finite graph with its resistance to electric flow when each edge is thought of as a wire with a unit of electrical resistance, giving rise to EL-type. To be a bit less cryptic, when an infinite graph is thought of as an electric network with each edge representing a wire of unit resistance, the question is whether electricity will flow from a base vertex to infinity when a unit potential is applied to the base vertex and infinity is grounded. This is made a bit more precise by asking what the effective resistance is from the base vertex to infinity for the network. When the effective resistance to infinity is infinite, no current flows and the network is \textit{EL-parabolic}, and when the resistance is finite so that current does flow, the network is \textit{EL-hyperbolic}. In the beautiful 1984 Carus Mathematical Monograph~\cite{DS84} entitled \textit{Random Walks and Electric networks}, Peter Doyle and J. Laurie Snell present an accessible proof that the RW- and EL-type of an infinite graph coincide. In 1962, Duffin~\cite{rD62} gave a combinatorial invariant of a graph, the \textit{edge extremal length}, that characterizes the RW- or EL-type according to whether the edge extremal length of the set of transient edge-paths is infinite or not. 

What is the edge extremal length of a path family? It is a discrete version of the classical conformal extremal length of a path family in a Riemann surface in quasiconformal analysis. For a graph $\mathcal{G}$, let $\Gamma$ be any family of edge-paths, infinite or not. The edge extremal length is obtained by measuring the minimal length-squared of the curves in $\Gamma$ divided by the area, this maximized over all metric assignments. This is the same as the classical definition, only what changes is how the admissible metrics are assigned. Here are the details. An \textit{edge-path} in $\mathcal{G}$ is a finite or infinite sequence $\mathbf{e} = e_{1}, e_{2}, \dots$ of directed edges of $\mathcal{G}$ with the terminal vertex of $e_{i}$ equal to the initial vertex of $e_{i+1}$. An \textit{edge metric} on $\mathcal{G}$ is a function $m: E(\mathcal{G}) \to [0, \infty]$ that assigns a non-negative value to each edge, and the \textit{area of $m$} is defined as $\text{area} (m) = \sum_{e\in E(\mathcal{G})} m(e)^{2}$. An edge metric is \textit{admissible} if its area is finite and I will let $M_{E} (\mathcal{G})$ denote the collection of admissible edge metrics. The \textit{$m$-length} of the edge-path $\mathbf{e}$ is $\ell_{m}(\mathbf{e}) = \sum_{i=1} m(e_{i})$. Finally, the \textit{edge extremal length} of the family $\Gamma$ of edge-paths is
\begin{equation*}
	\text{EEL}(\Gamma) = \sup_{m\in M_{E}(\mathcal{G})} \frac{\inf_{\mathbf{e} \in \Gamma} \ell_{m}(\mathbf{e})^{2}}{\text{area}(m)}.
\end{equation*}
The notation $\text{EEL}(\mathcal{G})$ is reserved for the case where $\Gamma$ is the set of paths to infinity that start at a given base vertex $v_{0}$. These are called the \textit{transient} edge-paths in $\mathcal{G}$ based at $v_{0}$, and any such transient edge-path $\mathbf{e} \in \Gamma$ has initial vertex $v_{0}$ at its first edge $e_{1}$ and is not contained in any finite collection of edges. One says that the graph $\mathcal{G}$ is \textit{EEL-parabolic} if $\text{EEL}(\mathcal{G}) = \infty$ and \textit{EEL-hyperbolic} otherwise. It is an easy exercise to confirm that EEL-type does not depend on which base vertex is chosen. Duffin's result of \cite{rD62} already mentioned is that both the RW- and EL-type of a graph coincides with the EEL-type. This was the state of the art in discrete type in the early nineteen-nineties when Stephenson connected CP-type with RW-type for bounded degree plane triangulation graphs. 

In 1995, He and Schramm~\cite{HS95a} in a remarkable article clarified the role of the bounded degree assumption. There, after constructing a plane triangulation graph that, though CP-parabolic, is RW-hyperbolic, they applied Cannon's vertex extremal length to characterize CP-type combinatorially in the way that edge extremal length characterizes RW-type. Cannon~\cite{jC94} introduced the vertex extremal length of a discrete curve family made of shinglings and used it as a tool for assigning combinatorial moduli to ring domains in the space at infinity of a negatively curved group. He and Schramm adapted Cannon's vertex extremal length to Duffin's development of EEL-type to create VEL-type. The adjustment merely replaces edge-paths by vertex-paths and edge metrics by vertex metrics. The \textit{vertex extremal length} of a family $\Delta$ of vertex paths is 
\begin{equation}\label{EQ:VEL}
	\text{VEL}(\Delta) = \sup_{m\in M_{V}(\mathcal{G})} \frac{\inf_{\mathbf{v} \in \Delta} \ell_{m}(\mathbf{v})^{2}}{\text{area}(m)}.
\end{equation}
Here, a \textit{vertex-path} is a sequence $\mathbf{v} = v_{1}, v_{2}, \dots$ where each $v_{i}$ is incident with its successor $v_{i+1}$, and a \textit{vertex metric} is a non-negative function $m:V(\mathcal{G}) \to [0,\infty]$ with \textit{area} $\text{area} (m) = \sum_{v\in V(\mathcal{G})} m(v)^{2}$. The $m$-length of the vertex-path $\mathbf{v}$ is $\ell_{m}(\mathbf{v}) = \sum_{i=1} m(v_{i})$ and the set of \textit{admissible metrics}, the ones of finite area, is denoted as $M_{V}(\mathcal{G})$. The VEL-type of $\mathcal{G}$ now is defined analogously to EEL-type. Indeed, $\text{VEL}(\mathcal{G})$ means $\text{VEL}(\Delta)$, where $\Delta$ is the set of \textit{transient} vertex-paths based at $v_{0}$, those that meet infinitely many vertices. The graph $\mathcal{G}$ is \textit{VEL-parabolic} if $\text{VEL}(\mathcal{G}) = \infty$ and \textit{VEL-hyperbolic} otherwise, and again it is an easy exercise to confirm that VEL-type does not depend on which base vertex is chosen. This seemingly innocuous adjustment to the definition of EEL-type turns out to be precisely the tool needed to characterize CP-type.

Though, easily, the EEL- and VEL-types of a bounded degree graph coincide, they may differ for a graph of unbounded degree. The relationships between the four types---RW, EL, EEL, VEL---are summarized in the next theorem.

\begin{DTTG} Let $\mathcal{G}$ be a connected, infinite, locally finite graph. 
	\begin{enumerate}
		\item[(i)]{\em [Nash-Willliams~\cite{N-W59}, Duffin~\cite{rD62}]} The three types---RW, EL, EEL---coincide for $\mathcal{G}$.
		\item[(ii)] {\em [He-Schramm~\cite{HS95a}]} If $\mathcal{G}$ is EEL-parabolic then it is VEL-parabolic. If $\mathcal{G}$ has bounded degree and is VEL-parabolic, then it is EEL-parabolic. 
		\item[(iii)] {\em [He-Schramm~\cite{HS95a}]} There is a VEL-parabolic plane triangulation graph that is EEL-hyperbolic, necessarily of unbounded degree.
	\end{enumerate}
\end{DTTG}

For a plane triangulation graph $\mathcal{T}$, all five types---RW, EL, EEL, VEL, CP---coincide provided $\mathcal{T}$ has bounded degree. As stated above, it was Stephenson who first proved this for RW- and CP-types. He and Schramm clarified the need for the bounded degree hypothesis, and the relationship between discrete types for plane triangulation graphs is summarized next.
\begin{DTTPTG}[He-Schramm~\cite{HS95a}]
	Let $\mathcal{T}$ be a plane triangulation graph. Then $\mathcal{T}$ is CP-parabolic if and only if it is VEL-parabolic. 
\end{DTTPTG}

The proofs of these theorems are quite difficult and involved, though still elementary, and space forbids any sort of discussion of the proofs that would do justice to the subject. Suffice it to say that the interested reader can do no better than to consult the references cited in this section to fill in gaps in the desired detail of proofs.

The Discrete Type Theorem for Plane Triangulation Graphs reduces the very difficult problem of determining whether the maximal circle packing for $\mathcal{T}$ is parabolic or hyperbolic to a combinatorial computation on the graph $\mathcal{T}$. The disappointment comes when one actually tries to do the computation of $\text{VEL}(\mathcal{T})$ from Equation~\ref{EQ:VEL} for almost any given plane triangulation graph. One then finds out just how difficult it is to perform this  computation; nonetheless, this development is useful for some theoretical considerations. For example, He and Schramm use the theorem to extend Stephenson's result on RW- and CP-type. Here is an interesting result of the author that uses the computation of Equation~\ref{EQ:VEL} for the proof of item (ii) of the theorem. 

\begin{Theorem}[Bowers~\cite{pB97}]\label{Theorem:Gromov} Let $\mathcal{G}$ be a connected, infinite, locally finite graph and $\mathcal{T}$ a plane triangulation graph.
\begin{enumerate}
	\item[(i)] If $\mathcal{G}$ is Gromov negatively curved and its Gromov boundary  contains a nontrivial continuum, then $\mathcal{G}$ is RW-hyperbolic.
	\item[(ii)] If $\mathcal{T}$ is Gromov negatively curved, then $\mathcal{T}$ is CP-parabolic if and only if its Gromov boundary is a singleton; alternately, it is CP-hyperbolic if and only if its Gromov boundary is a topological circle.
\end{enumerate}	
\end{Theorem}

I refer the reader to the appendix of the article~\cite{pB97} for definitions and basic theorems on Gromov negatively curved graphs and metric spaces. To show how the computation from Equation~\ref{EQ:VEL} may proceed, I'll prove the lemma used in~\cite{pB97} to prove the first assertion of item (ii) of Theorem~\ref{Theorem:Gromov}.

\begin{Lemma}
Let $v_{0}$ be a vertex in the connected, infinite, locally finite graph $\mathcal{G}$ and let $\{V_{n}\}$ be a sequence of pairwise disjoint sets of vertices, each of which separates $v_{0}$ from infinity. Suppose there exist positive constants $C$ and $\varepsilon$ such that, for $n \geq N$,
\begin{equation*}
	\mathrm{Card} (V_{n}) \leq Cn.
\end{equation*}
Then the graph $\mathcal{G}$ is VEL-parabolic.
\end{Lemma}

\begin{proof}
Define the vertex metric $m$ by $m(v) = 1/ (n\log n)$ for any $v\in V_{n}$ when $n\geq N$, and $m(v) =0$ otherwise. Then $m$ is admissible since
\begin{equation*}
\mathrm{area}(m) = \sum_{n=N}^{\infty} \frac{\mathrm{Card}(V_{n})}{(n \log n)^{2}} \leq  \sum_{n=N}^{\infty} \frac{C}{n(\log n)^{2}} < \infty.
\end{equation*}
For any transient vertex-path $\mathbf{v}$, the $m$-length satisfies $\ell_{m} (\mathbf{v} ) \geq \sum_{n=N}^{\infty} 1/ (n \log n) = \infty$, hence every transient vertex-path has infinite $m$-length. This implies that $\textrm{VEL} (\mathcal{G}) = \infty$ and $\mathcal{G}$ is VEL-parabolic.
\end{proof}
I'll end this section with a sixth version of discrete type that is of recent interest in several settings. It arose first for me when Ken Stephenson and I constructed expansion complexes of finite subdivision rules, for the first time in~\cite{BS97} when examining the pentagonal subdivision rule of Cannon, Floyd, and Parry~\cite{CFP01}. More recently it arises in our examination of hierarchical conformal tilings~\cite{BS14a,BS14b}, and in Gill and Rohde's~\cite{GR13} examination of random planar maps. I name this version of discrete type \textit{EQ-type} with \textit{EQ} an abbreviation for \textit{equilateral}. A plane triangulation graph $\mathcal{T} = K^{(1)}$ can be used to build a piecewise equilateral surface by setting each edge to unit length and isometrically gluing unit-sided equilateral triangles along their boundaries to the boundaries of the faces of $K$. This produces a piecewise flat surface $|\mathcal{T}|_{\mathrm{eq}}$ that has a natural conformal atlas obtained as follows. Each edge $e$ of $\mathcal{T}$ indexes a chart map $\varphi_{e}$ defined on the interior of the union of the faces incident with $e$. These have been identified with unit equilateral triangles and the chart map $\varphi_{e}$ is an orientation-preserving isometry to the plane $\mathbb{C}$. Each vertex $v$ also indexes a chart map $\varphi_{v}$ defined on the metric neighborhood of $v$ in $|\mathcal{T}|_{\mathrm{eq}}$ of radius $1/2$, and uses an appropriate complex power mapping to flatten that neighborhood to a disk in the plane $\mathbb{C}$. The overlap maps are conformal homeomorphisms between the appropriate domains. The chart family $\mathcal{A} = \{\varphi_{x} : x \in V(\mathcal{T}) \cup E(\mathcal{T})\}$ forms a complex atlas making $|\mathcal{T}|_{\mathrm{eq}}$ into a non-compact simply connected Riemann surface $S(\mathcal{T})$. The type problem now is manifest. Is $S(\mathcal{T})$ conformally the plane $\mathbb{C}$ or the disk $\mathbb{D}$? In the former case, $\mathcal{T}$ and $K$ are said to be \textit{EQ-parabolic}, in the latter \textit{EQ-hyperbolic}.

Notice that the question of the EQ-type of a plane triangulation graph is the classical question of the conformal type of a simply connected Riemann surface. It bares the moniker \textit{discrete} because of how the surface is built---using discrete building blocks, the equilateral triangles, glued in a combinatorial pattern encoded in $\mathcal{T}$. The desire is for a combinatorial invariant of $\mathcal{T}$ or $K$ that will determine its EQ-type. So, what relationship exists between the discrete types already discussed and EQ-type? For plane triangulation graphs of bounded degree, easy arguments using quasiconformal mappings show that EQ-type coincides with CP-type---just map the equilateral triangle in $|\mathcal{T}|_{\mathrm{eq}}$ at face $f$ to the corresponding geodesic triangle in $\mathbb{G}$. When $\mathcal{T}$ has bounded degree, this map is uniformly quasiconformal and so the EQ-type agrees with the conformal type of $\mathbb{G}$. For unbounded degree plane triangulation graphs, it remains an open question as to whether the EQ-type coincides with, say, the EEL- or the VEL-type, or perhaps neither. I am bold enough to offer the following conjecture.

\begin{Conjecture}
For any plane triangulation graph, EQ-type coincides with VEL-type, and therefore with CP-type.
\end{Conjecture}

A great reference for various expressions of discrete type and their stability under subdivision is Bill Wood's doctoral thesis~\cite{bW06} and the subsequent article~\cite{wW10}. I now turn our attention to Koebe's original inspiration for his circle packing theorem, his interest in circle domains, uniformization, and the Kreisnormierungsproblem.

\subsection{Koebe uniformization for countably-connected domains.}\index{Koebe uniformization conjecture!for countably-connected domains}
Zheng-Xu He and Oded Schramm's work on circle packing in the late nineteen-eighties and early nineteen-nineties led them to a study of Koebe's Uniformization Conjecture. Though the discrete circle packing tools they developed and used did not directly apply to Koebe's problem, the perspective they had gained turned out to be useful. By 1992-93, they had made the greatest advance on Koebe's problem since its articulation and had proved a circle packing version that greatly generalized the Discrete Uniformization Theorem. Their work is detailed in the \textsl{Annals of Mathematics} article \textit{Fixed points, Koebe uniformization, and circle packings}. The proofs are rather intricate and so I am content to state the two main results without any indication of the proofs, leaving it to the interested reader to peruse~\cite{HS93} for details.

\begin{HSKU}[He and Schramm~\cite{HS93}, Schramm~\cite{oS95}]
Every countably connected domain in the Riemann sphere is conformally homeomorphic to a circle domain. Moreover, the circle domain is unique up to M\"obius transformations and every conformal automorphism of the circle domain is the restriction of a M\"obius transformation.
\end{HSKU}

A \textit{domain triangulation graph} is the $1$-skeleton of a simplicial triangulation of a planar domain. 

\begin{HSDU}[He and Schramm~\cite{HS93}]
Every domain triangulation graph with at most countably many ends has a univalent circle packing in the plane $\mathbb{C}$ whose carrier is a circle domain. Moreover, the circle packing is unique up to M\"obius transformations.
\end{HSDU}

He and Schramm prove a theorem that generalizes their Uniformization Theorem to \textit{generaized domains} and \textit{generalized circle domains}. This more general unifomization theorem then is used to give a quick proof of their Discrete Uniformization Theorem. 

I'll close this section by mentioning that Schramm in a 1995 paper~\cite{oS95} introduced the notion of \textit{transboundary extremal length} that generalizes the classical extremal length of curve families. Transboundary extremal length is more suited to path families in multiply connected domains that allow for the curves of the family to pass through the complementary components of the domain. Using this tool, Schramm gives a short proof of Koebe uniformization of countably connected domains and generalizes it in two ways. First, he shows that circle domains as the target of uniformization may be replaced by more general domains, namely, those where the complementary components are what he calls \textit{$\tau$-fat sets}. Second, he shows that some domains with uncountably many complementary components may be uniformized to circle domains, namely those where the complementary components are uniformly fat. This includes for example domains whose boundary components are points and $\mu$-quasicircles for a fixed constant $\mu\geq 1$.

\section{Some Theoretical Applications.}
The theoretical work in circle packing has grown up hand-in-hand with various applications. In the past score of years, the needs of computer imaging have added a practical bent to the applications with the use of the theory for everything from medical imaging to 3D-printer head guidance. This has been one of the impetuses for the development of the discipline of discrete differential geometry with discrete conformal geometry as but one of its chapters. Circle packing theory \`a la Thurston as described in this article is one flavor of this, but several groups of computational geometers and computer scientists have developed discrete conformal geometry in a great variety of ways, with new techniques designed to solve both practical and theoretical problems. The discipline has grown to a vast enterprise too large and complicated for a review of this type. Rather than attempt a thorough discussion of these applications, I'll only mention a couple of the theoretical applications. The first  stands as one of the linchpins of the discipline, and the second generalizes the first. I'll leave it for the interested reader to peruse the many resources available to learn of the state of the art today in practical applications.

\subsection{Approximating the Riemann mapping.}\label{Section:DRMT}
The event that really got circle packing launched, piquing the interest of a small group of research mathematicians from as diverse fields as complex function theory, combinatorial and computational geometry, geometric topology, and the classical theory of polyhedra, was Bill Thurston's address entitled \textit{The Finite Riemann Mapping Theorem} at Purdue University in 1985. He presented there an algorithm for computing discrete versions of the Riemann mapping of a fixed, proper, simply connected domain in the complex plane $\mathbb{C}$ to the unit disk $\mathbb{D}$, with an indication of why the discrete mappings should converge to a conformal homeomorphism of the domain onto $\mathbb{D}$. Burt Rodin and Dennis Sullivan published in~\cite{RS87} a proof of Thurston's claims in 1987, and this began a steady output of published research on circle packings that continues today. Here I review the content of Thurston's 1985 talk and explain the Rodin-Sullivan verification of Thurston's claims.

Thurston's algorithm is illustrated nicely in the graphics of ~\Cref{Figure:DCM}.
\begin{figure}
\captionsetup[subfigure]{labelformat=empty}
\begin{subfigure}[b]{0.35\textwidth}
	\includegraphics[width=\textwidth]{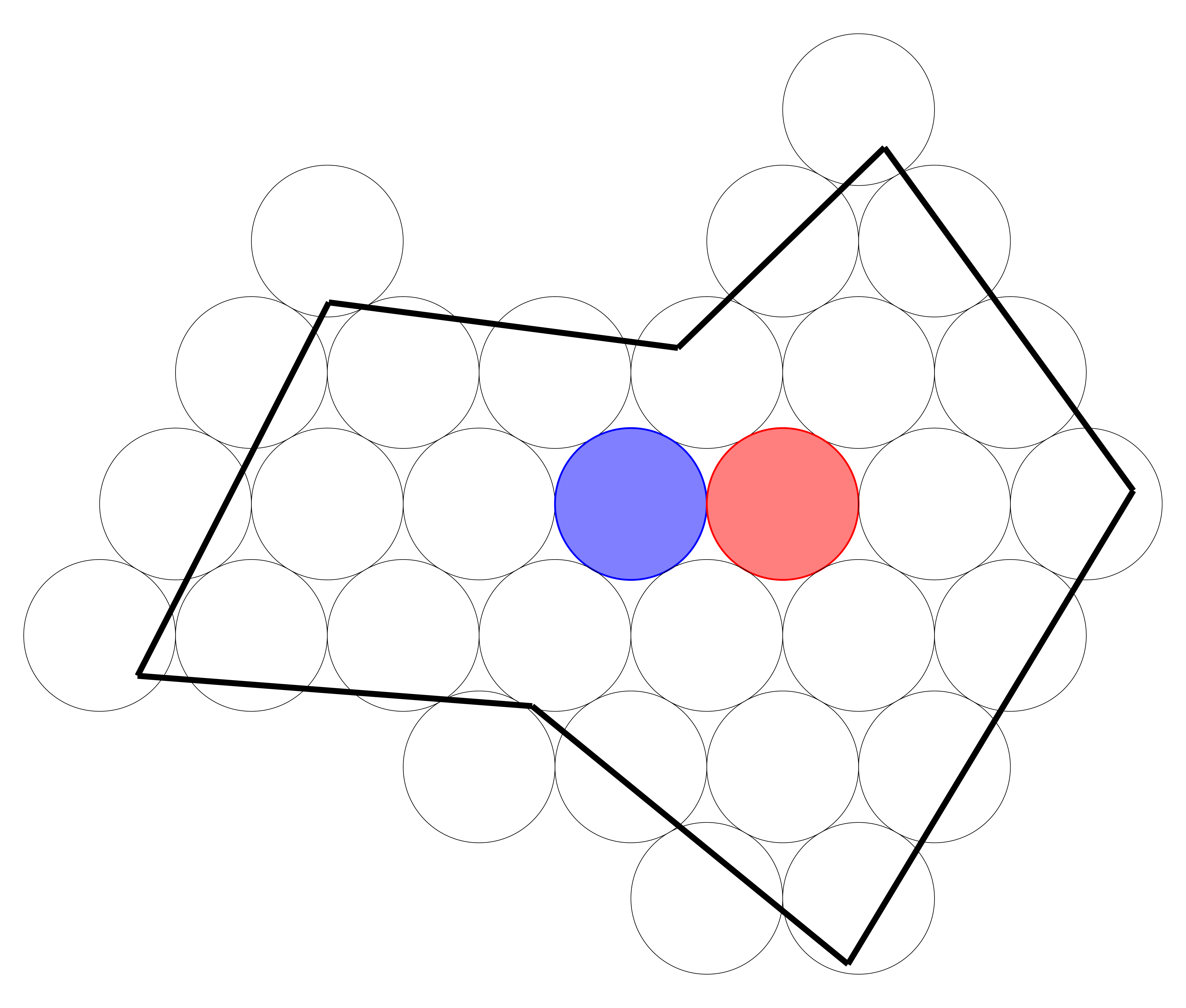}
\caption{}
\end{subfigure}
\quad
\begin{subfigure}[b]{0.35\textwidth}
	\includegraphics[width=\textwidth]{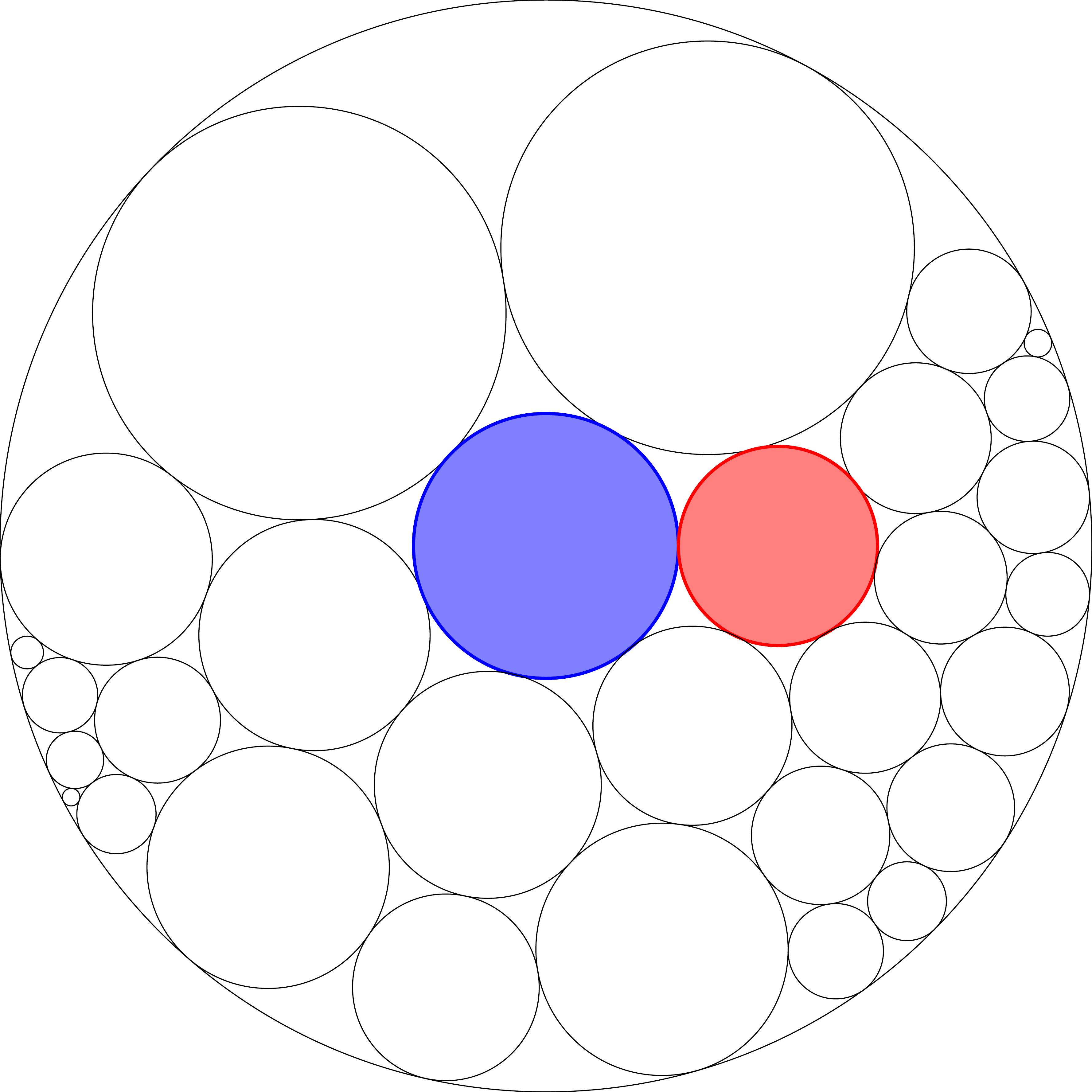}
\caption{}
\end{subfigure}

\begin{subfigure}[b]{0.35\textwidth}
	\includegraphics[width=\textwidth]{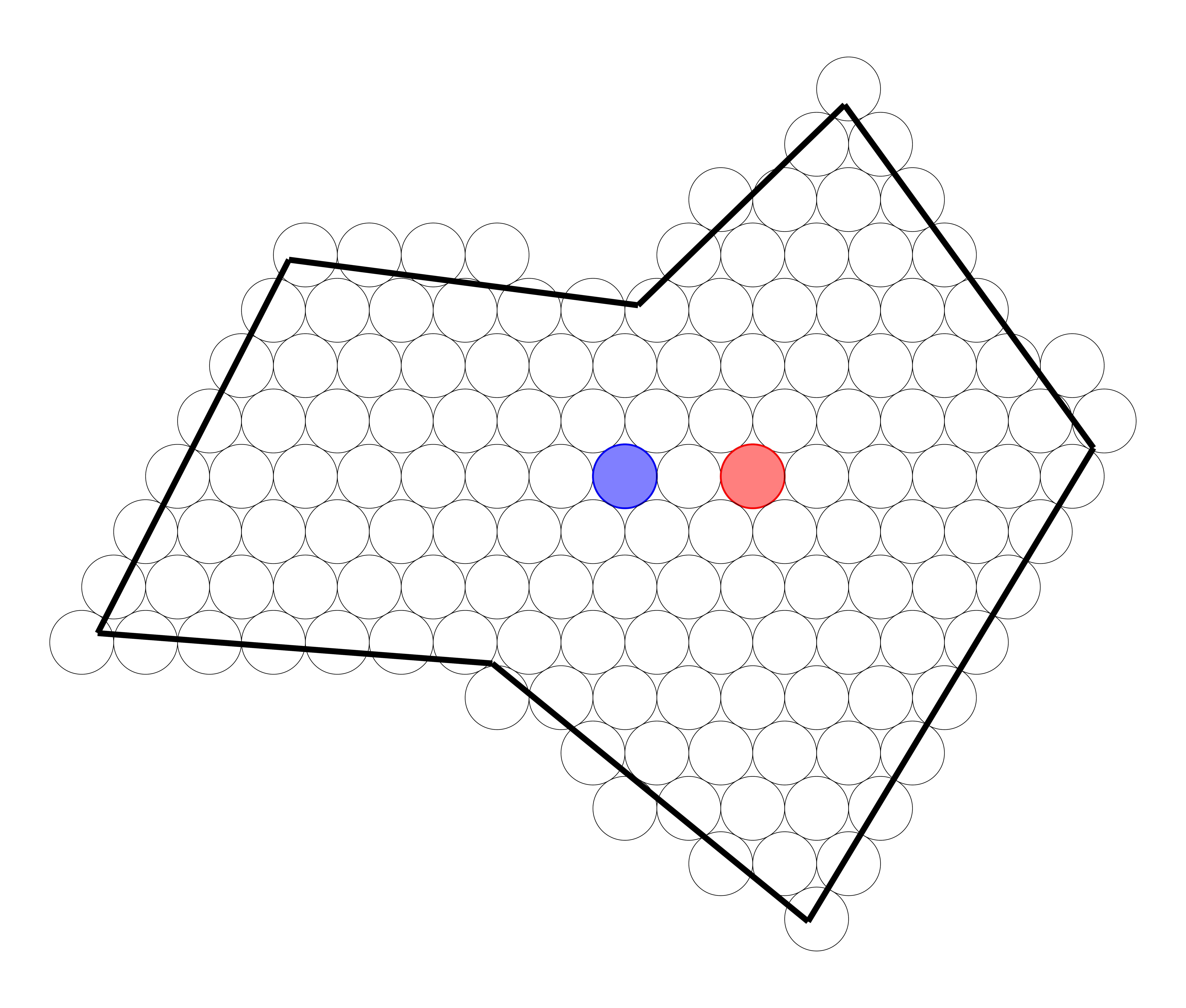}
\caption{}
\end{subfigure}
\quad
\begin{subfigure}[b]{0.35\textwidth}
	\includegraphics[width=\textwidth]{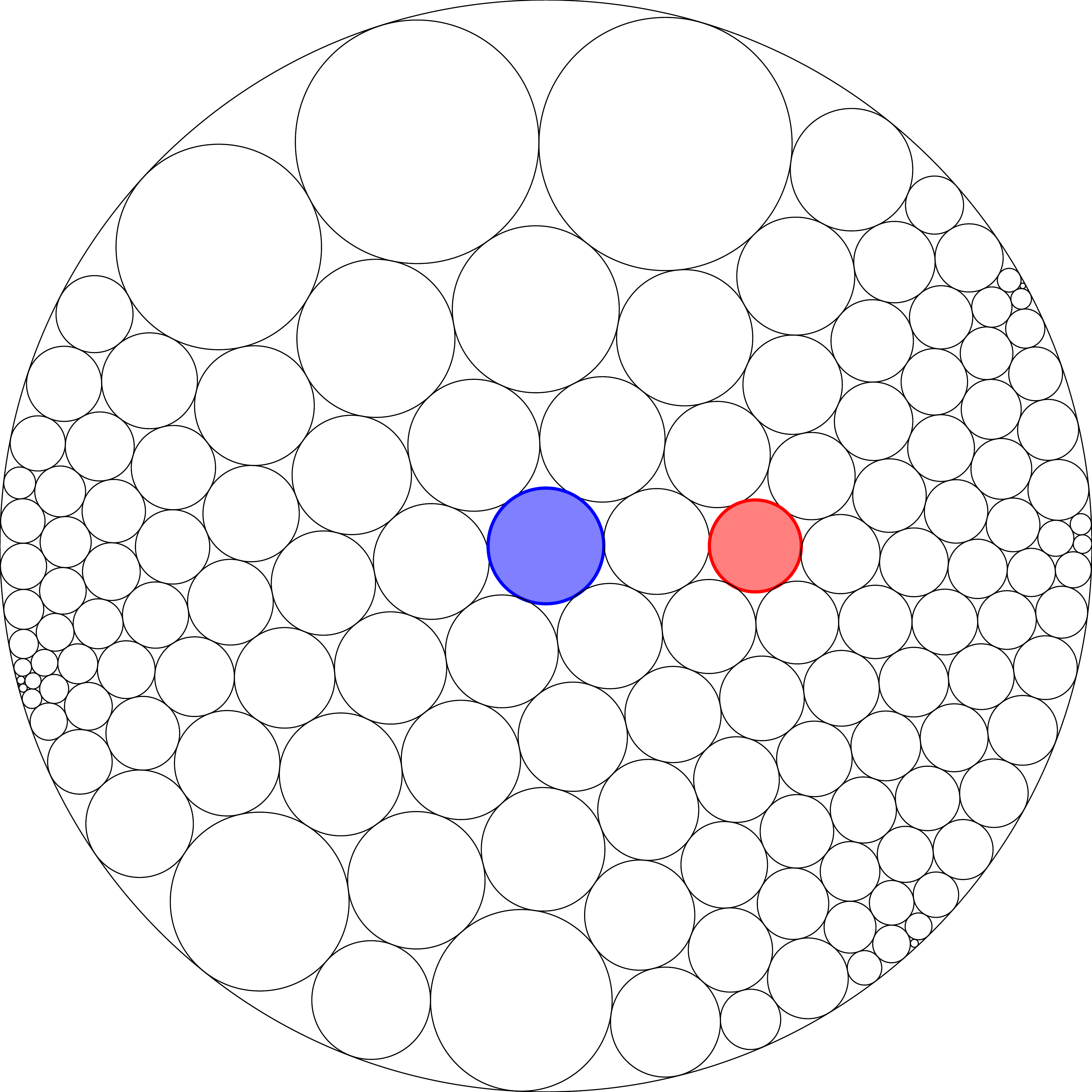}
\caption{}
\end{subfigure}
\begin{subfigure}[b]{0.35\textwidth}
	\includegraphics[width=\textwidth]{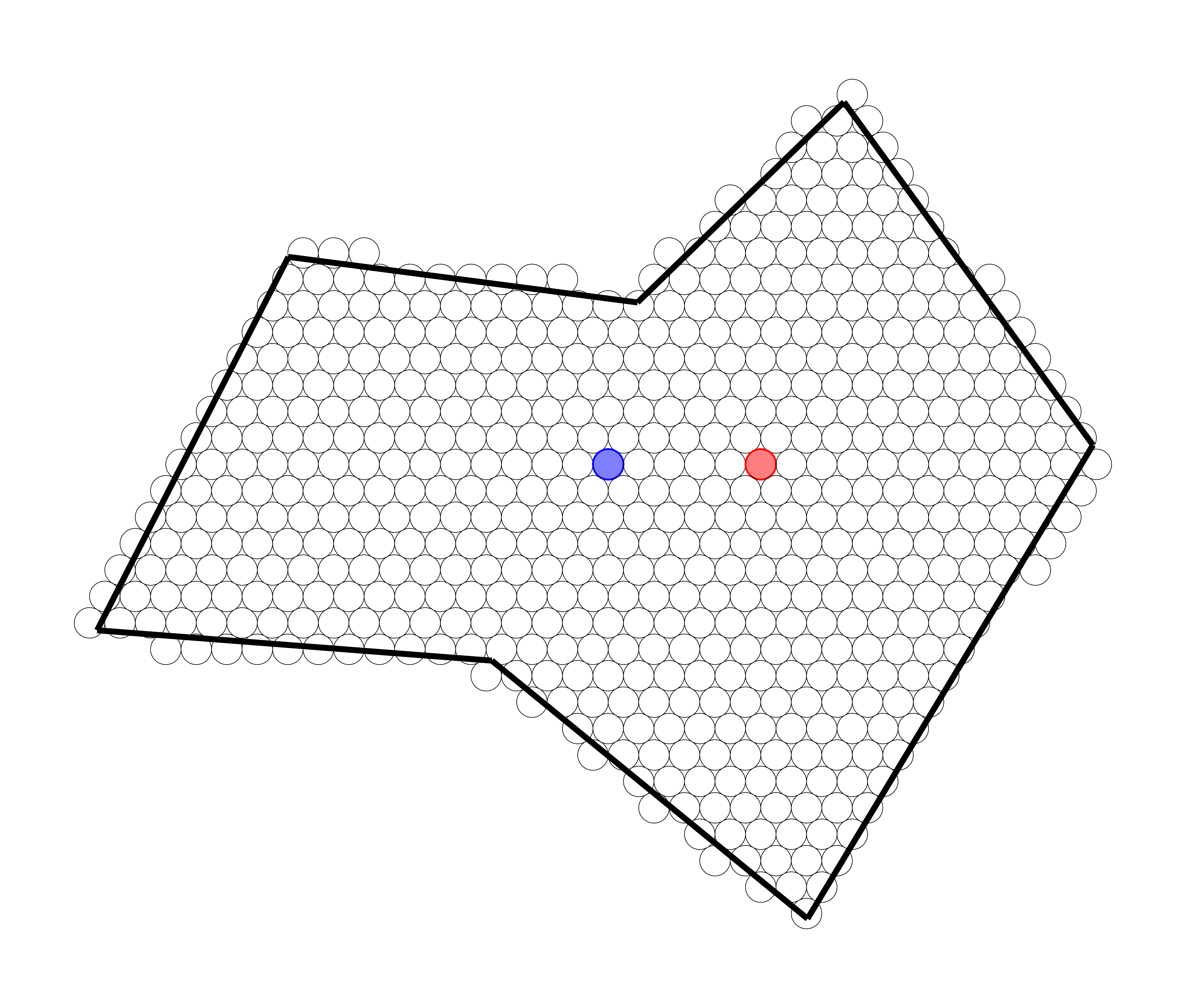}
\caption{}
\end{subfigure}
\quad
\begin{subfigure}[b]{0.35\textwidth}
	\includegraphics[width=\textwidth]{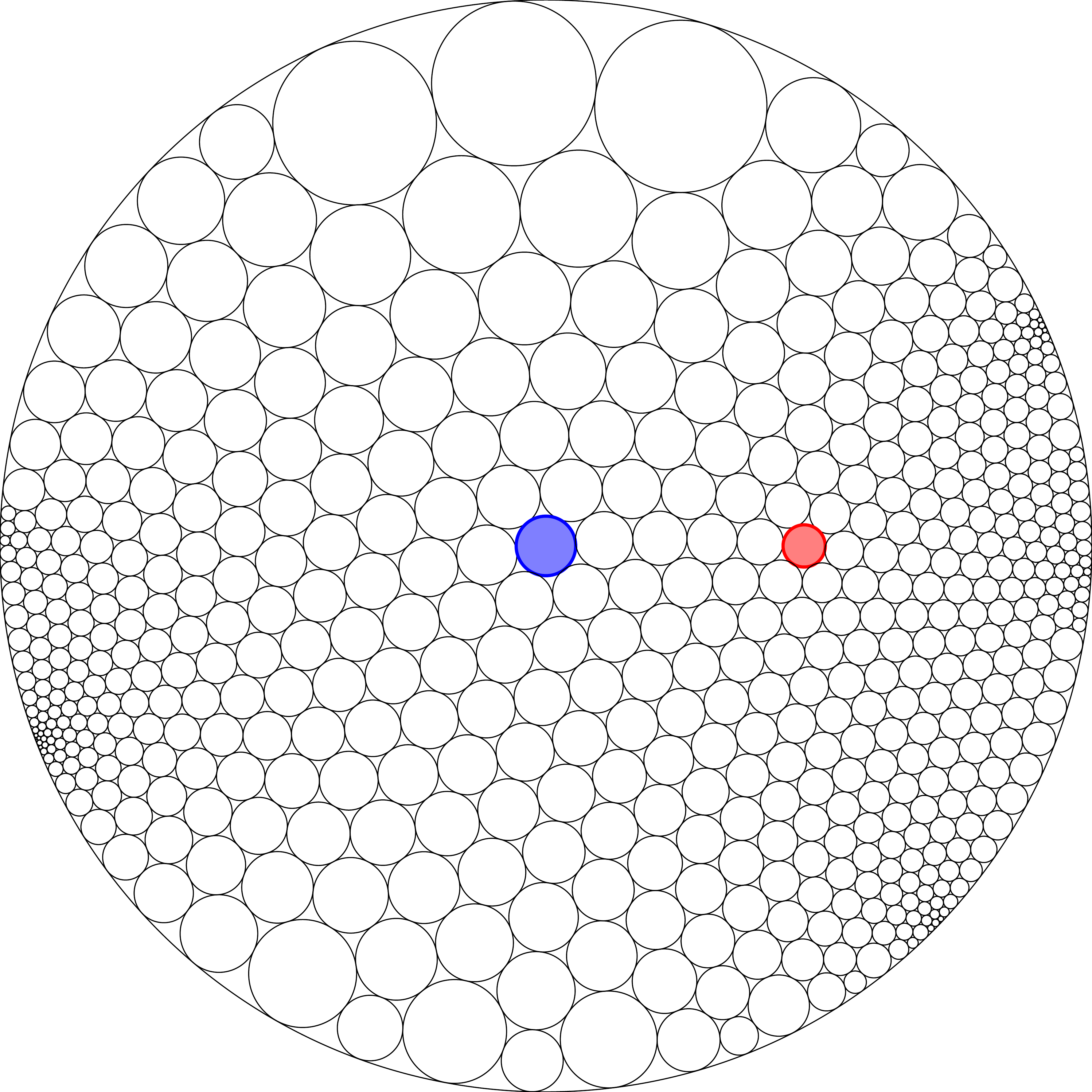}
\caption{}
\end{subfigure}
\caption{Discrete Riemann mappings with finer and finer hexagonal circle packings.}
\label{Figure:DCM}
\end{figure}
The scheme is rather simple. Overlay a domain $D$ with a hexagonal circle packing $\mathcal{H}_{\varepsilon}$ of constant circle radii $\varepsilon$, a `penny packing.'  Use the domain $D$ as a cookie cutter to cut out a portion of the packing, say $\mathcal{P}_{\varepsilon}$, whose combinatorics are given by the simplicial complex $T_{\varepsilon}$. Apply the Maximal Disk Packing Theorem to obtain a maximal circle packing $\mathcal{Q}_{\varepsilon}$ of the disk $\mathbb{D}$. Choosing two points $x$ and $y$ in the domain $D$, let $u_{\varepsilon}$ and $v_{\varepsilon}$ be the vertices of $T_{\varepsilon}$ whose corresponding circles are closest to the respective points $x$ and $y$. From the uniqueness of the Maximal Disk Packing Theorem, one may assume that the packings $\mathcal{Q}_{\varepsilon}$ have been normalized so that the circle corresponding to $u_{\varepsilon}$ is centered at the origin and the one corresponding to $v_{\varepsilon}$ is centered on the positive real axis. Define the \textit{discrete Riemann mapping}\index{discrete Riemann mapping} $ f_{\varepsilon} : \mathrm{carr}(\mathcal{P}_{\varepsilon}) \to \mathrm{carr}(\mathcal{Q}_{\varepsilon})$ as the piecewise linear mapping that takes centers of circles of $\mathcal{P}_{\varepsilon}$ to corresponding centers of circles of $\mathcal{Q}_{\varepsilon}$. Thurston's claim of his 1985 lecture that Rodin and Sullivan verified in 1987 is the content of the next theorem.

\begin{DRMT}[Rodin and Sullivan~\cite{RS87}]
The mappings $f_{\varepsilon}$ converge as $\varepsilon \downarrow 0$, uniformly on compact subsets of $D$, to the Riemann mapping $f$ of $D$ onto $\mathbb{D}$ with $f(x) = 0$ and $f(y) > 0$.
\end{DRMT}

Before I discuss the proof, I should say that there is nothing special about the hexagonal combinatorics. He and Schramm~\cite{HS96} verified that the particular combinatorics of the overlay packings are irrelevant as long as the maximum circle radii approach zero.

\begin{proof}[Sketch of proof.]
The proof applies classical tools from quasiconformal analysis to confirm convergence of the discrete mappings to the Riemann mapping. There are three parts. First, the Ring Lemma, already used on page~\pageref{RingLemma} in the proof of the Discrete Uniformization Theorem, is used to observe that the discrete Riemann mappings $f_{\varepsilon}$ for $\varepsilon > 0$ form a  family of uniformly quasiconformal mappings with, say, dilatation of all maps bounded by $\mu \geq 1$. Second, standard results of quasiconformal analysis imply that the mappings converge to a $\mu$-quasiconformal mapping $f$ of $D$ onto $\mathbb{D}$. Third, the limit mapping is proved to be $1$-quasiconformal, or just conformal, so that it is a Riemann mapping of the domain $D$ onto the disk $\mathbb{D}$. Allow me to fill in each of the three parts of the argument a bit.

The first part, that the discrete Riemann mappings have quasiconformal distortion uniformly bounded, uses the fact that simplicial homeomorphisms are $\mu$-quasiconformal with the distortion constant $\mu$ depending only on the shapes of the triangles involved. In particular, because the complexes $T_{\varepsilon}$ have constant degree six on interior vertices, the Ring Lemma implies that there is a minimum possible angle $\omega > 0$ for any of the triangles in the Euclidean carrier $\mathrm{carr}(\mathcal{Q}_{\varepsilon})$, this independent of $\varepsilon$. This implies that the discrete maps $f_{\varepsilon}$ are uniformly $\mu$-quasiconformal since the images of the equilateral triangles of $\mathrm{carr}(\mathcal{P}_{\varepsilon})$ are triangles of $\mathrm{carr}(\mathcal{Q}_{\varepsilon})$ of uniformly bounded distortion. 

The second part now follows from standard tools of quasiconformal analysis. The uniformly quasiconformal maps $f_{\varepsilon}$ are equicontinuous on compact subsets of $D$, as are the maps $f_{\varepsilon}^{-1}$ on compact subsets of $\mathbb{D}$. It follows that the family $\{f_{\varepsilon}\}_{\varepsilon >0}$ is a normal family and any limit mapping $f$ is bijective between $D$ and $\mathbb{D}$. This latter claim uses the fact that any limit mapping is necessarily $\mu$-quasiconformal, and the Carath\'eodory Kernel Theorem implies that $f$ takes $D$ onto $\mathbb{D}$.

Finally, that any limit mapping $f$ is conformal follows from the \textit{Hexagonal Packing Lemma}. This says that in a packing with hexagonal combinatorics, any two adjacent circles buried deeply within the packing have nearly equal radii. Here is the exact statement.

\begin{HPL}[Rodin and Sullivan~\cite{RS87}]
There is a sequence $c_{n}$ decreasing to zero as $n\to \infty$ such that in any packing with $n$ generations of the regular hexagonal combinatorics surrounding circle $C$, the ratio of radii of $C$ and any adjacent circle differs from unity by less than $c_{n}$. 
\end{HPL}

This lemma shows that as $\varepsilon \downarrow 0$, the mappings $f_{\varepsilon}$ restricted to a fixed compact subset of $D$ maps equilateral triangles to triangles of $\mathrm{carr}(\mathcal{Q}_{\varepsilon})$ that become arbitrarily close to equilateral, and this implies that any limit mapping is conformal.

This completes the proof modulo the proof of the Hexagonal Packing Lemma. This is proved as follows. Let $H_{n}$ be any packing of circles in the plane with combinatorics given by greater than or equal to $n$ generations of the hexagonal packing and whose central circle is the circle $C_{0}$ of unit radius centered at the origin. The Ring Lemma implies that the radii of the circles $n$ generations removed from $C_{0}$ in the packings $H_{m}$ for $m\geq n$ are bounded away from zero and infinity. A diagonal argument implies that there is a subsequence $H_{n_{i}}$ that geometrically converges to a packing $H$, which necessarily has hexagonal combinatorics. But the uniqueness of the Discrete Uniformization Theorem implies that $H = \mathcal{H}_{1}$, the penny packing of unit radius. If the lemma were not true, one could choose the sequence $H_{n}$ in such a way that the ratio of the center circle of $H_{n}$ to at least one of its neighbors differs from unity by at least a fixed constant $\delta > 0$. This would imply that the limit packing $H$ has a circle adjacent to $C_{0}$ of non-unit radius, contradicting uniqueness.
\end{proof}
I should mention that Rodin and Sullivan did not have access to the Discrete Uniformization Theorem in 1987 as it was published only in 1990. They had to prove uniqueness of the penny packing of the plane, which they did by invoking results of Dennis Sullivan~\cite{dS81} extending the Mostow Rigidity Theorem to non-compact three-manifolds whose volumes grow slowly enough. This initiated an attempt to prove the Hexagonal Packing Lemma using only elementary means, which ultimately led to a better understanding of the rigidity of infinite circle packings over the next decade. This paper of Rodin and Sullivan was highly influential and can claim to be the genesis of the serious study of circle packings that now includes in its accomplishments hundreds of articles, thousands of citations, and a huge reservoir of applications in a great variety of different settings.

\subsection{Uniformizing equilateral surfaces.}\index{equilateral surface}
I already have defined piecewise equilateral metrics determined by plane triangulation graphs in the context of the type problem. Of course there is nothing special about plane triangulation graphs. Any triangulation $T$ of a surface may be endowed with a piecewise equilateral metric by identifying faces with unit equilateral triangles. Exactly as explained in Section~\ref{Section:type}, this endows the surface with a complex atlas of conformal charts indexed by the vertices and edges of the triangulation. Equilateral surfaces have become important in several different areas of mathematics. They arise for example in Grothendieck's theory of dessins d'enfants and their corresponding Bely{\u \i} maps, see~\cite{BS04}, in Angel and Schramm's theory of uniform infinite planar triangulations~\cite{AS03}, in Gill and Rohde's study of random planar maps~\cite{GR13}, in Bowers and Stephenson's theory of conformal tilings and especially those that arise from expansion complexes~\cite{BS14a,BS14b}, and in discrete conformal flattening of surfaces in $\mathbb{R}^{3}$~\cite{Bowers:2003kr}. In this section I introduce a method of uniformizing these surfaces using the tools of Rodin-Sullivan~\cite{RS87} and basic surface theory.

Let $T$ be a triangulation of the topological surface $S$. The notation $|T|_{\text{eq}}$ is used to denote the piecewise equilateral metric space determined by the triangulation $T$ and $\mathcal{S}_{T}$ to denote the Riemann surface determined by the atlas $\mathcal{A} = \{ \varphi_{x} : x \in V(T) \cup E(T)\}$. Note that $T$ need not be a simplicial triangulation for this to make sense. A face $f$ of $T$ first is identified as an equilateral triangle in $|T|_{\text{eq}}$ and then as a curvilinear triangle in the canonical metric of constant curvature on the surface $\mathcal{S}_{T}$. What is the shape of $f$ in $\mathcal{S}_{T}$? One fact about the shape of this curvilinear triangle is that the angle that two of its sides makes that emanate from the same vertex is $2\pi /d$, where $d$ is the degree of the vertex. Another fact is that the sides are analytic arcs, and in fact any such arc is the fixed point set of an anti-conformal reflection that exchanges the two triangles incident with that arc. In the case $\mathcal{S}_{T}$ is parabolic or hyperbolic, $f$ can be lifted to the plane $\mathbb{C}$ or the Poincar\'e disk $\mathbb{D}$ and so this shape may be displayed as a curvilinear triangle in the plane. In case $\mathcal{S}_{T}$ is elliptic, this shape may be stereographically projected from the $2$-sphere to the plane. How does one get at this shape? The answer Ken Stephenson and I supplied in~\cite{BS04} is the content of this section.

For simplicity, let's restrict our attention to closed surfaces. The scheme for approximating a uniformizing map is to use the triangulation $T$ as a pattern for a circle packing, and then refine iteratively using so-called \textit{hex-refinement} to obtain a sequence $\mathcal{P}_{n}$ of finer and finer circle packings, after an initial barycentric subdivision. Hex-refinement applied to a triangular face just adds a vertex to each existing edge and then connects the three new vertices on the three edges of the face by a $3$-cycle of edges, thus subdividing the face into four smaller triangles.
%; see~\Cref{Figure:Hex}.
%\begin{figure}
%\caption{Hex-subdivision of a face applied once and four times.}
%\label{Figure:Hex}
%\end{figure}
Thus barycentric subdivision followed by hex-refinement produces $T_{1}$, and iteration of hex-refinement then produces the sequence $T_{n}$ with $\mathcal{P}_{n}$ the corresponding circle packing in the surface $\mathcal{S}_{n}$ in the pattern of $T_{n}$. There is an added layer of difficulty here in that, unlike with the use of the hexagonal packing in the Discrete Riemann Mapping Theorem, the circle packings in this setting do not occupy the same surface. The surfaces $\mathcal{S}_{n}$ are determined by the triangulations $T_{n}$ according to Theorem~\ref{Theorem:SurfacesPack}, and these need not be conformally equivalent to one another. Also, any face $f$ of $T$ with $nth$ hex-subdivision $f_{n}$ in $T_{n}$ determines a sequence $\mathcal{P}_{n}(f)$ of circle packings, those circles in $\mathcal{P}_{n}$ corresponding to the vertices of $f_{n}$.

\begin{DUTES}[Bowers and Stephenson~\cite{BS04}]
The surfaces $\mathcal{S}_{n}$ converge in moduli as $n\to \infty$ to a surface $\mathcal{S}$ that is conformally homeomorphic to the surface $\mathcal{S}_{T}$, the Riemann surface determined by the equilateral surface $|T|_{\text{\em eq}}$. For any face $f$ of $T$, the carriers of $\mathcal{P}_{n}(f)$ converge geometrically to the shape of $f$ in $\mathcal{S}_{T}$ when given its canonical constant curvature metric.
\end{DUTES}

The latter statement of the theorem may be understood to mean that when one lifts the carriers to the universal cover, the sphere $\mathbb{S}^{2}$, the plane $\mathbb{C}$, or the disk $\mathbb{D}$, and normalizes appropriately, the carriers converge in the Hausdorff metric on compacta to the appropriate lift of $f$ in $\mathcal{S}_{T}$.

\begin{proof}[Sketch of proof.]
Note that the realizations of the triangulation $T$ in the metric surface $|T|_{\text{eq}}$ and in the Riemann surface $\mathcal{S}_{T}$ are reflective, meaning that each edge $e$ is the fixed point set of an anti-conformal reflection that exchanges the two faces contiguous to $e$.\footnote{To be clear, the reflection is anti-conformal on the \textit{interior} of the union of the two faces incident at $e$, but not at the vertices.} Rather than the canonical constant curvature metric, I shall use the piecewise equilateral metric $\rho_{T}$ on $\mathcal{S}_{T}$ throughout the proof. Here is a key observation. Hex-subdivision may be performed metrically in $\mathcal{S}_{T}$ by adding new vertices $v(e)$ as the mid-points of the edges $e\in E(T)$ and connecting $v(e)$ to $v(e')$ by a Euclidean straight line segment in the metric $\rho_{T}$ in the face bounded by edges $e$, $e'$ and $e''$. This realizes the hex-refined triangulation $T_{1}$ as a reflective triangulation in $\mathcal{S}_{T}$.\footnote{Technically, this is after the initial barycentric subdivision, which also is performed in the metric $\rho_{T}$ and yields a reflective triangulation.} Iterating, $T_{n}$ may be realized as a reflective triangulation of $\mathcal{S}_{T}$ that metrically hex-subdivides $T_{n-1}$.

Define homeomorphisms $h_{n} : \mathcal{S}_{T} \to \mathcal{S}_{n}$ so that the image of vertex $v$ of $T_{n}$ under $h_{n}$ is the center of the circle that corresponds to $v$ in the circle  packing $\mathcal{P}_{n}$, extend linearly along edges and then with minimum quasiconformal distortion across faces. By the Ring Lemma, each mapping $h_{n}$ is quasiconformal, and since hex-refinement does not increase degree, any bound $\geq 6$ on the degrees of the vertices of $T$ also bounds the degrees of the vertices of $T_{n}$, for all $n \geq 1$. This implies that the homeomorphisms $h_{n}$ have uniformly bounded dilatations, and this implies that a subsequence of the surfaces $\mathcal{S}_{n}$ converges in moduli to a Riemann surface $\mathcal{S}$.  

My claim is that $\mathcal{S}$ is conformally equivalent to $\mathcal{S}_{T}$. This would be confirmed were the maximum dilatations of the homeomorphisms $h_{n}$ shown to limit to unity as $n\to \infty$, but unfortunately this does not occur. In fact these dilatations are bounded away from unity with large dilatations concentrated near the original vertices of $T$. To get around this, let $D$ be a compact domain in $\mathcal{S}_{T}$ disjoint from the vertex set $V(T)$. Note that the combinatorics of $T_{n}$ away from the vertices of $T$ is hexagonal, and this implies that as $n\to \infty$, the compact set  $D$ is surrounded by a number of generations of the hexagonal combinatorics that increases without bound. The Hexagonal Packing Lemma applies to confirm that the maximum dilatations of the restrictions of the homeomorphisms $h_{n}$ to $D$ converge to unity. This works for every compact domain that misses the vertex set $V(T)$, and this implies that the limit mapping $h :\mathcal{S}_{T} \to \mathcal{S}$ is conformal on the complement of the vertex set $V(T)$. Now the removability of isolated singularities comes into play and implies that the homeomorphism $h$ is conformal at the vertices, and so is a conformal homeomorphism of $\mathcal{S}_{T}$ onto $\mathcal{S}$.
\end{proof}

\begin{figure}
\includegraphics[width = 0.7\textwidth]{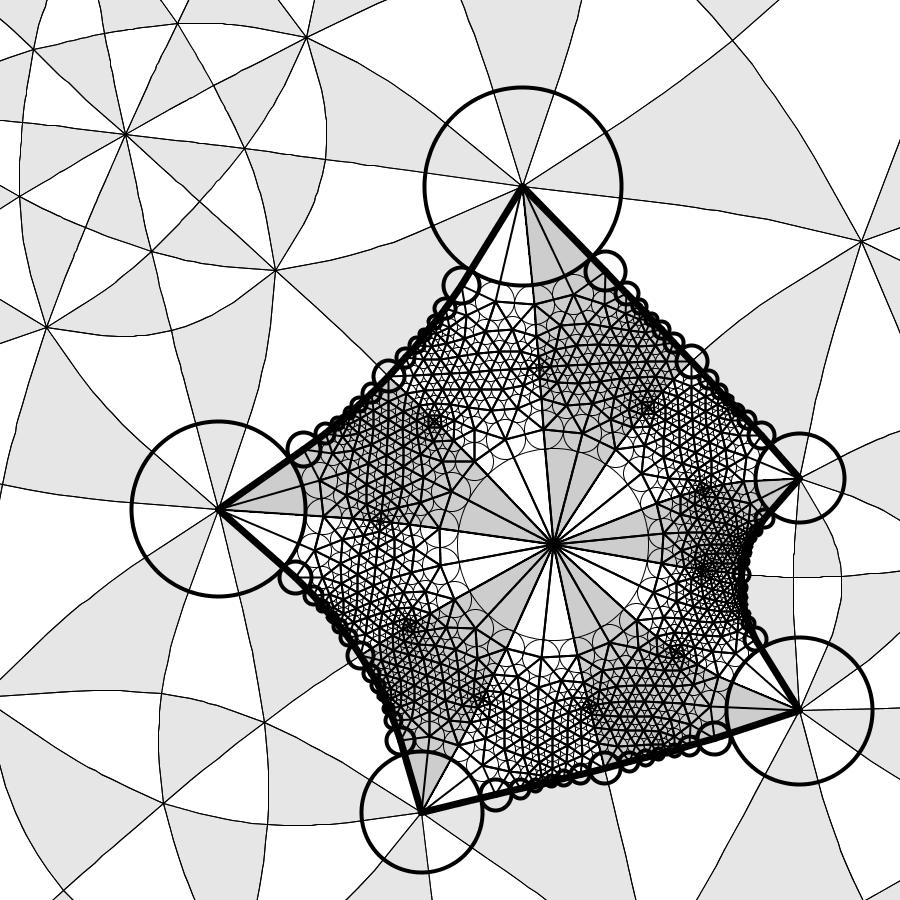}
\caption{Conformal shapes of equilateral triangles in a planar equilateral surface approximated with the circle packing of the twice hex-refined barycentric subdivision of the original triangulation.}
\label{Figure:EQSurfaces}
\end{figure}
\Cref{Figure:EQSurfaces} shows an example of an approximation to a portion of an equilateral surface uniformized in the plane. In this figure each edge is the fixed point set of an anti-conformal reflection that exchanges the grey-white pair of triangles sharing that edge. This is an approximation of the conformally correct shapes of the equilateral triangles forming the equilateral surface being imaged.

%\subsection{Discrete minimal surfaces.}

\section{Inversive Distance Circle Packings.}\index{circle packing!inversive distance} Around 2001, Ken Stephenson and I began thinking about \textit{inversive distance circle packings} and how they could be used to uniformize piecewise flat surfaces, those surfaces in which each face is identified with a flat Euclidean triangle, not necessarily equilateral. There is a tentative discussion of this in~\cite{BS04} and further discussion in~\cite{Bowers:2003kr} of the difficulties in proving convergence of discrete mappings to the uniformization mapping, though the method does seem to work well in practice; again see~\cite{Bowers:2003kr}. The first theoretical questions concern (1) the existence of circle packings with prescribed inversive distances between adjacent circles and (2) the rigidity and uniqueness of these packings. 

When all inversive distances lie in the unit interval, adjacent circles overlap with specified angle $0\leq \theta \leq \pi/2$. This is covered by the Koebe-Andre'ev-Thurston Theorems. When inversive distance is greater than unity, the circles do not overlap and the inversive distance is a M\"obius-invariant measure of how separated the circles are. In this case Problems (1) and (2) seem much more difficult to approach. Problem (1) is especially difficult in that there are local assignments of inversive distances that must be avoided as there are no circle configurations that realize those distances. These are difficult to catalogue, but even if there are no local obstructions to the existence of a packing, it is not at all clear whether still there may be global obstructions. Little progress has been made on Problem (1), but the situation for Problem (2) has enjoyed some progress, initially in 2011 and more recently in the past couple of years. It is these recent successes in approaching Problem (2) that occupies this section. My contention is that a change of viewpoint can be effective in approaching inversive distance circle packings, and a hint as to how to proceed comes from the  classical rigidity theory of bar-and-joint linkages. After a brief review of inversive distance, I will explore this new framework for circle packings and discuss some recent successes.

\subsection{A quick introduction to inversive distance.}
There are a number of ways to define the inversive distance between two circles in the Riemann sphere. I will present several of these below, starting with the most mundane that gives a Euclidean formula for the inversive distance between two planar circles.\footnote{This easily can be extended to the inversive distance between a circle and a line, or two lines. I will forgo this development since the next definition is completely general.} Let $C_{1}$ and $C_{2}$ be distinct circles in the complex plane $\mathbb{C}$ centered at the respective points $p_{1}$ and $p_{2}$, of respective radii $r_{1}$ and $r_{2}$, and bounding the respective \textit{companion disks} $D_{1}$ and $D_{2}$. 

\begin{Definition}[\textsc{inversive distance in the euclidean metric}]
The \textit{inversive distance} $\langle C_{1}, C_{2} \rangle$ between $C_{1}$ and $C_{2}$ is
\begin{equation}\label{EQ:planeID} 
	\langle C_{1}, C_{2} \rangle = \frac{|p_{1}-p_{2}|^{2} - r_{1}^{2} - r_{2}^{2}}{2r_{1}r_{2}}.
\end{equation}
The \textit{absolute inversive distance} between distinct circles is the absolute value of the inversive distance.
\end{Definition}

The absolute inversive distance is a M\"obius invariant of the placement of two circles in the plane. This means that there is a M\"obius transformation of $\mathbb{C}$ taking one circle pair to another if and only if the absolute inversive distances of the two pairs agree. The important geometric facts that make the inversive distance useful in inversive geometry and circle packing are as follows. When $\langle C_{1},C_{2}\rangle > 1$, $D_{1} \cap D_{2} = \emptyset$ and $\langle C_{1},C_{2}\rangle = \cosh \delta$, where $\delta$ is the hyperbolic distance between the totally geodesic hyperbolic planes in the upper-half-space model $\mathbb{C} \times (0,\infty)$ of $\mathbb{H}^{3}$ whose ideal boundaries are $C_{1}$ and $C_{2}$. When $\langle C_{1},C_{2}\rangle = 1$, $D_{1}$ and $D_{2}$ are tangent at their single point of intersection. When $1 >\langle C_{1},C_{2}\rangle \geq 0$, $D_{1}$ and $D_{2}$ overlap with angle $0 < \theta \leq \pi/2$ with $\langle C_{1},C_{2}\rangle = \cos \theta$. In particular, $\langle C_{1},C_{2}\rangle = 0$ precisely when $\theta = \pi/2$. When $\langle C_{1},C_{2}\rangle < 0$, then $D_{1}$ and $D_{2}$ overlap by an angle greater than $\pi/2$. This includes the case where one of $D_{1}$ or $D_{2}$ is contained in the other, this when $\langle C_{1},C_{2}\rangle \leq -1$. In fact, when $\langle C_{1},C_{2}\rangle < -1$ then $\langle C_{1},C_{2}\rangle = - \cosh \delta$ where $\delta$ has the same meaning as above, and when $\langle C_{1},C_{2}\rangle = -1$ then $C_{1}$ and $C_{2}$ are `internally' tangent. When $-1 < \langle C_{1},C_{2}\rangle <0$, then the overlap angle of $D_{1}$ and $D_{2}$ satisfies $\pi > \theta > \pi/2$ and again $\langle C_{1},C_{2}\rangle = \cos \theta$. 

The more general definition measures the inversive distance between oriented circles. Note that an oriented circle determines a unique closed \textit{companion} or \textit{spanning disk} that the circle bounds. Indeed, assuming fixed orientations for $\mathbb{S}^{2}$ and $\widehat{\mathbb{C}}$ that are compatible via stereographic projection, the companion disk determined by the oriented circle $C$ is the closed complementary disk $D$ (of the two available) whose positively oriented boundary $\partial^{+} D = C$, where of course the orientation of $D$ is inherited from that of $\mathbb{S}^{2}$ or $\widehat{\mathbb{C}}$. This is described colloquially by saying that $D$ lies to the left of $C$ as one traverses $C$ along the direction of its orientation. 

\begin{Definition}[\textsc{general inversive distance}]
Let $C_{1}$ and $C_{2}$  be oriented circles in the extended plane $\widehat{\mathbb{C}}$ bounding their respective companion disks $D_{1}$ and $D_{2}$, and let $C$ be any oriented circle mutually orthogonal to $C_{1}$ and $C_{2}$.  Denote the points of intersection of $C$ with $C_{1}$ as $z_{1},z_{2}$ ordered so that the oriented sub-arc of $C$ from $z_{1}$ to $z_{2}$ lies in the disk $D_{1}$.  Similarly denote the ordered points of intersection of $C$ with $D_{2}$ as $w_{1},w_{2}$.  The \textit{general inversive distance} between $C_{1}$ and $C_{2}$, denoted as $\langle C_{1},C_{2}\rangle$, is defined in terms of the cross ratio
\begin{equation*}
[z_{1},z_{2};w_{1},w_{2}]=\frac{(z_{1}-w_{1})(z_{2}-w_{2})}{(z_{1}-z_{2})(w_{1}-w_{2})}
\end{equation*}
\noindent by
\begin{equation*}
\langle C_{1},C_{2}\rangle =2[z_{1},z_{2};w_{1},w_{2}]-1.
\end{equation*}
Subsequently, I'll drop the adjective \textit{general} and refer to the inversive distance $\langle C_{1}, C_{2} \rangle$ with its absolute value $|\langle C_{1}, C_{2} \rangle |$ the \textit{absolute inversive distance}.\footnote{The author first learned of defining inversive distance in this way from his student, Roger Vogeler. He has looked for this in the literature and, unable to find it, can only surmise that it is original with Prof.~Vogeler. The definition appeared in~\cite{Bowers:2003kr} in 2003.}
\end{Definition}

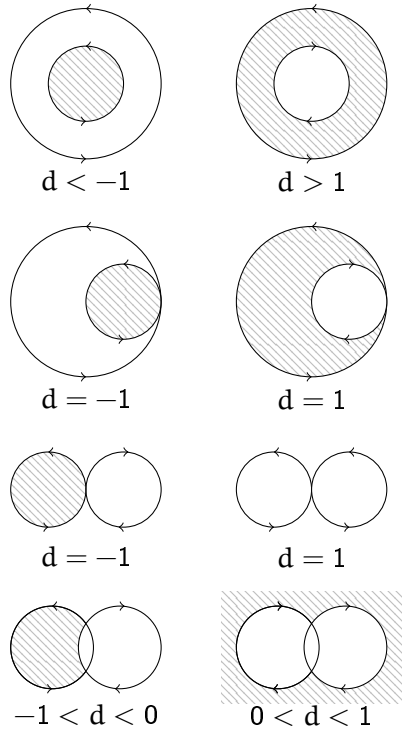
\begin{figure}

\begin{tikzpicture}
   
   \clip (-1.1, -6.5) rectangle (4.3, 3.7);
   
	% Top Left
	\begin{scope}[decoration = {markings, mark = at position 0.25 with {\arrow{>}}, mark = at
     position 0.75 with {\arrow{>}}, }]
     \draw[postaction = decorate] (0,2.4) circle (1.0cm);
   \end{scope}
   
   \begin{scope}[decoration = {markings, mark = at position 0.25 with {\arrow{>}}, mark = at
     position 0.75 with {\arrow{>}}, }]
     \draw[postaction = decorate, pattern=north west lines,  pattern color=lightgray] (0,2.4) circle (0.5cm);
   \end{scope}
   
   \node at (0,1.1) {$d < -1$};
   
   % Top Right
   
   \begin{scope}[decoration = {markings, mark = at position 0.25 with {\arrow{>}}, mark = at
     position 0.75 with {\arrow{>}}, }]
     \draw[postaction = decorate, pattern=north west lines,  pattern color=lightgray] (3,2.4) circle (1.0cm);
   \end{scope}
   
   \begin{scope}[decoration = {markings, mark = at position 0.25 with {\arrow{<}}, mark = at
     position 0.75 with {\arrow{<}}, }]
     \draw[postaction = decorate, fill=white] (3,2.4) circle (0.5cm);
   \end{scope}
   
   \node at (3,1.1) {$d > 1$};

	% 2 Left
   \begin{scope}[decoration = {markings, mark = at position 0.25 with {\arrow{>}}, mark = at
     position 0.75 with {\arrow{>}}, }]
     \draw[postaction = decorate, pattern=north west lines,  pattern color=lightgray] (0.5,-0.5) circle (0.5cm);
   \end{scope}
   
   \begin{scope}[decoration = {markings, mark = at position 0.25 with {\arrow{>}}, mark = at
     position 0.75 with {\arrow{>}}, }]
     \draw[postaction = decorate] (0,-0.5) circle (1.0cm);
   \end{scope}
   
   \node at (0,-1.8) {$d = -1$};
   
   % 2 Right
   
   \begin{scope}[decoration = {markings, mark = at position 0.25 with {\arrow{>}}, mark = at
     position 0.75 with {\arrow{>}}, }]
     \draw[postaction = decorate, pattern=north west lines,  pattern color=lightgray] (3,-0.5) circle (1.0cm);
   \end{scope}
   
   \begin{scope}[decoration = {markings, mark = at position 0.25 with {\arrow{<}}, mark = at
     position 0.75 with {\arrow{<}}, }]
     \draw[postaction = decorate, fill=white] (3.5,-0.5) circle (0.5cm);
   \end{scope}
   
   \node at (3,-1.8) {$d = 1$};
   
	% 3 Left
   \begin{scope}[decoration = {markings, mark = at position 0.25 with {\arrow{<}}, mark = at
     position 0.75 with {\arrow{<}}, }]
     \draw[postaction = decorate] (0.5,-3) circle (0.5cm);
   \end{scope}

   \begin{scope}[decoration = {markings, mark = at position 0.25 with {\arrow{>}}, mark = at
     position 0.75 with {\arrow{>}}, }]
     \draw[postaction = decorate, pattern=north west lines,  pattern color=lightgray] (-0.5,-3) circle (0.5cm);
   \end{scope}
   
   \node at (0,-3.9) {$d = -1$};

   % 3 Right
   
   \begin{scope}[decoration = {markings, mark = at position 0.25 with {\arrow{>}}, mark = at
     position 0.75 with {\arrow{>}}, }]
     \draw[postaction = decorate] (2.5,-3) circle (0.5cm);
   \end{scope}
   
   \begin{scope}[decoration = {markings, mark = at position 0.25 with {\arrow{>}}, mark = at
     position 0.75 with {\arrow{>}}, }]
     \draw[postaction = decorate, fill=white] (3.5,-3) circle (0.5cm);
   \end{scope}
   
   \node at (3,-3.9) {$d = 1$};

	% 4 Left
   
   \begin{scope}[decoration = {markings, mark = at position 0.25 with {\arrow{>}}, mark = at
     position 0.75 with {\arrow{>}}, }]
     \draw[postaction = decorate, pattern=north west lines,  pattern color=lightgray] (-0.45,-5.1) circle (0.55cm);
   \end{scope}
   
   \begin{scope}[decoration = {markings, mark = at position 0.25 with {\arrow{<}}, mark = at
     position 0.75 with {\arrow{<}}, }]
     \draw[postaction = decorate, fill=white] (0.45,-5.1) circle (0.55cm);
   \end{scope}
   
   \draw (-0.45,-5.1) circle (0.55cm);
   
   \node at (0,-6.0) {$-1 < d < 0$};
   
   % 4 Right
   
   \path [pattern=north west lines,  pattern color=lightgray] (1.8,-4.35) rectangle (4.2, -5.85);
   
   \begin{scope}[decoration = {markings, mark = at position 0.25 with {\arrow{<}}, mark = at
     position 0.75 with {\arrow{<}}, }]
     \draw[postaction = decorate, fill=white] (2.55,-5.1) circle (0.55cm);
   \end{scope}
   
   \begin{scope}[decoration = {markings, mark = at position 0.25 with {\arrow{<}}, mark = at
     position 0.75 with {\arrow{<}}, }]
     \draw[postaction = decorate, fill=white] (3.45,-5.1) circle (0.55cm);
     
   \draw[postaction = decorate] (2.55,-5.1) circle (0.55cm);
   \end{scope}
   
   \node at (3,-6.0) {$0 < d < 1$};
   
\end{tikzpicture}

\caption{Inversive distances $d=\langle C_{1}, C_{2} \rangle$. The shaded regions are the intersections $D_{1}\cap D_{2}$, the points common to the spanning disks $D_{1}$ and $D_{2}$ for both circles $C_{1}$ and $C_{2}$.}
\label{fig:ID}
\end{figure}
Recall that cross ratios of ordered $4$-tuples of points in $\widehat{\mathbb{C}}$ are invariant under M\"{o}bius transformations and that there is a M\"obius transformation taking an ordered set of four points of $\widehat{\mathbb{C}}$ to another ordered set of four if and only if the cross ratios of the sets agree.  This implies that which circle $C$ orthogonal to both $C_{1}$ and $C_{2}$ is used in the definition is irrelevant as a M\"{o}bius transformation that set-wise fixes $C_{1}$ and $C_{2}$ can be used to move any one orthogonal circle to another.  Which one of the two orientations on the orthogonal circle $C$ is used is irrelevant as the cross ratio satisfies $[z_{1},z_{2};w_{1},w_{2}]=[z_{2},z_{1};w_{2},w_{1}]$.  This equation also shows that the inversive distance is preserved when the orientation of both circles is reversed so that it is only the relative orientation of the two circles that is important for the definition. In fact, the general inversive distance is a relative conformal measure of the placement of an oriented circle pair on the Riemann sphere. By this I mean that two oriented circle pairs are inversive equivalent if and only if their inversive distances agree. All of this should cause one to pause to develop some intuition about how companion disks may overlap with various values of inversive distances. See~\Cref{fig:ID} for some corrections to possible misconceptions. Finally, the inversive distance is symmetric with $\langle C_{1},C_{2}\rangle  = \langle C_{2},C_{1}\rangle$ since $[z_{1},z_{2};w_{1},w_{2}]=[w_{1},w_{2};z_{1},z_{2}]$.

The inversive distance is real since the cross ratio of points lying on a common circle is real and, in fact, every real value is realized as the inversive distance of some oriented circle pair. Notice that if the orientation of only one member of a circle pair is reversed, the inversive distance merely changes sign.  This follows from the immediate relation $[z_{1},z_{2};w_{2},w_{1}]=1-[z_{1},z_{2};w_{1},w_{2}]$. Despite its name, the inversive distance is not a metric as it fails to be non-negative and fails to satisfy the triangle inequality.\footnote{Some authors, perhaps more aptly, call the inversive \textit{distance} the inversive \textit{product} of $C_{1}$ and $C_{2}$.}

 The third definition is entirely in terms of the spherical metric.
 
\begin{Definition}[\textsc{inversive distance in the spherical metric}]
In the $2$-sphere $\mathbb{S}^{2}$, the inversive distance may be expressed as
\begin{equation}\label{EQ:sphereID} 
	\langle C_{1}, C_{2} \rangle = \frac{-\cos \sphericalangle ( p_{1},p_{2} ) + \cos(r_{1}) \cos(r_{2})}{\sin(r_{1}) \sin(r_{2})} =\frac{- p_{1}\cdot p_{2}  + \cos(r_{1}) \cos(r_{2})}{\sin(r_{1}) \sin(r_{2})}.
\end{equation}
Here, $\sphericalangle (p_{1},p_{2}) = \cos^{-1} (p_{1} \cdot p_{2})$ denotes the spherical distance between the centers, $p_{1}$ and $p_{2}$, of the respective companion disks, $p_{1} \cdot p_{2}$ the usual Euclidean inner product between the unit vectors $p_{1}$ and $p_{2}$, and $r_{1}$ and $r_{2}$ the respective spherical radii of the companion disks. Note that $r_{i} = \cos^{-1} (p_{i} \cdot q_{i})$ for any point $q_{i}$ on the circle $C_{i}$, for $i=1,2$.
\end{Definition}

Verifying the equivalence of this with the general definition is an exercise in the use of trigonometric identities after a standard placement of $C_{1}$ and $C_{2}$ on $\mathbb{S}^{2}$ followed by stereographic projection. This standard placement is obtained by finding the unique great circle $C$ orthogonal to both $C_{1}$ and $C_{2}$ and then rotating the sphere so that this great circle is the equator, which then stereographically projects to the unit circle in the complex plane. The details are left to the reader.

Here are two more quick descriptions of inversive distance. For those conversant with the representation of circles in $\mathbb{S}^{2}$ by vectors in de Sitter space, the inversive distance is the Minkowski inner product between the two points of de Sitter space that represent the two oriented circles. This is, perhaps, the most elegant formulation of the product. The final way I'll describe the inversive distance is a neat little curiosity. Let $C_{1} = \partial\mathbb{D} = \mathbb{S}^{1}$ be the unit circle oriented clockwise and $C_{2}$ a circle oriented counterclockwise that meets the open unit disk non-trivially. Then, as explained on page~\pageref{Page:ConstantCurve}, the intersection $c_{2}$ of $C_{2}$ with the open disk is a curve of constant geodetic curvature in the Poincar\'e disk $\mathbb{D} \cong \mathbb{H}^{2}$. The inversive distance is $\langle C_{1},C_{2} \rangle = \mathrm{curv}(c_{2})$, the geodetic curvature of the cycle $c_{2}$ in the Poincar\'e metric on $\mathbb{D}$. This includes all three cases for the cycle $c_{2}$---a hyperbolic circle in $\mathbb{D}$, a horocycle that meets $\partial \mathbb{D}$ at a single point, or a hypercycle that meets $\partial \mathbb{D}$ at two points.\footnote{My student, Opal Graham, noticed, then proved this when I was lecturing on the curves of constant geodetic curvature in the hyperbolic plane.}

\subsection{Some advances on the rigidity question.} In~\cite{BS04}, \textit{inversive distance circle packings} were introduced. Rather than preassigned overlap angles labeling edges of a triangulation of a surface as in the Koebe-Andre'ev-Thurston Theorems, preassigned inversive distances label the edges. As stated already, questions of interest are of the existence and uniqueness of circle configurations in geometric structures on surfaces that realize the inversive distance data. Though the existence question is wide open, in 2011-12 there were three advances on the uniqueness question for inversive distance packings. First, Guo~\cite{Guo:2011kf} proved that inversive distance packings of closed surfaces of positive genus, ones supporting flat or hyperbolic metrics, are locally rigid whenever the inversive distances are non-negative. Shortly after that, Luo~\cite{Luo:2011ex} improved this to global rigidity, or uniqueness of the packings in the cases considered by Guo. Then in a surprising result of the year following, Ma and Schlenker~\cite{Ma:2012hl} produced a counterexample to global uniqueness for packings of the $2$-sphere. They gave examples of pairs of circle packings of $\mathbb{S}^{2}$ in the pattern of the octahedral triangulation with six circles that satisfy the same inversive distance data, but that are not M\"obius equivalent.

The ingredients of Ma and Schlenker's example are Sch\"onhardt's twisted octahedron, which is an infinitesimally flexible polyhedron in Euclidean space $\mathbb{E}^{3}$, embeddings in de Sitter space $\mathbb{S}^{3}_{1}$, and special properties of the Pogorelov map between different geometries. In 2017, John Bowers and I~\cite{Bowers2017} constructed a large family of Ma-Schlenker-like examples using only inversive geometry, producing many counterexamples to the uniqueness of inversive distance circle packings in the $2$-sphere.

The Sch\"onhardt octahedron is an example of a bar-and-joint linkage important in the rigidity theory of Euclidean frameworks, and its use in the Ma-Schlenker example hinted at a way forward in understanding the rigidity theory of inversive distance circle packings in the $2$-sphere. This led to a fruitful change in viewpoint and a reformatting of the question of uniqueness of inversive distance circle packings to the question of the rigidity---local, global, and infinitesimal---of more general \textit{circle frameworks}. These are  
analogues in M\"obius geometry of the Euclidean frameworks in Euclidean geometry with point configurations in $\mathbb{E}^{3}$ replaced by circle configurations in $\mathbb{S}^{2}$ and the Euclidean metric replaced by the non-metric inversive distance. The analogy is not exact, but the theory of linkages in $\mathbb{E}^{3}$ has been found to be a good guide for understanding some of the rigidity theory of circle frameworks. Part of why this works so well is because the space of circles in the $2$-sphere is a $3$-dimensional incidence geometry that has much in common with the space of points in Euclidean $3$-space. The lines of this geometry are coaxial circle families and the planes are what Carath\'eodory in~\cite{cC08} called bundles of circles. This allows one to define what is meant by a convex collection of circles, planar collections of circles, circle polyhedra, bounded circle configurations, etc. Space constraints in this article interfere with even a cursory account of these issues, so I am content with listing a couple of recent successes of the theory without all the definitions needed for a precise understanding, and then taking some time to set up the language of this change of viewpoint.

The two theorems following are the result, both the statements and the proofs, of an engagement between circle packing theory and the classical rigidity theory  of Euclidean frameworks in $\mathbb{E}^{3}$.

\begin{Theorem}[Bowers, Bowers, and Pratt~\cite{BBP18}]\label{Theorem:Cauchy1}
	Let $\mathcal{C}$ and $\mathcal{C}'$ be two non-unitary, inversive distance circle packings with ortho-circles for the same oriented edge-labeled triangulation of the $2$-sphere $\mathbb{S}^{2}$. If $\mathcal{C}$ and $\mathcal{C}'$ are convex and proper, then there is a M\"obius transformation $T:\mathbb{S}^2\rightarrow\mathbb{S}^2$ such that $T(\mathcal{C}) = \mathcal{C}'$. 
\end{Theorem}

The \textit{edge-label} refers to prescribed inversive distances labeling each edge. 
\textit{Non-unitary} means that the inversive distance between any pair of adjacent circles is not unity; in fact, these inversive distances are in the set $(-1,1)\cup(1, \infty)$. Having \textit{ortho-circles} means that each triple of mutually adjacent circles have an orthogonal circle. This generalizes to a global rigidity theorem about \textit{circle polyhedra}, circle configurations in the pattern of $3$-dimensional polyhedra whose faces correspond to circle configurations that are planar in the incidence geometry of circle space; see~\cite{BBP18} for details.

\begin{Theorem}[Bowers, Bowers, and Pratt~\cite{BBP18}]\label{Theorem:Cauchy2}
	Any two convex and proper non-unitary circle polyhedra with M\"obius-congruent faces that are based on the same oriented abstract spherical polyhedron and are consistently oriented are M\"obius-congruent.
\end{Theorem}

Theorem~\ref{Theorem:Cauchy1} coupled with the Ma-Schlenker example of~\cite{Ma:2012hl} and the examples of~\cite{Bowers2017} show that the uniqueness of inversive distance circle packings, and more generally, of circle polyhedra is exactly analogous to that of Euclidean polyhedra---convex and bounded polyhedra in $\mathbb{E}^{3}$ are prescribed uniquely by their edge lengths and face angles whereas non-convex or unbounded polyhedra are not. The proof of this for convex and bounded Euclidean polyhedra is Cauchy's celebrated rigidity theorem \cite{Cauchy1813}, which is reviewed in Section~\ref{Section:CRT}. The proof of Theorem~\ref{Theorem:Cauchy2} follows Cauchy's original argument, which splits the proof into two components---a combinatorial lemma and a geometric lemma. Cauchy's combinatorial lemma deals with a certain labeling of the edges of any graph on a sphere, and applies to the present setting. The geometric lemma, known as \textit{Cauchy's Arm Lemma}, requires that a polygon with certain properties be defined for each vertex of the polyhedron, and fails to apply here. The main work of the proof is in describing and analyzing a family of hyperbolic polygons called {\em green-black polygons} that are defined for each vertex of a circle polyhedron in a M\"obius-invariant manner. An analogue of Cauchy's Arm Lemma for convex green-black polygons is developed and used to to prove these theorems.

\subsection{Circle frameworks and M\"obius rigidity.} I'll close out this section with a description of the change in viewpoint from circle packings to circle frameworks. This can be done using only absolute inversive distance, but I find it advantageous to remain as general as possible in setting up the viewpoint. The goal is to generalize the language of circle packings and patterns of triangulations and quadrangulations of the $2$-sphere to that of circle realizations of oriented circle frameworks. Let $G$ be a graph, by which I mean a set of vertices $V= V(G)$ and simple edges $E = E(G)$. Both loops and multiple edges are disallowed. An oriented edge incident to the initial vertex $u$ and terminal vertex $v$ is denoted as $uv$, and $-uv$ means the oppositely oriented edge $vu$. I will use the same notation, $uv$, to denote an un-oriented edge, context making the meaning clear. A \textit{circle framework with adjacency graph $G$}, or \textit{\textit{c}-framework} for short, is a collection $\mathcal{C} = \{ C_{u} : u \in V(G) \}$ of oriented circles in $\mathbb{S}^{2}$ indexed by the vertex set of $G$. This is denoted by $G(\mathcal{C})$. Two \textit{c}-frameworks $G(\mathcal{C})$ and $G(\mathcal{C}')$ are \textit{equivalent} if $\langle C_{u}, C_{v} \rangle = \langle C_{u}', C_{v}'\rangle$ whenever $uv$ is an edge of $G$. Let $H$ be a subgroup of the inversive group $\mathrm{Inv}(\mathbb{S}^{2})$ of the $2$-sphere. Two collections $\mathcal{C}$ and $\mathcal{C}'$ of oriented circles indexed by the same set are \textit{$H$-equivalent} or \textit{$H$-congruent} provided there is a mapping $T\in H$ such that $T(\mathcal{C}) = \mathcal{C}'$, respecting the common indexing and the orientations of the circles. When $H$ is not so important they are \textit{inversive-equivalent} or \textit{inversive-congruent}, and when $T$ can be chosen to be a M\"obius transformation, they are \textit{M\"obius-equivalent} or \textit{M\"obius-congruent}. The global rigidity theory of \textit{c}-frameworks concerns conditions on $G$ or $G(\mathcal{C})$ that ensure that the equivalence of the \textit{c}-frameworks $G(\mathcal{C})$ and $G(\mathcal{C}')$ guarantees their $H$-equivalence. Often one restricts attention to \textit{c}-frameworks in a restricted collection $\mathscr{F}$ of \textit{c}-frameworks. In Theorem~\ref{Theorem:Cauchy2}, $\mathscr{F}$ is the collection of non-unitary, convex and proper \textit{c}-polyhedra and the interest is in M\"obius equivalence.

\begin{Definition}[\textsc{labeled graph and circle realization}]
	An \textit{edge-label} is a real-valued function $\beta : E(G) \to \mathbb{R}$ defined on the edge set of $G$, and $G$ together with an edge-label $\beta$ is denoted as $G_{\beta}$ and called an \textit{edge-labeled graph}. The \textit{c}-framework $G(\mathcal{C})$ is a \textit{circle realization} of the edge-labeled graph $G_{\beta}$ provided $\langle C_{u}, C_{v} \rangle = \beta (uv)$ for every edge $uv$ of $G$, which henceforth is denoted as $G_{\beta}(\mathcal{C})$. See \Cref{fig:Octa}.
\end{Definition}

\begin{figure}
\begin{subfigure}[b]{0.45\textwidth}
	\includegraphics[width=\textwidth]{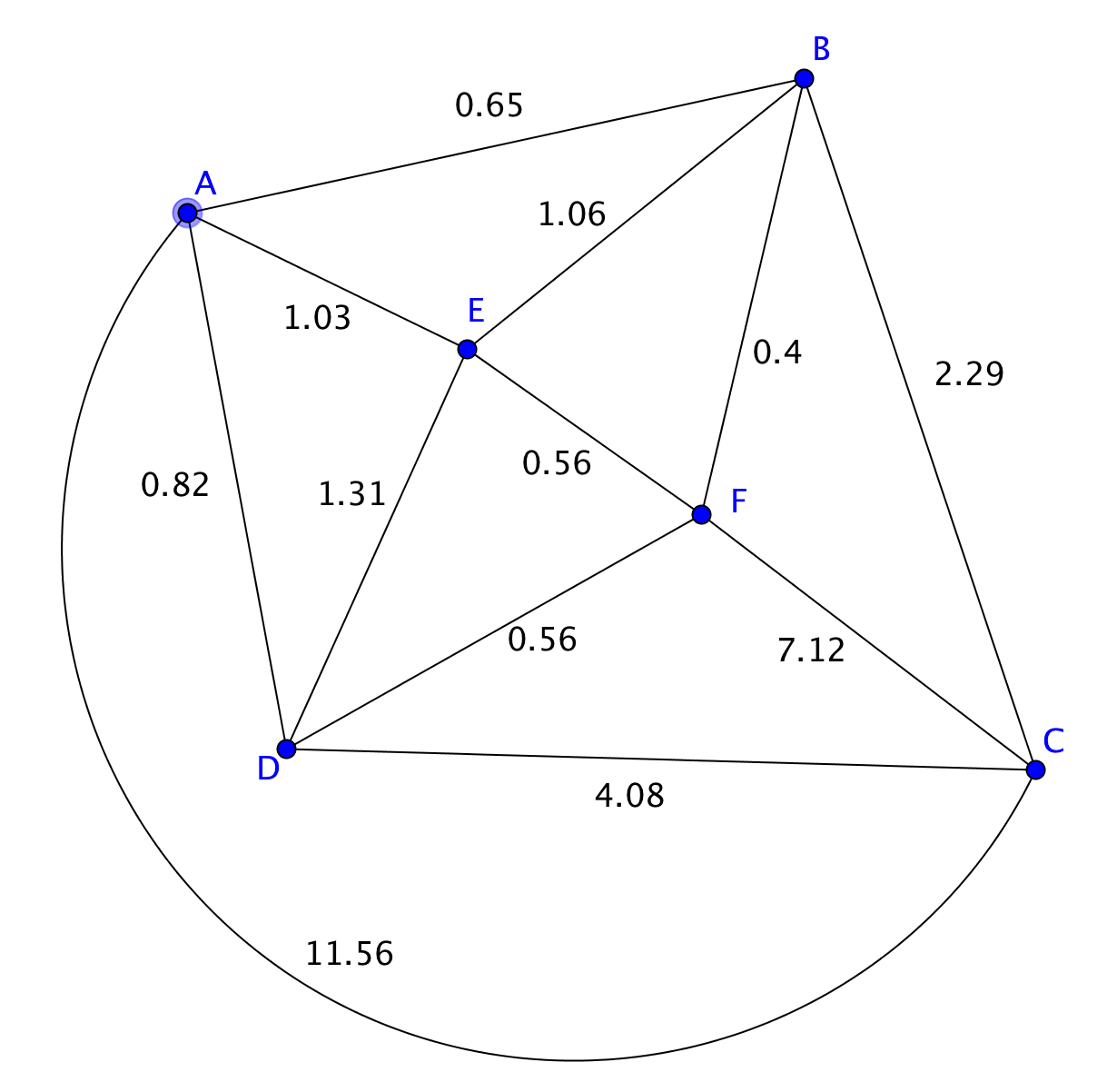}
\caption{An edge-labeled octahedral graph $\mathcal{O}_{\beta}$. Labels $< 1$ imply overlapping circles, $>1$ separated ones.}\label{Oct.A}
\end{subfigure}
\quad
\begin{subfigure}[b]{0.45\textwidth}
	\includegraphics[width=\textwidth]{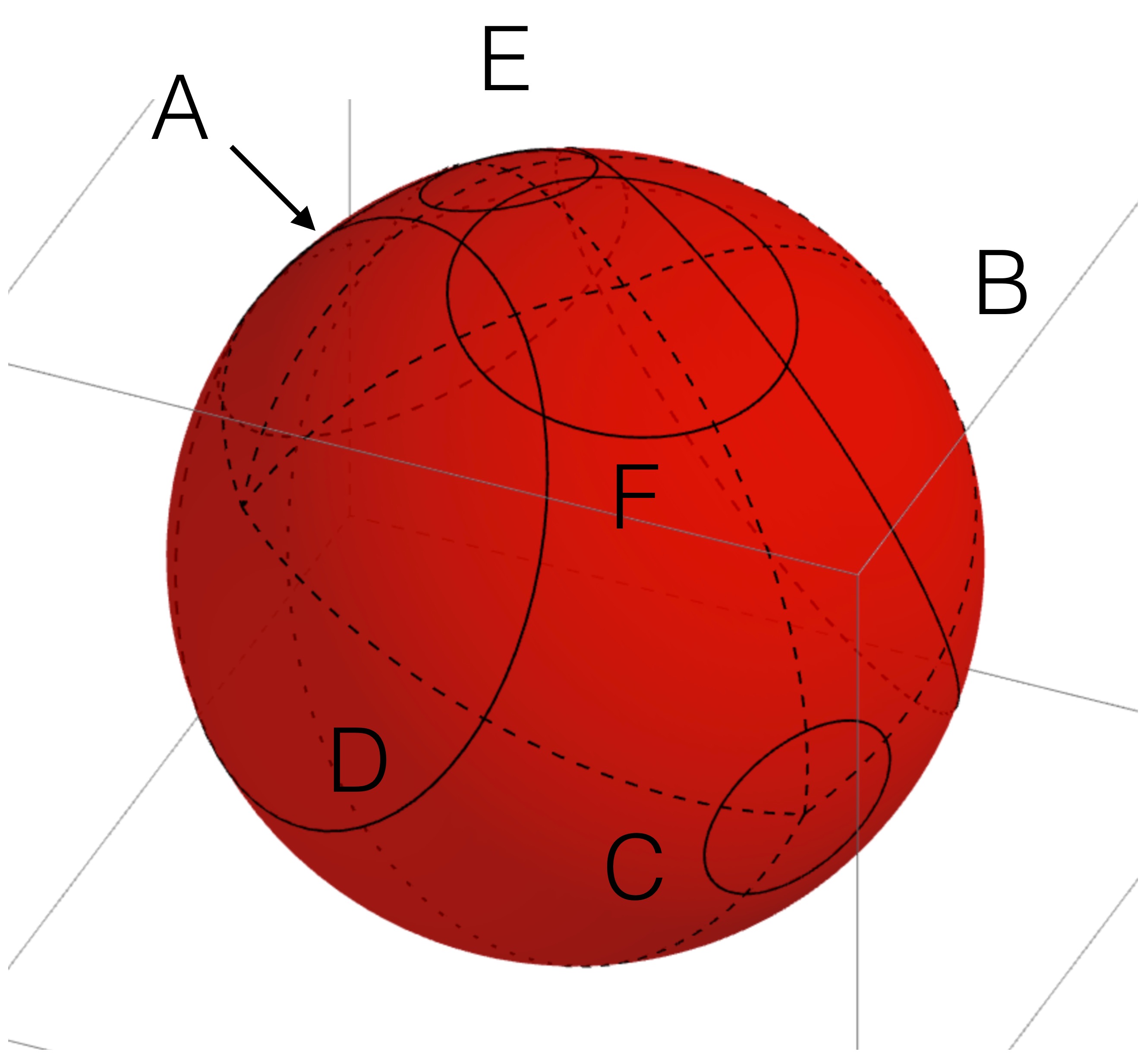}
\caption{A corresponding \textit{c}-framework realizing $\mathcal{O}_{\beta}$. Circle $\mathsf{A}$ is hidden on the back side of the sphere.}\label{Oct.B}
\end{subfigure}
      \caption{An edge-labeled octahedral graph and its circle realization. The labels are proposed inversive distances between the circles corresponding to the vertices.}\label{fig:Octa}
\end{figure}

Circle packings are circle realizations of edge-labeled graphs that arise as the $1$-skeletons of oriented triangulations of the $2$-sphere that also satisfy certain properties that ensure that the realizations of the triangular boundaries of faces respect orientation. The general definition allows for branch vertices and configurations of circles in which the open geodesic triangles cut out by connecting centers of adjacent circles overlap. There are subtleties in which I have no interest, so I am going to adapt a restricted definition that corresponds to the circle packings that arise from spherical polyhedral metrics on triangulated surfaces. These are circle realizations of the edge-labeled $1$-skeleton $G_{\beta}= K^{(1)}_{\beta}$ of an oriented triangulation $K$ of $\mathbb{S}^{2}$ that produce oriented geodesic triangulations\footnote{By this I mean that the orientation of the geodesic triangulation determined by the packing is consistent with the orientation on $K$.} of the 2-sphere when adjacent circle centers are connected by geodesic arcs. The assumption here is that the centers of no two adjacent circles are antipodal, so that there is a unique geodesic arc connecting them, and that the centers of three circles corresponding to the vertices of a face of $K$ do not lie on a great circle. Now this causes no particular problems when all adjacent circles overlap nontrivially, the traditional playing field of circle packing, but does cause some real concern when adjacent circles may have inversive distance greater than unity. For example, a circle realization may produce a geodesic triangulation of the sphere by connecting adjacent centers while its M\"obius image may not. This is traced directly to the fact that neither circle centers nor radii, nor geodesic arcs, are M\"obius invariants in the inversive geometry of the sphere. This behavior does not occur for inversive distance circle packings of the Euclidean or hyperbolic planes (and surfaces), precisely because circle centers and geodesics are invariant under automorphisms and radii are invariant up to scale in Euclidean geometry and invariant in hyperbolic geometry. My belief is that using centers and radii of circles in inversive geometry should be avoided except where these can be used to simplify computations (as in the use of the spherical definition of inversive distance). The shift then is from inversive distance circle packings to inversive distance circle realizations. One is less concerned with possible underlying geodesic triangulations and more concerned with M\"obius-invariant quantities. For example, rather than working with a geodesic face formed by connecting the centers of three mutually adjacent circles, one is more interested in the existence of an ortho-circle, a circle mutually orthogonal to the three, which is a M\"obius invariant. Though the initial motivation was circle packing as reflected in Theorem~\ref{Theorem:Cauchy1}, the real interest has evolved to circle realizations as reflected in the more general version represented by Theorem~\ref{Theorem:Cauchy2}.

It turns out that Theorem~\ref{Theorem:Cauchy2} has implications for the rigidity of generalized hyperbolic polyhedra in $\mathbb{H}^{3}$. Thurston was the first to exploit this connection between circle configurations on $\mathbb{S}^{2}$ and hyperbolic polyhedra in $\mathbb{H}^{3}$ in really significant ways, and his observations inspired several avenues of clarification and generalization. It is to this that I turn in the penultimate section of this article.

\section{Polyhedra---From Steiner (1832) to Rivin (1996), and Beyond}\label{Section:Cage}\index{polyhedron}
In this section I survey the rich mathematical vein that has been mined in the geometric theory of polyhedra, particularly of three-dimensional hyperbolic polyhedra, that has its origins in Thurston's insights on using his circle packing theorem to characterize certain hyperbolic polyhedra. The initial observation of Thurston was that the study of polyhedra in hyperbolic three-space can be transferred to the study of overlapping circle packings in the two-sphere by realizing the Riemann sphere as the boundary of the Beltrami-Klein model of $\mathbb{H}^{3}$ sitting as the unit ball $\mathbb{B}^{3}$ in the real projective three-space. Theorems in one of these venues correspond to theorems in the other. Later Thurston's students, Oded Schramm and Igor Rivin, made great strides in the theory of both $3$-dimensional Euclidean and hyperbolic polyhedra, not so much using the techniques of circle packing but instead using very intricate and clever geometric arguments, often times in this classical setting of $\mathbb{H}^{3} \cong \mathbb{B}^{3} \subset  \mathbb{E}^{3} \subset \mathbb{RP}^{3}$. There is here a beautiful interplay among the classical geometries illustrating Arthur Cayley's aphorism that ``All geometry is projective geometry.'' Here one sees the Beltrami-Klein model of hyperbolic three-space as a sub-geometry of the real projective three-space, with its orientation-preserving isometry group naturally identified with the Lorentz group of Minkowski space-time, which itself restricts to the two-sphere boundary of hyperbolic space as the group of circle-preserving transformations of the two-sphere, the group of M\"obius transformations. This one geometry, the real projective geometry of dimension three, presents a playing field for studying three-dimensional polyhedra---classical Euclidean polyhedra, hyperbolic polyhedra of various types and generalizations, projective polyhedra, and circle polyhedra of M\"obius geometry. 

I will begin with an application of Thurston's circle packing theorem on using  polyhedra to cage a sphere, and move then to Schramm's generalization. From there I will discuss the characterization of certain hyperbolic polyhedra---compact by Hodgson and Rivin, ideal by Rivin, and hyper-ideal by Bao and Bonahon---and will finish with very recent work by Chen and Schlenker that characterizes those convex projective polyhedra all of whose vertices lie on the ideal boundary of hyperbolic space. I include a bonus final section on Cauchy's 1813 Rigidity Theorem for the reader who is approaching this subject as a novice. This is the fundamental theorem of rigidity theory, and the techniques and tools Cauchy developed have been used time and again in proofs of rigidity in the past two hundred years. Both Schramm and Rivin make use of Cauchy's toolbox in their theorems on convex hyperbolic and Euclidean polyhedra, as do Bao and Bonahon as well as Bowers, Pratt and the author. Before these recent developments, previous generations of mathematicians who delved into the study of polyhedra made use of Cauchy's toolbox---Dehn in his proof of infinitesimal rigidity, Aleksandrov in his rigidity results, Gluck in his examination of generic rigidity, and Connelly in various of his contributions.

\subsection{Caging eggs---Thurston and Schramm.} In 1832, Jakob Steiner~\cite{jS1832} asked 
\begin{quote}
In which cases does a convex polyhedron have a combinatorial equivalent which is inscribed in, or circumscribed about, a sphere?
\end{quote}\index{polyhedron!inscribed or circumscribed}
When a convex polyhedron $P$ is \textit{inscribed} in the sphere $S$ so that its vertices lie on $S$, then its polar dual \textit{circumscribes} the sphere $S$, so that each face of the dual $P^{*}$ meets $S$ in a single point. It wasn't until 1928 that Ernst Steinitz found families of non-inscribable polyhedral types with the example of a cube truncated at one vertex being the simplest. Marcel Berger~\cite{mB10} (p.~532) takes this long duration of time between Steiner and Steinitz as evidence that the theory of polyhedra in the years intervening had fallen into disrespect among mathematicians, being a subject of the old-fashioned mathematics of synthetic geometry.\footnote{Berger~\cite{mB10} uses the word \textit{disdain} to describe the prevailing opinion of the study of polyhedra.} One would be hard pressed to say that the study of polyhedra in the time between Steinitz and Thurston was anything but a curiosity to many a mathematician schooled in the rarified heights of abstraction that had captured the mathematical mind of the time. The sort of ``pedestrian geometry'' offered by the study of polyhedra captured the imagination of a select few. There has been a healthy development of the rigidity theory of polyhedra, notably by Aleksandrov in the nineteen-fifties, and Gluck and Connelly in the nineteen-seventies. Aleksandrov's work was largely ignored in the West until the nineteen-eighties. Coxeter had done truly foundational work in the combinatorial structure of polyhedra in the nineteen-forties and -fifties, and Victor Klee and Branko Gr\"unbaum began their foundational studies a bit later. Coxeter's work in geometry was routinely dismissed by much of mainstream mathematics as old-fashioned nineteenth century mathematics, uninteresting and pedestrian. Both Aleksandrov and Coxeter were ``rehabilitated'' by the larger community of geometers and topologists when their work of the forties and fifties---Aleksandrov's on metric geometry and Coxeter's on reflection groups---became important to the development of geometric group theory after Gromov's publication of his hyperbolic groups essay~\cite{mG87} in 1987. With apologies to Aleksandrov, Coxeter, Klee, and Gr\"unbaum, it has taken the attention of Thurston and his students Schramm and especially Rivin to resurrect more intense interest among topologists in this venerable old subject of classical geometry.\footnote{Gr\"unbaum~\cite{bG03} addresses the disinterest of the mathematical community in the combinatorial theory of polytopes in the preface to his book.}

Steinitz's basic tool for attacking the Steiner question is the following observation. Suppose the polyhedron $P$ circumscribes the sphere $S$. Let $e = uv$ be an edge of $P$ with adjacent faces $f$ and $g$. Since $P$ circumscribes $S$, the face $f$ is tangent to $S$ at a point $p$ and $g$ is tangent at a point $q$. Then the angle $\angle upv = \angle uqv$ in measure and we let $\Theta (e)$ denote this common value. It is immediate that summing these edge labels for the edges of any face yields an angle sum of $2\pi$. The reader might want to use this observation to see why a dodecahedron truncated at every vertex admits no inscribed sphere as there is no edge labeling $\Theta$ for this polyhedron that satisfies this property.

According to Steinitz then, the condition that an edge label $\Theta: E(P) \to (0,\pi)$ exists for the polyhedron $P$ whose sum for the edges of each face is $2\pi$ is a necessary condition that $P$ have a combinatorially equivalent realization that circumscribes a sphere, but it is not sufficient. It was not until Rivin's study of hyperbolic polyhedra in the late nineteen-eighties and early -nineties that a characterization of polyhedra of \textit{circumscribable type}, ones combinatorially equivalent to polyhedra that may circumscribe a sphere, was found. The definitive result is due to Rivin and reported in Hodgson, Rivin, and Smith~\cite{HRS92}, and follows from his characterization of ideal convex hyperbolic polyhedra that is presented in a later section.

\begin{CTC}[Rivin]
A polyhedron $P$ is of circumscribable type if and only if there exists a label $\Theta: E(P) \to (0,\pi)$ such that the sum of the labels $\Theta(e)$ as $e$ ranges over any circuit bounding a face is $2\pi$, while the sum as $e$ ranges over any simple circuit not bounding a face is strictly greater than $2\pi$.
\end{CTC}\label{CTC}
A polyhedron is of \textit{inscribable type} if it is combinatorially equivalent to one that may be inscribed in a sphere.
\begin{ITC}[Rivin]

A polyhedron $P$ is of inscribable type if and only if its dual $P^{*}$ is of circumscribable type.
\end{ITC}

The proofs will be discussed later, but first I want to generalize this discussion a bit. Inscription and circumscription are the respective cases, $m=0$ and $m=d-1$, of the question of whether a $d$-dimensional convex polytope has a realization in $\mathbb{E}^{d}$ each of whose $m$-dimensional faces meets a fixed $(d-1)$-dimensional sphere in a single point. One says that the polytope is \textit{$(m,d)$-scribable} in this case. Egon Shulte~\cite{eS85} proved in the mid-nineteen-eighties that when $0\leq m < d$ and $d > 2$, then there are combinatorial types of $d$-dimensional polytopes that are not $(m,d)$-scribable, except for the single exceptional case when $(m,d) = (1,3)$. The exceptional case then is when a convex polyhedron in $\mathbb{E}^{3}$ \textit{midscribes}\index{polyhedron!midscribe} a sphere $S$, so that each edge of $P$ is tangent to $S$, meeting $S$ in exactly one point.

In light of Shulte's result it perhaps is surprising that in his exceptional case, every convex polyhedron in $\mathbb{E}^{3}$ has a combinatorially equivalent realization that is midscribable about, say, the unit sphere $\mathbb{S}^{2}$. Thurston in Chapter 13 of GTTM states that this is a consequence of Andre'ev's theorems in~\cite{Andreev:1970a,Andreev:1970b}. The proof I give merely applies the Koebe-Andre'ev-Thurston Theorem to an appropriately edge-labeled graph.

\begin{MCP}[Thurston~\cite{Thurston:1980}]
Every convex polyhedron in $\mathbb{E}^{3}$ has a combinatorially equivalent realization that is midscribable about the unit sphere $\mathbb{S}^{2}$. Considering $\mathbb{E}^{3} \subset \mathbb{RP}^{3}$, any such realization is unique up to projective transformations of $\mathbb{RP}^{3}$ that set-wise fix the unit sphere $\mathbb{S}^{2}$.\footnote{The projective transformations that  fix $\mathbb{S}^{2}$ act as M\"obius transformations on $\mathbb{S}^{2}$.}
\end{MCP}

\begin{figure}
\includegraphics[width = 0.7\textwidth]{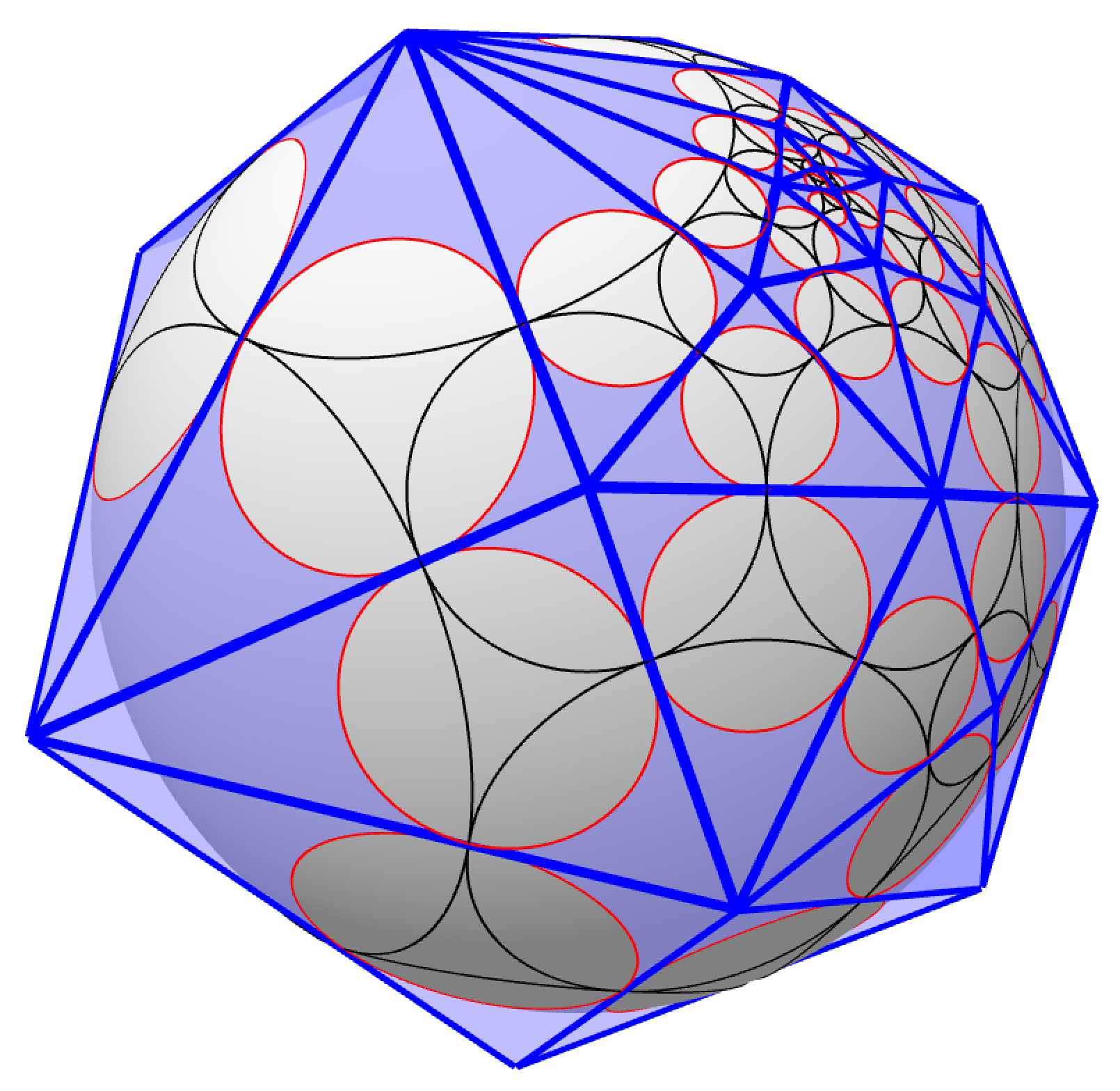}
\caption{A midscribed polyhedron. Each face meets the sphere $\mathbb{S}^{2}$ in a red circle and each vertex is the cone point of a black circle. Each edge $e^{*}$ meets $\mathbb{S}^{2}$ in exactly one point, at the intersection of the two red circles determined by the faces incident to $e^{*}$, or at the intersection of the two black circles determined by the endpoints of $e^{*}$.}
\label{Figure:Midscribe}
\end{figure}

\begin{proof}
Let $P$ be a convex polyhedron in $\mathbb{E}^{3}$ and let $K$ be the simplicial $2$-complex obtained by adding a vertex to each open face of $P$ and starring to the vertices. Precisely, the vertices of $K$ are those of $P$ along with a new vertex $v_{f}$ for each face $f$ of $P$. The edges are the edges of $P$ along with edges of the form $vv_{f}$, where $v$ is a vertex of $f$. The faces are the $2$-simplices of the form $uvv_{f}$ where $uv$ is an edge of $f$. Write the edge-set of $K$ as $E(K) = E(P) \cup E'$, where $E'$ are the new edges of the form $vv_{f}$. Define an angle map $\Phi : E(K) \to [0,\pi/2]$ by $\Phi(e) = 0$ when $e\in E(P)$ and $\Phi(e) = \pi/2$ when $e \in E'$. An application of the Koebe-Andre'ev-Thurston Theorem I produces a circle packing $K(\mathcal{C})$ on the $2$-sphere $\mathbb{S}^{2}$  and a geodesic triangulation in the pattern of $K$ with overlap angles of adjacent circles given by $\Phi$. For each face $f$ of $P$, let $H_{f}$ be the half-space in $\mathbb{E}^{3}$ that meets all the circles of $K(\mathcal{C})$ and whose bounding plane $\partial H_{f}$ contains $C_{v_{f}}$. My claim is that the convex polyhedron $Q = \cap_{f \in F(P)} H_{f}$ midscribes $\mathbb{S}^{2}$ and is combinatorially equivalent to $P$.

For any vertex $v$ of $P$, let $v^{*}$ be the apex of the cone in $\mathbb{E}^{3}$ that is tangent to $\mathbb{S}^{2}$ along the circle $C_{v}$, and when $e = uv$ is an edge of $P$, let $e^{*} = u^{*}v^{*}$ be the segment with endpoints $u^{*}$ and $v^{*}$. Let $f$ be a face of the polyhedron $P$ with vertices $v_{1}, \dots, v_n$ written in cyclic order. Since the circle $C_{v_{f}}$ is orthogonal to the circles $C_{v_{i}}$, the apexes $v_{i}^{*}$ all lie on the bounding plane $\partial H_{f}$, for $i=1, \dots, n$. Let $f^{*}$ denote the convex hull of the points $v_{1}^{*}, \dots, v_{n}^{*}$ in $\partial H_{f}$. A moment's thought should convince the reader that the convex polyhedron $Q$ may be described as the convex hull of the set $V(Q) = \{ v^{*} : v\in V(P)\}$. It follows that the vertex set of $Q$ is $V(Q)$, edge set is $E(Q) = \{ e^{*} : e\in E(P)\}$, and face set is $F(Q) = \{ f^{*} : f \in F(P) \}$. This verifies that $P$ and $Q$ are combinatorially equivalent. Moreover, the edge $e^{*} = u^{*} v^{*}$ is tangent to the sphere $\mathbb{S}^{2}$ at the point of intersection of the circles $C_{u}$ and $C_{v}$, which are tangent since $\Phi (e) = 0$. Hence $Q$ midscribes the sphere $\mathbb{S}^{2}$. See~\Cref{Figure:Midscribe}.

Uniqueness up to projective transformations that fix the unit sphere follows from the M\"obius uniqueness of the circle packing $K(\mathcal{C})$ with edge angle data $\Phi$ and the fact that the M\"obius group extends its action on $\mathbb{S}^{2}$ to a projective action of $\mathbb{RP}^{3}$ set-wise fixing $\mathbb{S}^{2}$.
\end{proof}

Schulte introduced in~\cite{eS85} the question of whether the sphere can be replaced by other convex bodies. Schramm~\cite{oS91b} proved that when the convex polyhedron $P$ is simplicial, then for any smooth convex body $S$, a combinatorially equivalent polyhedron $Q$ exists that midscribes $S$. Of course this means that each edge of $Q$ is tangent to the boundary $\partial S$. Shortly thereafter, Schramm improved his result by removing the requirement that $P$ be simplicial. A convex body $S$ is \textit{strictly convex} if its boundary contains no non-degenerate line segment, and is \textit{smooth} if each point of the boundary has a unique supporting plane. This latter condition is equivalent to the boundary being $C^{1}$-smooth. Schramm's definitive result on midscription is the main theorem of his \textsl{Inventiones} article~\cite{oS92} whimsically entitled \textit{How to cage an egg}.

\begin{CBM}[Schramm~\cite{oS92}]
	Let $P$ be a convex polyhedron and $S$ a smooth strictly convex body in $\mathbb{E}^{3}$. Then there exists a convex polyhedron $Q$ combinatorially equivalent to $P$ that midscribes $S$. 
\end{CBM}

\begin{proof}[Discussion of Proof]
The proof is rather involved and so I am content to give the briefest of indication of its method. Schramm defines the \textit{configuration space} $\mathcal{Z} = (\mathbb{E}^{3})^{V(P)} \times \mathrm{G}(2,3)^{F(P)}$, where $V(P)$ and $F(P)$ are the respective sets of vertices and faces of $P$, and $\mathrm{G}(2,3)$ is the manifold of oriented affine planes in $\mathbb{E}^{3}$. In this way $P$ is identified with a single point of $\mathcal{Z}$, and the combinatorial type of $P$ defines a submanifold ${\mathcal{Z}}_{P}$ of $\mathcal{Z}$ corresponding to various convex polyhedra in $\mathbb{E}^{3}$ that are combinatorially equivalent to $P$. Schramm then shows that there is a $C^{2}$ convex body $S_{0}$ with positively curved boundary that $P$ midscribes. Let $S_{t}$, $0\leq t \leq 1$, be a $C^{2}$-path of convex bodies with positively curved boundaries with $S_{1} = S$. The idea now is to flow $S_{0}$ to $S_{1}$ along this path and drag combinatorial realizations of $P$ along as midscribing polyhedra. The proof relies on a fine analysis of the configuration space $\mathcal{Z}$ and its submanifold ${\mathcal{Z}}_{P}$, and the method is to show that when $S_{t}$ is midscribed by a realization of $P$, then so is $S_{t'}$ for all $t'$ in an open interval about $t$. Then a delicate argument shows also that the set of parameter values for which $S_{t}$ is midscribable by a realization of $P$ is a closed set. Being open and closed, and nonempty since $P$ midscribes $S_{0}$, this set of parameter values must be the whole of the unit interval, hence $S=S_{1}$ is midscribed by a realization of $P$.
\end{proof}

The remaining discussion on hyperbolic polyhedra has little to do, at least directly, with the Koebe-Andre'ev-Thurston Theorem. The arguments tend to be clever and technical, but ultimately involve the elementary geometry of hyperbolic space, often times realized as the unit ball in projective $3$-space where the machinery of the Minkowski inner product and of de Sitter space is available. I include the discussion in order to complete for the reader the current state of affairs in the study of convex hyperbolic polyhedra, a study which I view as having been revitalized by Thurston's articulation of KAT I and pushed forward into the broader mathematical consciousness by the seminal work of Thurston's students, Oded Schramm and especially Igor Rivin.

\subsection{Compact and convex hyperbolic polyhedra---Hodgson and Rivin.}\index{polyhedron!compact and convex hyperbolic} In his doctoral thesis of 1986, Igor Rivin studied convex hyperbolic polyhedra. Therein he gave a characterization of compact, convex hyperbolic polyhedra that generalizes the Andre'ev results of~\cite{Andreev:1970a}\footnote{See Roeder, Hubbard, and Dunbar's paper \cite{Roeder2007} for a readable proof of Andre'ev's classification of compact hyperbolic polyhedra with non-obtuse exterior dihedral angles.}, and in articles in the early nineteen-nineties, extended his characterization to ideal polyhedra, generalizing Andre'ev's results in~\cite{Andreev:1970b}. He used this latter generalization to answer definitively Steiner's question of 1832 asking for a characterization of those polyhedra that circumscribe a sphere. This of course is the content of the Circumscribable Type Characterization Theorem of the preceding section. In this section, I present an overview of Rivin's characterization of compact and convex hyperbolic polyhedra in terms of a generalized Gauss map. The overview embellishes Hodgson's outline presented in~\cite{cH93} (and repeated in ~\cite{HodgsonRivin:1993}). In the section following, I outline Rivin's characterization of ideal polyhedra and make his observation that the Circumscribable Type Characterization Theorem is an immediate corollary of his characterization of ideal polyhedra.

To lay the groundwork, let's review the \textit{Gauss map} $\mathtt{G}$ of a compact and convex Euclidean polyhedron $P$ to the unit sphere $\mathbb{S}^{2}$. This is a set-valued map from the $2$-complex forming the boundary of $P$ that assigns to the point $p$ of $\partial P$ the set of outward pointing unit normals to support planes to $P$ at $p$. Thus when $p$ is a point of an open face $f$, $\mathtt{G}(p) = \mathtt{G}(f)$ is a single point determined by the outward unit normal to $f$. When $p$ is in the open edge $e$ incident to faces $f$ and $g$, $\mathtt{G}(p) = \mathtt{G}(e)$ is the great circular arc connecting $\mathtt{G}(f)$ to $\mathtt{G}(g)$ of length equal to the exterior dihedral angle between $f$ and $g$. Finally, for a vertex $p$ of $P$, $\texttt{G}(p)$ is the convex spherical polygon bounded by the arcs $\mathtt{G}(e)$ for edges $e$ incident with $p$. When edges $e$ and $e'$ of the face $f$ are incident at $p$, the interior angle of the polygon $\mathtt{G}(p)$ at the vertex $\mathtt{G}(f)$ is $\pi - \alpha$, where $\alpha$ is the interior angle of the face $f$ at $p$. In this way the Gauss map realizes the Poincar\'e dual $P^{*}$ of $P$ as a geodesic cellular decomposition of the $2$-sphere $\mathbb{S}^{2}$. Notice that the Gauss map does not encode all the information needed to reconstruct the polyhedron $P$. It encodes the interior angles of all the faces and the dihedral angles of all adjacent faces, but there is no encoding of side-lengths of the edges of $P$. For example, all rectangular boxes have the same image under the Gauss map, namely, a regular right-angled octahedral decomposition of the sphere $\mathbb{S}^{2}$. 

%This construction of the Poincar\'e dual $P^{*}$ as a cellular decomposition of the $2$-sphere using the Gauss map does not readily apply to convex hyperbolic polyhedra, though Rivin manages to define a generalized Gauss map that works. I will present Rivin's generalization below, but first I present another way to construct the Poincar\'e dual on $\mathbb{S}^{2}$. 

Another way to describe the convex spherical polygon $\mathtt{G}(p)$ for a vertex $p$ of $P$ is as the polar dual $\mathtt{L}^{*}(p)$ of the infinitesimal link $\mathtt{L}(p)$ of $p$ in $P$.\footnote{For a Euclidean polyhedron, $\mathtt{L}(p)$ is the intersection of $P$ with a small sphere centered at $p$, one whose radius is smaller than the lengths of edges incident with $p$, rescaled to unit radius, and is oriented so that its interior is ``to the left'' as one traverses the polygon in its positive direction.} Note that $\mathtt{L}(p)$ is a convex spherical polygon with internal angles equal to the dihedral angles of the faces of $P$ incident with $p$, and edge-lengths equal to the internal angles at the vertex $p$ in the faces of $P$ incident with $p$. Recall that an oriented great circle in $\mathbb{S}^{2}$ and its spherical center are polar duals of one another. The polar dual $\mathtt{L}^{*}(p)$ is obtained by replacing the edges of $\mathtt{L}(p)$ by the polar dual centers of their supporting great circles, and the vertices by appropriate arcs of the polar dual great circles. A nice exercise in spherical geometry verifies that $\mathtt{L}^{*}(p)$ is isometric to $\mathtt{G}(p)$. This gives an alternate construction of the Poincar\'e dual $P^{*}$ as a geodesic, cellular decomposition of the $2$-sphere---just isometrically glue the polar duals $\mathtt{L}^{*}(p)$ together as $p$ ranges over the vertices of $P$ along corresponding edges, $\mathtt{L}^{*}(p)$ glued to $\mathtt{L}^{*}(q)$ whenever $pq$ is an edge of $P$.\footnote{The edge $pq$ determines respective vertices $u$ and $v$ of $\mathtt{L}(p)$ and $\mathtt{L}(q)$ whose respective polar edges $u^{*}$ and $v^{*}$ have the same lengths, namely the exterior dihedral angle of $P$ at edge $pq$.} Obviously this gluing produces a $2$-sphere, not only homeomorphic, but also  isometric to the standard $2$-sphere $\mathbb{S}^{2}$, and reproduces the cellular decomposition determined by the Gauss map.

It is this latter construction of the Poincar\'e dual $P^{*}$ as a cellular decomposition of the $2$-sphere that readily generalizes to convex and compact hyperbolic polyhedra. Indeed, let $P$ now be a convex and compact hyperbolic polyhedron in $\mathbb{H}^{3}$ and for each vertex $p$, let $\mathtt{L}^{*}(p)$ be the polar dual of the infinitesimal link $\mathtt{L}(p)$ of $p$ in $P$.\footnote{This is the link in the tangent space of $\mathbb{H}^{3}$ of the pre-image of the intersection of $P$ with a small neighborhood of $p$ under the exponential map.} The link $\mathtt{L}(p)$, as in the Euclidean case, is an oriented convex spherical polygon in $\mathbb{S}^{2}$ with internal angles equal to the dihedral angles of the faces of $P$ incident with $p$, and edge-lengths equal to the internal angles at the vertex $p$ in the faces of $P$ incident with $p$. The polar dual $\mathtt{L}^{*}(p)$ then encodes the exterior dihedral angles at the edges of $P$ incident with $p$ as the lengths of its edges, and the interior angles $\alpha$ of the faces incident with $p$ as its interior angles in the form $\pi - \alpha$. This construction acts as a local Gauss map in a small neighborhood of the vertex $p$. Now exactly as before, isometrically glue the polar duals $\mathtt{L}^{*}(p)$ together as $p$ ranges over the vertices of $P$ along corresponding edges. The result is again a $2$-sphere topologically, which is called the \textit{Gaussian image} of $P$ and denoted as $\mathtt{G}(P)$, with a spherical metric of constant unit curvature, except at the vertices. The vertices have cone type singularities with concentrated negative curvature. Indeed, at the vertex corresponding to the face $f = p_{1} \cdots p_{n}$ of $P$, the angle sum is $\theta (f) = n\pi - \sum_{i=1}^{n} \alpha_{i}$, where $\alpha_{i}$ is the internal angle of $f$ at the vertex $p_{i}$. In the hyperbolic plane, the compact and convex polygon $f$ always has interior angle sum strictly less than $(n-2) \pi$ so that $\theta (f) > 2\pi$.

This brings us to Rivin's characterization of compact and convex hyperbolic polyhedra.

\begin{CCHPC}[Rivin]
A metric space $(M, g)$ homeomorphic to $\mathbb{S}^{2}$ can arise as the Gaussian image $\mathtt{G}(P)$ of a compact and convex polyhedron $P$ in $\mathbb{H}^{3}$  if and only if these three conditions adhere.
\begin{enumerate}
\item[(i)] The metric $g$ has constant curvature $+1$ except at a finite number of cone points.
\item[(ii)] The cone angle\ at each cone point is greater than $2\pi$.
\item[(iii)] The lengths of the nontrivial closed geodesics of $(M,g)$ are all strictly greater than $2\pi$.
\end{enumerate}
Moreover, the metric $g$ determines $P$ uniquely up to hyperbolic congruence.
\end{CCHPC}

Recall that the Gauss map does not determine Euclidean polyhedra up to congruence since it contains no information about side lengths. In contrast, a hyperbolic polyhedron is determined up to a global hyperbolic isometry by its Gaussian image. The proof of this uniqueness uses Cauchy's toolbox that is reviewed in Addendum~\ref{Section:CRT}, wherein I recall the tools Cauchy used to prove his celebrated rigidity theorem of 1813. The necessity of items (i) and (ii) follows from the previous discussion and that of (iii) uses the fact that the total geodedic curvature of a non-trivial closed hyperbolic space curve is greater than $2\pi$, a hyperbolic version of Fenchel's Theorem on Euclidean space curves. The proof of sufficiency is based on Aleksandrov's Invariance of Domain method used in his study of Euclidean polyhedra in~\cite{aA05}.

Rivin also uses Cauchy's toolbox to prove this rather interesting theorem that illustrates again the enhanced rigidity of hyperbolic polyhedra vis-\`a-vis Euclidean ones.

\begin{FAR}[Rivin]
The face angles of a compact and convex polyhedron in $\mathbb{H}^{3}$ determine it up to congruence.
\end{FAR}

The characterization of compact and convex hyperbolic polyhedra in terms of the Gaussian image surveyed here suffers from the same defect as Aleksandrov's characterization of compact and convex Euclidean polyhedra. Both characterizations posit a singular positively curved metric on a $2$-sphere, but neither provides a way to decode from this metric space $(M,g)$ the combinatorial type of the polyhedron $P$ encoded in $(M,g)$. The proof is not constructive, but depends on a topological analysis within the space of admissible metrics on the $2$-sphere satisfying the three conditions of the characterization and yields, finally, the abstract fact of existence of an appropriate polyhedron, without describing its combinatorial type.

\subsection{Convex ideal hyperbolic polyhedra---Rivin}\index{polyhedron!ideal}
Rivin turns his attention to convex ideal polyhedra in $\mathbb{H}^{3}$ in~\cite{Rivin:1996} where he gives a full characterization in terms of exterior dihedral angles. The characterization begins with an analysis of the exterior dihedral angles of such a polyhedron reported in~\cite{HRS92} with details in~\cite{Rivin93} that goes as follows. Label each edge $e^{*}$ of the polyhedron $P^{*}$ dual to the ideal convex polyhedron $P$ by the exterior dihedral angle $\theta(e^{*})$ of the corresponding edge $e$ of $P$. Rivin's argument that these labels satisfy the following conditions is reproduced in the next two theorems.
\begin{enumerate}
\item[(i)] $0 < \theta(e^{*}) <\pi$ for all edges $e$ of $P$.
\item[(ii)] If the edges $e_{1}^{*}$, \dots , $e_{n}^{*}$ are the edges bounding a face of $P^{*}$, then $\theta(e_{1}^{*}) + \cdots  + \theta(e_{n}^{*}) = 2\pi$.
\item[(iii)] If $e_{1}^{*}$, \dots , $e_{n}^{*}$ forms a simple nontrivial circuit that does not bound a face of $P^{*}$, then $\theta(e_{1}^{*}) + \cdots  + \theta(e_{n}^{*}) > 2\pi$.
\end{enumerate}
Compare these conditions with the hypotheses of the Circumscribable Type Characterization on page~\pageref{CTC}. Now Condition (i) is a requirement of convexity and Condition (ii) is seen easily in the upper-half-space model by placing one of the ideal vertices $v$ of $P$ at infinity and observing that the link of $v$ is a convex Euclidean polygon. Indeed, the faces incident with $v$ lie on vertical Euclidean planes whose intersections with the $xy$-plane cut out a convex Euclidean polygon $L(v)$, and quite easily the sum $\theta(e_{1}^{*}) + \cdots  + \theta(e_{n}^{*})$ is precisely the sum of the turning angles of $L(v)$. Condition (iii) is a consequence of the following discrete, hyperbolic version of Fenchel's Theorem, in this case for closed polygonal curves in $\mathbb{H}^{3}$.

\begin{DTCHC}[Rivin~\cite{Rivin:1996}]
The total discrete geodesic curvature of a closed, polygonal, hyperbolic space curve is greater than $2\pi$, unless the vertices are collinear, in which case the total curvature is $2\pi$.
\end{DTCHC}

\begin{proof}
The total discrete geodesic curvature of the polygonal hyperbolic space curve $\gamma$ with vertices $p_{1}, \dots, p_{k} , p_{k+1} = p_{1}$ is $\sum_{i=1}^{k} \alpha_{i}$, where $\alpha_{i}$ is the turning angle of $\gamma$ at $p_{i}$. The angle $\alpha_{i}$ is just the exterior angle at $p_{i}$ of the triangle $\tau_{i} = p_{i-1}p_{i}p_{i+1}$. For $2\leq i \leq k-1$, let $T_{i}$ be the triangle $T_{i} = p_{1}p_{i}p_{i+1}$ with internal angles $a_{i}$, $b_{i}$, and $c_{i}$ at the respective vertices $p_{1}$, $p_{i}$, and $p_{i+1}$. Note that by considering the triangles $\tau_{i}$, $T_{i-1}$ and $T_{i}$ with common vertex $p_{i}$, the spherical triangle inequality gives
\begin{equation*}
c_{i-1} + b_{i} \geq \pi - \alpha_{i} \quad \text{for} \quad 3\leq i \leq k-1,
\end{equation*}
and
\begin{equation*}
b_{2} = \pi -\alpha_{2}, \quad c_{k-1} = \pi - \alpha_{k}, \quad \text{and} \quad\sum_{i=2}^{k-1} a_{i} \geq \pi - \alpha_{1}. 
\end{equation*}
Recalling that $\pi \geq a_{i} + b_{i} + c_{i}$ with equality only when $p_{1}$, $p_{i}$, and $p_{i+1}$ are collinear, and then summing, one has 
\begin{equation*}
(k-2)\pi \geq \sum_{i=2}^{k-1} (a_{i} + b_{i} + c_{i} ) \geq  k\pi - \sum_{i=1}^{k} \alpha_{i},
\end{equation*}
with equality only when $p_{1}, \dots, p_{k}$ are collinear.  
\end{proof}

\begin{Theorem}[Rivin~\cite{Rivin93}]\label{Theorem:Ideal}
 	The edge label $\theta(e^{*})$ of the polyhedron $P^{*}$ dual to the ideal convex polyhedron $P$ defined above satisfies Conditions {\em (i)--(iii)}.
\end{Theorem}

\begin{proof}
Conditions (i) and (ii) already are verified. For Condition (iii), the circuit $e_{1}^{*}$, \dots , $e_{n}^{*}$ that does not bound a face of $P^{*}$ corresponds to a chain of contiguous faces $f_{1}, \dots, f_{n}$ in $P$ with $f_{i} \cap f_{i+1} = e_{i}$. $F = \cup_{i=1}^{n} f_{i}$ is a hyperbolic surface with boundary and cusps, and can be completed by extending geodesically across the boundary components to a complete immersed surface $\widetilde{F}$ in $\mathbb{H}^{3}$ without boundary. The surface $\widetilde{F}$ is an immersed hyperbolic cylinder with both ends of infinite-area. This observation uses the fact that the circuit $e_{1}^{*}$, \dots , $e_{n}^{*}$ does not bound a face of $P^{*}$. Let $\gamma$ be the unique closed geodesic path on the surface $\widetilde{F}$ that is freely homotopic to the meridian. The curve $\gamma$ is immersed in $\mathbb{H}^{3}$ as a polygonal curve lying on $\widetilde{F}$ with turning angles at the edges $e_{i}$. But it is easy to see that the turning angle of $\gamma$ at edge $e_{i}$ is no more than the exterior dihedral angle of the faces $f_{i}$ and $f_{i+1}$ that meet along $e_{i}$. This implies that the sum, $\theta(e_{1}^{*}) + \cdots  + \theta(e_{n}^{*})$, which is the sum of these dihedral angles, is at least as large as the discrete geodesic curvature of $\gamma$, which in turn is greater than $2\pi$ by an application of the preceding theorem.
\end{proof}
Rivin was able to turn this around and prove a converse to the theorem, which gives the following characterization of convex, ideal hyperbolic polyhedra. The existence is proved in~\cite{Rivin:1996}, uniqueness in~\cite{Rivin:1994kg}, and necessity of the three conditions in~\cite{Rivin93}.

\begin{CIHP}[Rivin~\cite{Rivin:1996}]
Let $P^{*}$ be an abstract polyhedron. Then for any label $\theta: E(P^{*}) \to (0, \pi)$ that satisfies Conditions {\em (i)--(iii)}, there is a convex, ideal hyperbolic polyhedron $P$ in $\mathbb{H}^{3}$ whose Poincar\'e dual is $P^{*}$, and whose exterior dihedral angles at edges $e$ are given by the values $\theta(e^{*})$. Moreover, $P$ is unique up to hyperbolic congruence. Conversely, every such polyhedron $P$ satisfies Conditions {\em (i)--(iii)} as shown in Theorem~\ref{Theorem:Ideal}.
\end{CIHP}

This characterization also proves the Circumscribable and Inscribable Type Characterizations, answering Steiner's question of 1832. This is because a convex, ideal hyperbolic polyhedron in the Beltrami-Klein projective model of $\mathbb{H}^{3}$ is represented by a convex Euclidean polyhedron inscribed in the $2$-sphere $\mathbb{S}^{2}$.

Since Rivin's work of the nineteen-nineties, several topologists and geometers have taken up the mantel and continued to unearth these beautiful gems of discrete geometry. I'll close this survey with the mention of two examples in the next section, the first from the first decade of the new century, and the second of very recent origin.

\subsection{New millennium excavations}
Space constraints forbid too much further development of the topic, but I would be remiss if I didn't mention at least these two beautiful theorems, the first characterizing convex hyperideal hyperbolic polyhedra by Bao and Bonahon, and the second giving a complete answer to Steiner's original question when interpreted as broadly as possible, this time by Chen and Schlenker. I develop just enough of these topics to state the main results, and leave the interested reader the task of perusing the original articles for details of the proofs.

\subsubsection{Hyperideal polyhedra---Bao and Bonahon}\index{polyhedron!hyperideal}
A \textit{hyperideal} polyhedron in $\mathbb{H}^{3}$ is a non-compact polyhedron that may be described most easily in the Beltrami-Klein projective model $\mathbb{H}^{3} = \mathbb{B}^{3} \subset \mathbb{RP}^{3}$ as the intersection with $\mathbb{B}^{3}$ of a projective polyhedron all of whose vertices lie outside of $\mathbb{B}^{3}$ while each edge meets $\mathbb{B}^{3}$. Bao and Bonahon~\cite{BaoBonahon:2002} classify hyperideal polyhedra up to hyperbolic congruence in terms of their dihedral angles and combinatorial type in much the same vein as Rivin's classification of ideal hyperbolic polyhedra. Note that Bao and Bonahon do allow for the vertices to lie on the sphere $\mathbb{S}^{2} = \partial \mathbb{B}^{3}$ and hence their characterization reduces to Rivin's for ideal polyhedra.

I will state the characterization in terms of conditions on the $1$-skeletal graph of the dual polyhedron using Steinitz's famous characterization of those graphs that may serve as the dual graph of a convex polyhedron in $\mathbb{E}^{3}$ as precisely the planar, $3$-connected graphs. 

\begin{CCHP}[Bao and Bonahon~\cite{BaoBonahon:2002}]
Let $\mathcal{G}$ be a $3$-connected graph embedded in $\mathbb{S}^{2}$ and $\theta: E(\mathcal{G}) \to (0, \pi)$. There is a hyperideal polyhedron $P$ in $\mathbb{H}^{3}$ with dual graph isomorphic with $\mathcal{G}$ and exterior dihedral angles given by $\theta$ if and only if the following conditions are satisfied.
\begin{enumerate}
\item[(i)] If $e_{1}$, \dots , $e_{n}$ forms a simple nontrivial circuit of edges of $\mathcal{G}$, then $\theta(e_{1}) + \cdots  + \theta(e_{n}) \geq 2\pi$, with equality possible only if $e_{1}, \dots, e_{n}$ bounds a component of $\mathbb{S}^{2} - \mathcal{G}$.
\item[(ii)] If $\gamma = e_{1}$, \dots , $e_{n}$ forms a simple path of edges of $\mathcal{G}$ that connects two vertices of $\mathcal{G}$ that lie in the closure of a component $C$ of $\mathbb{S}^{2} - \mathcal{G}$, but $\gamma$ does not lie in the boundary of $C$, then $\theta(e_{1}) + \cdots  + \theta(e_{n}) > \pi$.
\end{enumerate}
Moreover if $P'$ is the projective polyhedron with $P' \cap \mathbb{H}^{3} = P$, a vertex $v$ of $P'$ is located on the sphere at infinity of $\mathbb{H}^{3}$ if and only if equality holds in Condition {\em (i)} for the boundary of the corresponding component of $\mathbb{S}^{2} - \mathcal{G}$.

Finally, the hyperideal polyhedron $P$ is unique up to hyperbolic congruence.
\end{CCHP}

I should mention that Hodgson and Rivin's~\cite{HodgsonRivin:1993} characterization of compact and convex hyperbolic polyhedra can be applied to appropriate truncated polyhedra associated with those hyperideal polyhedra for which no vertex lies on the sphere at infinity to characterize them.

Define a \textit{strictly hyperideal} polyhedron to be the intersection of $\mathbb{B}^{3}$ with a projective polyhedron $P$ all of whose vertices lie outside the closed unit ball $\overline{\mathbb{B}}^{3} = \mathbb{B}^{3} \cup \mathbb{S}^{2}$, but all of whose faces meet $\mathbb{B}^{3}$. Note that this definition allows that an edge of $P$ may lie entirely outside the closed ball $\overline{\mathbb{B}}^{3}$. These are yet to be characterized, but I mention that the article~\cite{BBP18} verifies the rigidity of these that are bounded and convex, as long as no edges are tangent to the unit sphere. The proof again uses Cauchy's toolbox.

\subsubsection{Weakly inscribed polyhedra---Chen and Schlenker} Recall Steiner's question of which polyhedra inscribe or circumscribe a sphere that Rivin answered. A more faithful translation of Steiner's question from the German is ``Does every polyhedron have a combinatorially equivalent realization that is inscribed or circumscribed to a sphere, or to another quadratic surface? If not, which polyhedra have such realizations?'' He includes the definition that ``A polyhedron $P$ is \textit{inscribed to} a quadratic surface $S$ if all the vertices of $P$ lie on $S$,'' and further defines that $P$ is \textit{circumscribed to} $S$ if all of its facets are tangent to $S$. As before I will concentrate on inscription since polarity relates circumscription to inscription. In the very recent preprint~\cite{CS17}, Chen and Schlenker point out that the apparent grammar mistake---inscribed \textit{to} instead of \textit{in} $S$---makes a significant distinction.

Generally Steiner's question has been interpreted to ask about inscription of the polyhedron $P$ to a quadratic surface $S$ in Euclidean space $\mathbb{E}^{3}$, and in this setting $P$ is contained in the bounded component of the complement of $S$, i.e., $P$ is ``inside'' $S$, hence the change from inscribed ``to'' to ``in''. But Steiner's question makes sense in projective space as well, and in this setting a polyhedron may be inscribed to a surface without being inscribed in the surface. To be a bit more illustrative, consider the unit sphere $\mathbb{S}^{2}$ sitting in $\mathbb{E}^{3} \subset \mathbb{RP}^{3}$. Now $\mathbb{S}^{2}$ usually is thought of as the boundary of the open unit ball $\mathbb{B}^{3}$ that serves as the projective model of hyperbolic space, and this is what Rivin exploited in his characterization of those polyhedra inscribable in $\mathbb{S}^{2}$. But $\mathbb{S}^{2}$ is also the boundary of the complement $\mathbb{RP}^{3} - \overline{\mathbb{B}}^{3}$, which has a complete metric making it into a model of de Sitter space $d\mathbb{S}^{3}$. In this setting a projective polyhedron may have its vertices on the sphere $\mathbb{S}^{2}$ and yet not lie entirely in the ball $\mathbb{B}^{3}$ so that it is inscribed to $\mathbb{S}^{2}$, but not inscribed in $\mathbb{S}^{2}$ in the usual meaning. Following Chen and Schlenker, I will revise Steiner's terminology to emphasize the difference between \textit{inscribed in} and \textit{inscribed to but not in}.

\begin{Definition}[\textsc{strong and weak inscription}]
In the real projective space $\mathbb{RP}^{3}$, a polyhedron $P$ inscribed to a quadratic surface $S$ is \textit{strongly inscribed in} $S$ if the interior of $P$ is disjoint from $S$, and \textit{weakly inscribed to} $S$ otherwise.
\end{Definition}

Before presenting a characterization of those polyhedra weakly inscribed to a sphere in $\mathbb{RP}^{3}$, allow a word about polyhedra inscribed to other quadratic surfaces. This topic has been neglected until rather recently. There are only three quadratic surfaces in $\mathbb{RP}^{3}$ up to projective transformations, and these are the sphere, the one-sheeted hyperboloid, and the cylinder. Danciger, Maloni, and Schlenker in~\cite{DMS14} characterized the combinatorial types of polyhedra that are strongly inscribable in a one-sheeted hyperboloid or in a cylinder, and of course Rivin takes care of those strongly inscribable in a sphere. Chen and Schlenker's work reported here characterizes those polyhedra weakly inscribable to a sphere, and the characterization of those weakly inscribable to the remaining two quadratic surfaces is the subject of current research by Chen and Schlenker. 

\begin{WIC}[Chen and Schlenker~\cite{CS17}]
	A $3$-connected planar graph $\Gamma$ is the $1$-skeleton of a polyhedron $P\subset \mathbb{RP}^{3}$ weakly inscribed to a sphere if and only if $\Gamma$ admits a vertex-disjoint cycle cover by two cycles $C_{1}$ and $C_{2}$ with the following property. Color edge $uv$ red if $u$ and $v$ both belong to $C_{1}$ or both belong to $C_{2}$, and color it blue otherwise. Then there is a weight function $w : E(\Gamma) \to \mathbb{R}$ such that
	\begin{enumerate}
		\item[(i)] $w > 0$ on red edges and $w< 0$ on blue ones;
		\item[(ii)] $w$ sums to $0$ over the edges adjacent to a vertex $v$, unless $v$ is the only vertex on $C_{1}$ or $C_{2}$ (trivial cycle), in which case $w$ sums to $-2\pi$ over the edges adjacent to $v$.
	\end{enumerate}
\end{WIC}

I end this survey of progress in the characterization of polyhedra since Thurston's observation that every polyhedron type in $\mathbb{E}^{3}$ has a realization that midscribes a sphere with a description of the original rigidity theorem of Cauchy that is so instrumental in many of the proofs of the results surveyed here.

\subsection{Addendum:~Cauchy's toolbox} %1813 Rigidity Theorem}
\label{Section:CRT}\index{polyhedron!Cauchy rigidity}
In this bonus section I review Cauchy's celebrated rigidity theorem~\cite{Cauchy1813} of 1813 on the uniqueness of convex, bounded polyhedra in $\mathbb{E}^{3}$. The theorem concerns two convex polyhedra with equivalent combinatorics and with corresponding faces congruent. Cauchy's Rigidity Theorem states that the two polyhedra must be congruent globally, meaning that there is a Euclidean isometry of the whole of $\mathbb{E}^{3}$ mapping one to the other. Like many of the great theorems of mathematics, the proof is of more importance than the theorem itself. As stated earlier in the introduction to this section, the toolbox Cauchy developed has been instrumental in the past two hundred year development of the theory of polyhedra, especially in its rigidity theory. The proof, though at places clever and even subtle, overall is rather straightforward with a simplicity that belies its importance.  

Cauchy's proof has two components---the one geometric and the other combinatorial. The geometric component is the Discrete Four Vertex Lemma, which follows from an application of Cauchy's Arm Lemma. Denote a convex planar or spherical polygon $P$ merely by listing its vertices in cyclic order, say as $P = p_{1} \dots p_{n}$. The Euclidean or spherical length of the side $p_{i}p_{i+1}$ is denoted as $|p_{i}p_{i+1}|$ and the interior angle at $p_{i}$ is denoted as $\angle p_{i}$.

\begin{CAL}\label{lem:cauchyArm}
	Let $P = p_1 \dots p_n$ and $P' = p'_1 \dots p'_n$ be two convex planar or spherical polygons such that, for $1\leq i < n$, $|p_i p_{i+1}| = |p_i' p_{i+1}'|$, and for $1\leq i < n-1$, $\angle p_{i+1}  \leq \angle  p'_{i+1} $. Then $|p_n p_1| \leq |p_n' p_1'|$ with equality if and only if $\angle  p_{i+1}  = \angle  p'_{i+1} $ for all $1\leq i < n - 1$. 
\end{CAL}

Cauchy's original proof of the lemma had a gap that subsequently was filled by Ernst Steinitz. A straightforward inductive proof, such as the one in \cite{FuchsTab:2007}, relies on the law of cosines and the triangle inequality. 

Now let $P$ and $P'$ be convex planar or spherical polygons with the same number of sides whose corresponding sides have equal length. Label each vertex of $P$ with a plus sign $+$ or a minus sign $-$ by comparing its angle with the corresponding angle in $P'$: if the angle at $p_{i}$ is larger than that at $p_{i}'$, label it with a $+$, if smaller, a $-$, and if equal, no label at all. Using the Cauchy Arm Lemma, the proof of the following lemma is straightforward. 

\begin{DFVL}\label{lem:cauchyFourVertex}
	Let $P$ and $P'$ be as in the preceding paragraph and label the vertices of $P$ as described. Then either $P$ and $P'$ are congruent, or a walk around $P$ encounters at least four sign changes, from $-$ to $+$ or from $+$ to $-$.
\end{DFVL}

\begin{proof}
	First note that because a polygon is a cycle, the number of sign changes must be even. If no vertex is labeled, then the two polygons are congruent. Assume then that some of the vertices are labeled, but all with the same label. Then Cauchy's Arm Lemma implies that there exists a pair of corresponding edges in $P$ and $P'$ with different lengths, a contradiction. 
	
	Assume now that there are exactly two sign changes of the labels of $P$.  Select two edges $p_i p_{i+1}$ and $p_j p_{j+1}$ (oriented counter-clockwise) of $P$ such that all of the $+$ signs are along the subchain from $p_{i+1}$ to $p_j$ and all of the $-$ signs are along the subchain from $p_{j+1}$ back to $p_i$. Subdivide both edges in two by adding a vertex at the respective midpoints $X$ and $Y$ of $p_i p_{i+1}$ and $p_j p_{j+1}$. Similarly, subdivide the corresponding edges $p'_i p'_{i+1}$ and $p'_j p'_{j+1}$ in $P'$ at midpoints $X'$ and $Y'$. Denote the subchain of $P$ from $X$ to $Y$ by $P_{+}$ and the subchain from $Y$ back to $X$ by $P_{-}$. Similarly for $P'_{+}$ and $P'_{-}$ in $P'$. Applying the arm lemma to $P_+$ and $P_+'$ implies that $|X Y| > |X' Y'|$, and, similarly, an application to $P_{-}$ and $P'_{-}$ implies that $|X Y| < |X' Y'|$, a contradiction. 
\end{proof}

This brings us to the combinatorial component of Cauchy's proof. A  nice proof of the following lemma appears in~\cite{FuchsTab:2007} and follows from an argument based on the Euler characteristic of a sphere.

\begin{CCL}\label{lem:cauchycombinatorial}
	Let $P$ be an abstract spherical polyhedron. Then for any labeling of any non-empty subset of the edges of $P$ with $+$ and $-$ signs, there exists a vertex $v$ that is incident to an edge labeled with a $+$ or a $-$ sign for which one encounters at most two sign changes in labels on the edges adjacent to $v$ as one walks around the vertex.
\end{CCL}

\begin{CRT}
If two bounded, combinatorially equivalent, convex polyhedra in $\mathbb{E}^{3}$ have congruent corresponding faces, then they are congruent by a Euclidean isometry of $\mathbb{E}^{3}$.
\end{CRT}

\begin{proof}
Assume that bounded, convex polyhedra $P$ and $P'$ have the same combinatorics and congruent corresponding faces. For each edge of $P$, label its dihedral angle with a $+$ or a $-$ depending on whether it is larger or smaller than the corresponding dihedral angle in $P'$. If $P$ and $P'$ are not congruent, Cauchy's Combinatorial Lemma provides a vertex $v$ that is incident to an edge labeled with a $+$ or a $-$ sign, and around which there are at most two sign changes. Intersect $P$ with a small sphere centered at $v$ (one that contains no other vertex of $P$ on its interior) to obtain a convex spherical polygon, and intersect $P'$ with a sphere centered at the corresponding vertex $v'$ and of the same radius. By construction both spherical polygons have the same edge lengths, and the angles between edges are given by the dihedral angles between faces at $v$ and $v'$. An application of the Four Vertex Lemma implies that there are at least four sign changes, contradicting that there are at most two. It follows that $P$ and $P'$ are congruent. 
\end{proof}
 
Both the bounded and convex requirements are necessary. For example, a polyhedron $\hat{H}$ in the shape of a cubical house with a shallow pyramidal roof has a cousin $\check{H}$ obtained by inverting the roof. $\hat{H}$ is not congruent to $\check{H}$, though these are combinatorially equivalent with congruent corresponding faces.

%\textbf{The Key Players:} \textit{Bill Thurston, Oded Schramm, Igor Rivin, Hogsden, Ba, Francis Bonahon, Jean-Mark Schlenker}.

%\section{Study III: Discrete and Combinatorial Riemann Mappings} \textbf{The Key Players:} \textit{Oded Schramm, Jim Cannon, Bill Floyd, Walter Parry, Mikhail Gromov, Richard Kenyon, Ken Stephenson}.

\section{In Closing, an Open Invitation.}
This has been a whirlwind tour through the four decade history of the influence of one theorem brought to prominence by the mathematician we celebrate in this volume. Any result that has spawned such a great body of significant work leaves in its wake a bounty of open questions, problems, conjectures, and possible applications that await the right insights for resolution and explanation. What of the Koebe Uniformization Conjecture, of the question of where EQ-type sits among EEL- and VEL-type, of circle packings that mimic rational functions with arbitrary branching, of the existence and rigidity of inversive distance circle packings, of characterizations of projective polyhedra up to M\"obius equivalence generalizing Bao-Bonahon, or of combinatorial rather than metric characterizations of hyperbolic polyhedra of various stripes? I have not covered in this survey the myriad of applications that circle packing has spawned, particularly in the realm of computer graphics and imaging, where each month sees more and more new and original publications. And so I close this tribute to the influence of this one theorem of Bill Thurston with an invitation to any reader who has been captured by the beauty and elegance of the results outlined in this survey to explore further on his or her own the wider discipline of Discrete Conformal Geometry, in both its theoretical and practical bents, and perhaps to add to our understanding and appreciation of this beautiful landscape opened up by the imagination of Bill Thurston.

\bibliographystyle{amsplain}
\bibliography{OurBib}

\end{document}